\documentclass{amsart}

\usepackage{amssymb}
\usepackage{amsmath}
\usepackage{amsfonts}
\usepackage{geometry}
\usepackage{bbm}
\usepackage{hyperref}
\usepackage{tikz}
\usepackage{mathrsfs}
\usepackage{multirow}
\usetikzlibrary{matrix,arrows,decorations.pathmorphing}
\usepackage{verbatim }
\usepackage{mathtools}
\usepackage[author={Sebastian}]{pdfcomment}

\newcommand{\myquad}[1][1]{\hspace*{#1em}\ignorespaces}

\newcounter{cprop}[section]

\newtheorem{theorem}[cprop]{Theorem}

\theoremstyle{plain}

\newtheorem{corollary}[cprop]{Corollary}

\newtheorem{lemma}[cprop]{Lemma}
\newtheorem{proposition}[cprop]{Proposition}

\newtheorem{assumption}[cprop]{Assumption}

\numberwithin{equation}{section}

\theoremstyle{definition}
\newtheorem{definition}[cprop]{Definition}
\newtheorem{example}[cprop]{Example}

\theoremstyle{remark}
\newtheorem{remark}[cprop]{Remark}

\renewcommand{\P}{\mathbb{P}}

\newcommand{\R}{\mathbb{R}}
\newcommand{\N}{\mathbb{N}}
\newcommand{\Z}{\mathbb{Z}}

\renewcommand{\d}{\mathrm{d}}

\newcommand{\vertiii}[1]{{\left\vert\kern-0.25ex\left\vert\kern-0.25ex\left\vert #1 
		\right\vert\kern-0.25ex\right\vert\kern-0.25ex\right\vert}}

\begin{document}
	\title[Integrable bound for rough for rough SPDEs]{An integrable bound for rough stochastic partial differential equations with applications to invariant manifolds and stability}
	\author{M. Ghani Varzaneh}
	\address{Mazyar Ghani Varzaneh\\
		Fakult\"at f\"ur  Mathematik und Informatik, FernUniversit\"at in Hagen, Hagen, Germany}
	\email{mazyar.ghanivarzaneh@fernuni-hagen.de}

	\author{S. Riedel}
	\address{Sebastian Riedel \\
		Fakult\"at f\"ur  Mathematik und Informatik, FernUniversit\"at in Hagen, Hagen, Germany}
	\email{sebastian.riedel@fernuni-hagen.de }
	
	
	
	\begin{abstract}
		We study semilinear rough stochastic partial differential equations as introduced in [Gerasimovi{\v{c}}s, Hairer; EJP 2019]. We provide $\mathcal{L}^p(\Omega)$-integrable a priori bounds for the solution and its linearization in case the equation is driven by a suitable Gaussian process. Using the Multiplicative Ergodic Theorem for Banach spaces, we can deduce the existence of a Lyapunov spectrum for the linearized equation around stationary points. The existence of local stable, unstable, and center manifolds around stationary points is also provided. In the case where all Lyapunov exponents are negative, local exponential stability can be deduced. We illustrate our findings with several examples. 
	\end{abstract}
	
	\maketitle
	
	\maketitle
	\section{Introduction}
	In \cite{GH19}, Gerasimovi{\v{c}}s and Hairer introduced a new solution concept to study parabolic semilinear stochastic partial differential equations (SPDEs) driven by a finite-dimensional noise. One important property of this theory is that it is completely \emph{pathwise} since no stochastic integration theory is used to define the solution to the equation. Instead, the paper uses key ideas of Lyons' rough paths theory \cite{Lyo98}, meaning that the stochastic integral is replaced by a pathwise defined \emph{rough} integral. A nice feature of this concept of a \emph{rough partial differential equation} (RPDE) is that it is fully compatible with classical rough path theory. In particular, one can easily study RPDEs driven by Gaussian noises, e.g. a multidimensional fractional Brownian motion, as introduced in \cite{CQ02,FV10-2,FGGR16}. Let us mention that this framework of RPDEs was later generalized to study non-autonomous SPDEs, too \cite{GHT21}. \smallskip
	
	Avoiding stochastic integration when defining solutions to SPDEs has many advantages, but also leads to new challenges since probabilistic properties of the solution are often not so easy to obtain. This is particularly true when the driving signal is not a Brownian motion, but a general Gaussian process. One of these problems concerns the $\mathcal{L}^p(\Omega)$-integrability of the solution to a random RPDE. The first important result of our work is an \emph{integrable a priori bound} for the solution and the linearization of an RPDE. Although an priori for an RPDE was already given in \cite{HN22}, one can check that it fails to be integrable in case the equation is driven by a Gaussian process. The same problem already occured in the context of (finite-dimensional) rough differential equations (RDEs) driven by Gaussian processes. In fact, this issue constituted a major obstacle when generalizing H\"ormander theory to rough differential equations driven by general Gaussian processes where obtaining the $\mathcal{L}^p(\Omega)$-integrability of the Jacobian of an RDE is a crucial step \cite{CF10, CHLT15}. The subtle problem was that the known a priori bounds for solutions to deterministic RDEs, formulated in terms of the rough path norm, were optimal \cite{FV10}, but not integrable if the noise was replaced by a Gaussian process. This problem was solved in the seminal paper \cite{CLL13} where the known a priori bounds were modified in such a way that probabilistic properties of the Gaussian rough paths could be applied. Later, the results in \cite{CLL13} were slightly simplified and extended to study a larger class of rough differential equations, too \cite{FR12b}. For RPDEs in the sense of \cite{GH19}, the problem of finding an integrable a priori bound remained unsolved up to now. In fact, it is stated in \cite[page 51]{GH19} that \textit{The [integrable] moment bounds for the rough path norms of solution and Jacobian (...) [for the RPDE] might not be easy to obtain in general and require a closer look as a separate problem on its own. We decide to postpone the study of such moments but refer the reader to \cite{FR12b} where this question was answered for the rough SDE case}. In the present paper, we provide exactly these bounds, cf. Theorem \ref{thm:integrability_RPDE} and Theorem \ref{thm:integrabiliy_linearization}, which are our main results in this regard. We believe that these bounds and the techniques to obtain them will be useful when extending non-Markovian H\"ormander theory to RPDEs as initiated in \cite{GH19}. \smallskip
	
	The pathwise solution concept of RPDEs becomes very useful when studying their long-time behaviour with L.~Arnold's concept of \emph{random dynamical systems} (RDS) \cite{Arn98}. In fact, to apply RDS, pathwise solutions are a necessary prerequisite. For stochastic ordinary differential equations (SODEs), pathwise solutions can often be deduced by applying the Kolmogorov-Chentsov continuity theorem. If the solution to the equation takes values in an infinite-dimensional space, which is the case for an SPDE, this theorem can not be applied. A common strategy to circumvent this problem is to transform the SPDE into a random PDE that can be solved pathwise without stochastic integration theory. However, this trick only works under rather restrictive structural assumptions on the equation, e.g. for additive noise or when the diffusion parameter takes a very specific form. For RPDEs in the sense of \cite{GH19}, random dynamical systems were already successfully applied to study center manifolds \cite{KN23}, unstable manifolds \cite{MG22} and random attractors \cite{YLZ23}. \smallskip
	
	In our work, we deduce the existence of local stable, unstable and center manifolds for RPDEs driven by certain Gaussian processes including a fractional Brownian motion with Hurst parameter $H \in (\frac{1}{3},\frac{1}{2}]$, cf. Theorem \ref{stable_manifold}, Theorem \ref{unstable} and Theorem \ref{thm:center_manifolds}. The techniques we use differ from those in \cite{KN23, MG22} in many regards. To wit, the most important tool for us is the Multiplicative Ergodic Theorem (MET) for Banach spaces that we use to deduce the existence of a spectrum of Lyapunov exponents, cf. Theorem \ref{MET}. To apply the MET, it is crucial that the linearized equation satisfies a certain integrability condition which, in fact, can be deduced from the integrable a priori bounds we derived in Section \ref{sec:integrable_bound}. With the MET at hand, the existence of invariant manifolds can be deduced by carefully performed fixed point arguments. Compared to the Lyapunov-Perron method used in \cite{KN23, MG22}, our approach has several advantages. For instance, we can deduce the existence of invariant manifolds around general, even random stationary points (cf. Remark \ref{remark:random_stationary_point} for an example of a random stationary point). With our terminology, the only stationary point that was considered in \cite{KN23} and \cite{MG22} is $0$. This general approach leads to less restrictive assumptions on the equation. For instance, in \cite{KN23} and \cite{MG22}, it is assumed that the drift term $F$ satisfies $F(0) = DF(0) = 0$ and for the diffusion term $G$, it is assumed that $G(0) = DG(0) = D^2G(0) = 0$. We emphasize that we do not have to impose such strict conditions. 
	However, the probably most important point is that our method allows us to deduce the existence of stable manifolds: to our knowledge, Theorem \ref{stable_manifold} is the first stable manifold theorem for RPDEs. The stable manifold theorem describes directions in which the solution of the RPDE decays exponentially fast towards the stationary point. If all Lyapunov exponents are negative, the stable manifold theorem can be used to deduce pathwise local exponential stability of the solution, cf. Section \ref{sec:stability}. We want to emphasize that proving pathwise exponential stability for stochastic differential equations driven by non-Brownian paths is a challenging task even in finite-dimensional spaces due to the lack of the Markov property. Some partial results for rough differential equations driven by a multidimensional fractional Brownian motion were obtained in \cite{GANS18} and \cite{DHC19} for an Hurst parameter $H > \frac{1}{2}$ and in \cite{GAS18} for an Hurst parameter $H \in (\frac{1}{3}, \frac{1}{2})$. A stability result for rough evolution equations driven by a fractional Brownian motion with Hurst parameter $H > \frac{1}{2}$ was obtained in \cite{DGANS18}. The stability problem for RPDEs driven by a fractional Brownian motion with Hurst parameter $H \in (\frac{1}{3},\frac{1}{2})$ was investigated first in \cite{Hes22}. The author can prove local exponential stability around zero provided $F(0) = DF(0) = 0$ and $G(0) = DG(0) = 0$. Compared to our stability results, cf. Theorem \ref{thm:exp_stability}, these assumptions are more restrictive since we do not have to assume that the first derivates have a fixed point at zero. Furthermore, the method we are using allows to easily generalize the local stability results to RPDEs around random stationary points.

	\subsection*{Notation and basic definitions}
	
	The symbol $\circ$ usually denotes an inner product. In estimates, $a \lesssim b$ means that there is a constant $C$ that might depend on some irrelevant parameters such that $a \leq C b$. In this article, we will often consider indexed families of Banach spaces $\{(\mathcal{B}_{\alpha}, |\cdot|_{\mathcal{B}_{\alpha}})\}_{\alpha}$. Mostly, the norm $|\cdot|_{\mathcal{B}_{\alpha}}$ will simply be denoted by $|\cdot|_{\alpha}$. For a Banach space $(\mathcal{B}, |\cdot|)$, $x_0 \in  \mathcal{B}$ and $\epsilon > 0$, we set
	\begin{align*}
		B_{\mathcal{B}}(x_0,\epsilon) &\coloneqq B(x_0,\epsilon) \coloneqq \lbrace x\in \mathcal{B} \,:\, | x-x_0 | < \epsilon\rbrace \quad \text{and} \\
		\overline{B_{\mathcal{B}} (x_0,\epsilon)} &\coloneqq \overline{B(x_0,\epsilon)} \coloneqq \lbrace x\in \mathcal{B} \,:\, | x-x_0 | \leq \epsilon\rbrace.
	\end{align*}
	If $\mathcal{B}$ and $\tilde{\mathcal{B}}$ are Banach spaces, the space $\mathcal{L}(\mathcal{B}, \tilde{\mathcal{B}})$ consists of all bounded linear functions from $\mathcal{B}$ to $\tilde{\mathcal{B}}$ and is equipped with the usual operator norm. We write $\mathcal{L}(\mathcal{B}) \coloneqq \mathcal{L}(\mathcal{B},\mathcal{B})$. Let $I$ be an interval. For a function $Y \colon I \to \mathcal{B}$ and $s,t \in I$, we set $\delta X_{s,t} \coloneqq X_t - X_s$. The space $C(I; \mathcal{B})$ consists of all continuous functions $X \colon I \to \mathcal{B}$. Similarly, $C_2(I; \mathcal{B})$ denotes the space of continuous functions $Z \colon I \times I \to \mathcal{B}$. Both spaces are equipped with the sup-norm. For $\gamma > 0$, $X \in C^{\gamma}(I; \mathcal{B})$ if and only if
	\begin{align*}
		\|Y\|_{C^{\gamma}} \coloneqq \sup_{t \in I} |Y_t| +   \|Y\|_{\gamma} < \infty
	\end{align*}
	where
	\begin{align*}
		\|Y\|_{\gamma} \coloneqq \|Y\|_{\gamma,I} \coloneqq \sup_{\substack{s,t \in I \\ s \neq t}}  \frac{|\delta Y_{s,t}|}{|t-s|^{\gamma}}.
	\end{align*}
	Similarly, $Z \in C^{\gamma}_2(I; \mathcal{B})$ if and only if
	\begin{align*}
		\|Z\|_{C_2^{\gamma}} \coloneqq \sup_{s,t \in I} |Z_{s,t}| + \|Z\|_{\gamma} < \infty
	\end{align*}
	where
	\begin{align*}
		\|Z\|_{\gamma} \coloneqq \|Z\|_{\gamma,I} \coloneqq \sup_{\substack{s,t \in I \\ s \neq t}}  \frac{|Z_{s,t}|}{|t-s|^{\gamma}}.
	\end{align*}
	By the derivative of a Banach space valued function we mean the derivative in Fr\'echet-sense.

	
	
	

	\subsection*{Review on Rough stochastic partial differential equations}
	
	We assume that the reader is familiar with the basic notions of rough path theory as it is presented, for instance, in \cite{FH20}. We will mostly consider $\gamma$-H\"older rough paths $\mathbf{X}$ for $\gamma \in (\frac{1}{3},\frac{1}{2}]$, i.e. $\mathbf{X}$ has two components, $\mathbf{X} = (X,\mathbb{X})$. The space space of all $\gamma$-H\"older rough paths defined on $[0,T]$ is denoted by $\mathscr{C}^{\gamma}([0,T],\R^n)$. We write $\mathbf{X} \in \mathscr{C}^{\gamma}([0,\infty),\R^n)$ if and only if $\mathbf{X}|_{[0,T]} \in \mathscr{C}^{\gamma}([0,T],\R^n)$ for every $T > 0$. For $[s,t]\subseteq [0,T]$, we set 
	\begin{align*}
		\varrho_{\gamma}(\mathbf{X},[s,t]):=1+\Vert X\Vert_{\gamma,[s,t]}+\Vert\mathbb{X}\Vert_{2\gamma,[s,t]}.
	\end{align*} \smallskip
	
	We are interested in the solution of a rough SPDE of the form
	\begin{align}\label{SPDE_EQU}
		\mathrm{d}Z_t = AZ_t \, \mathrm{d}t+F(Z_t) \, \mathrm{d}t+G(Z_t)\circ\mathrm{d}\mathbf{X}_t, \ \ \ \  Z_{0}=z_0\in \mathcal{B}_{\alpha}
	\end{align}
	where $\mathbf X$ is a rough path. This family of SPDEs is studied in \cite{GH19} and \cite{GHT21}. We quickly review some basic definitions and notations. For more details, the reader is referred to \cite{GHT21}. The following definition is taken from  \cite[Definition 2.1]{GHT21}.
	\begin{definition}
		We call a family of indexed separable Banach spaces $\lbrace (\mathcal{B}_{\beta},\vert\ \cdot \vert_{\beta})\rbrace_{\beta\in\mathbb{R}}$  a \emph{monotone
			family of interpolation spaces} if
		\begin{itemize}
			\item [(i)] For every $\alpha\leq \beta$:  $\mathcal{B}_{\beta}$ is a dense subset of $\mathcal{B}_{\alpha}$ and the  identity map $\operatorname{id} \colon \mathcal{B}_{\beta}\rightarrow \mathcal{B}_{\alpha}$  is continuous.
			\item [(ii)] For every $\alpha\leq\beta\leq\theta$ and $x\in\mathcal{B}_{\alpha}\cap\mathcal{B}_{\theta}: \  \vert x\vert_{\beta}\lesssim \vert x\vert_{\alpha}^{\frac{\theta-\beta}{\theta-\alpha}}\vert x\vert_{\gamma}^{\frac{\beta-\alpha}{\gamma-\alpha}}$.
		\end{itemize}
	\end{definition}
	We will assume the following:
	\begin{assumption}\label{existence}
		Let $\frac{1}{3}<\gamma\leq\frac{1}{2}$ and $\mathbf{X}=(X,\mathbb{X})\in\mathscr{C}^{\gamma}([0,\infty),\mathbb{R}^n)$ be a $\gamma$-H\"older rough path. Furthermore, let $0\leq\sigma <1$, $0\leq \eta <\gamma $ and $0\leq\theta\leq 2\gamma$. Assume that
		\begin{itemize}
			\item $F \colon \mathcal{B}_{\alpha}\rightarrow\mathcal{B}_{\alpha-\sigma}$ is a locally Lipschitz continuous with linear growth, i.e. there are some constants $p_1,p_2$ such that $\vert F(x)\vert_{\alpha-\sigma}\leq p_{1}+p_{2}\vert x\vert_{\alpha}$ for all $x \in \mathcal{B}_{\alpha}$.
			\item $G \colon \mathcal{B}_{\alpha-\theta}\rightarrow\mathcal{B}_{\alpha-\theta-\eta}^{n}$ is a bounded Fr\'echet differentiable function up to three times with bounded derivatives or a bounded  linear function.
			\item $A$ generates a continuous semigroup $(S_{t})_{t\geq 0}$ such that for every $\beta \in [\min\lbrace \alpha-2\gamma-\eta, \alpha-\sigma\rbrace,\alpha]$, $S_{t}\in\mathcal{L}(\mathcal{B}_{\beta})$. Also for every  ${\sigma}_1\in [0,1)$ with $\beta+{\sigma}_1\leq \alpha$,
			\begin{align}\label{SEM_I}
				\begin{split}
					\vert S_{t}x\vert_{\beta+{\sigma}_1} &\lesssim t^{-{\sigma}_1}\vert x\vert_{{\beta}}, \\  
					\vert(I-S_{t})x\vert_{{\beta}} &\lesssim t^{{\sigma}_1}\vert x\vert_{\beta+{\sigma}_1}.
				\end{split}
			\end{align}	
		\end{itemize}  
	\end{assumption}
	\begin{remark}
		Note that as a direct result of \eqref{SEM_I} for $\beta \in [\min\lbrace \alpha-2\gamma-\eta, \alpha-\sigma\rbrace,\alpha]$ and  $\sigma_1,\sigma_2\in [0,1)$ such that  $\beta-\sigma_1+\sigma_2\in  [\min\lbrace \alpha-2\gamma-\eta, \alpha-\sigma\rbrace,\alpha]$, we have
		\begin{align}\label{SEMI_II}
			\vert S_{t-u}(I-S_{u-v})x\vert_{\beta}\lesssim (t-u)^{-\sigma_1}(u-v)^{\sigma_2}\vert x\vert_{\beta-\sigma_1+\sigma_2}.\ \  
		\end{align}
	\end{remark}
	Let us quickly review the required framework to solve \eqref{SPDE_EQU} and also some preliminary definitions. Most of this is taken from \cite{GHT21}. 
	\begin{definition}
		For an interval $I\subset\mathbb{R}$, set
		\begin{align*}
			\mathcal{E}^{0,\gamma}_{\alpha-\gamma,I} \coloneqq C(I;\mathcal{B}_{\alpha-\gamma})\cap C^{\gamma}(I;\mathcal{B}_{\alpha-2\gamma}) \quad \text{and} \quad \mathcal{E}^{\gamma,2\gamma}_{\alpha;I} \coloneqq C^{\gamma}_{2}(I;\mathcal{B}_{\alpha-\gamma})\cap C^{2\gamma}_{2}(I;\mathcal{B}_{\alpha-2\gamma}).
		\end{align*}
		We write $Z\in \mathcal{D}_{\mathbf{X},\alpha}^{\gamma}(I)$ if there exists a $Z' \in (\mathcal{E}^{0,\gamma}_{\alpha-\gamma,I})^n$ such that for $Z^{\#}_{s,t} \coloneqq Z_{s,t} - Z^{\prime}_{s}\circ(\delta X)_{s,t}$, $s,t\in I$, we have
		\begin{align*}
			\Vert (Z,Z^{\prime})\Vert_{\mathcal{D}_{\mathbf{X},\alpha}^{\gamma}(I)} \coloneqq \Vert Z\Vert_{C(I;\mathcal{B}_{\alpha})}+\Vert Z^{\prime}\Vert_{(\mathcal{E}^{0,\gamma}_{\alpha-\gamma,I})^n}+\Vert Z^{\#}\Vert_{\mathcal{E}^{\gamma,2\gamma}_{\alpha;I}}<\infty
		\end{align*}
		where 
		\begin{align*}
			\Vert Z\Vert_{C(I;\mathcal{B}_{\alpha})} &\coloneqq \sup_{\tau\in I}\vert Z_{\tau}\vert_{\alpha},\ \ \Vert Z^{\prime}\Vert_{(\mathcal{E}^{0,\gamma}_{\alpha-\gamma,I})^n}:= \max \left\{ \sup_{\tau\in I}\vert Z^{\prime}_{\tau}\vert_{\alpha-\gamma}^{(n)},\sup_{\substack{\tau,\nu\in I,\\ \tau<\nu}}\frac{\vert (\delta Z^{\prime})_{\tau,\nu}\vert_{\alpha-2\gamma}^{(n)}}{(\nu-\tau)^{\gamma}} \right\} \ \text{and}\\ \vert Z^{\#}\vert_{\mathcal{E}^{\gamma,2\gamma}_{\alpha;I}} &\coloneqq \max \left\{ \sup_{\substack{\tau,\nu\in I,\\ \tau<\nu}}\frac{\vert Z^{\#}_{\tau,\nu}\vert_{\alpha-\gamma}}{(\nu-\tau)^{\gamma}},\sup_{\substack{\tau,\nu\in I,\\ \tau<\nu}}\frac{\vert Z^{\#}_{\tau,\nu}\vert_{\alpha-2\gamma}}{(\nu-\tau)^{2\gamma}} \right\}.
		\end{align*}
		Above, for $Z'=(Z'_{i})_{1\leq i\leq n}$, we use the notation $|Z'|_{\beta}^{(n)} \coloneqq \sup_{1\leq i\leq n}|Z'_i|_{\beta}$.
	\end{definition} 
	It follows directly from the definition that
	\begin{align}\label{NMMMMMLY}
		\sup_{\substack{\tau,\nu\in I,\\ \tau<\nu}}\frac{\vert (\delta Z)_{\tau,\nu}\vert_{\alpha-\gamma}}{(\nu-\tau)^{\gamma}}
		\leq (1+\Vert X\Vert_{\gamma,I})\Vert (Z,Z^{\prime})\Vert_{\mathcal{D}_{\mathbf{X},\alpha}^{\gamma}(I)}.
	\end{align}
	We recall that for $(Z,Z^\prime)=(Z^i,(Z^{i})^{\prime})_{1\leq i\leq n}\in(\mathcal{D}_{\mathbf{X},\alpha}^{\gamma}(I))^{n}$, the following limit exists: 
	\begin{align}\label{SEW}
		\int_{s}^{t}S_{t-\tau}Z_{\tau}\circ\mathrm{d}\mathbf{X}_{\tau}:=\lim_{\substack{|\pi|\rightarrow 0,\\ \pi=\lbrace s = \tau_0 < \tau_{1} < \ldots < \tau_{m}=t \rbrace}}\sum_{0\leq j<m}\big{[}S_{t-\tau_j}Z_{\tau_j}\circ (\delta X)_{\tau_j,\tau_{j+1}}+S_{t-\tau_j}Z^{\prime}_{\tau_j}\circ\mathbb{X}_{\tau_j,\tau_{j+1}}\big{]}
	\end{align}
	where for $(Z^{i}_{s})^{\prime}=((Z^{i,j}_{s})^{\prime})_{1\leq j\leq n}$, 
	\begin{align*}
		S_{t-s}Z^{\prime}_{s}\circ\mathbb{X}_{s,t} \coloneqq \sum_{1\leq i,j\leq n}S_{t-s}(Z^{i,j}_{s})^{\prime}\circ\mathbb{X}^{i,j}_{s,t}.
	\end{align*}
	Here, $\pi=\lbrace s = \tau_0 < \tau_{1} < \ldots < \tau_{m}=t \rbrace$ denotes a finite partition of $[s,t]$ and
	\begin{align*}
		|\pi| = \max_{i = 0,\ldots, m-1}|\tau_{i+1} - \tau_i|.
	\end{align*}
	
	Let $ Z\in\mathcal{D}_{\mathbf{X},\alpha}^{\gamma}(I)$. Then, it can easily be shown that $G(Z) \in\mathcal{D}_{\mathbf{X},\alpha-\eta}^{\gamma}(I)$. Also from \cite[Theorem 4.5.]{GHT21}, for $0\leq\eta<\gamma$ and $I=[s,t]$, the linear map
	\begin{align*}
		&(\mathcal{D}_{\mathbf{X},{\alpha}-\eta}^{\gamma}(I))^{n}\longrightarrow \mathcal{D}_{\mathbf{X},{\alpha}}^{\gamma}(I),
		\\ & \quad (Z,Z^{\prime})\longrightarrow \left( \int_{s}^{.}S_{.-\tau}Z_{\tau}\circ\mathrm{d}\mathbf{X}_{\tau},Z \right)
	\end{align*}
	is well defined. In addition, if $t-s\leq 1$ then for $i\in\lbrace 0,1,2\rbrace$,
	\begin{align}\label{BBBBBB1}
		\begin{split}
			&\left|\int_{s}^{t}S_{t-\tau}Z_{\tau}\circ\mathrm{d}\mathbf{X}_{\tau}-S_{t-s}Z_{s}\circ (\delta X)_{s,u}-S_{t-s}Z_{s}^{\prime}\circ \mathbb{X}_{s,u}\right|_{\alpha-i\gamma}\\&\quad\leq  C_{\gamma,\eta}(t-s)^{\gamma i+\gamma-\eta}(1+\Vert X\Vert_{\gamma,I}+\Vert\mathbb{X}\Vert_{2\gamma,I})\Vert (Z,Z^{\prime})\Vert_{(\mathcal{D}_{\mathbf{X},{\alpha-\eta}}^{\gamma}(I))^{n}}.\\
		\end{split}
	\end{align}
	Also,
	\begin{align}\label{BBBBBB}
		\begin{split}
			&	\left\| \big( \int_{s}^{.}S_{.-\tau}Z_{\tau}\circ\mathrm{d}\mathbf{X}_{\tau},Z\big) \right\| _{\mathcal{D}_{\mathbf{X},{\alpha}}^{\gamma}(I)}\\&\quad \leq C_{\gamma,\eta}\bigg( \vert Z_{s}\vert_{\alpha-\eta}^{(n)}+\vert Z^{\prime}_s\vert_{\alpha-\eta-\gamma}^{(n\times n)} 	\varrho_{\gamma}(\mathbf{X},[s,t])+ (t-s)^{\gamma-\eta}	\varrho_{\gamma}(\mathbf{X},[s,t])\Vert (Z,Z^{\prime})\Vert_{(\mathcal{D}_{\mathbf{X},{\alpha-\eta}}^{\gamma}(I))^{n}}\bigg).
		\end{split}
	\end{align}
	\begin{remark}
		We will prove\eqref{BBBBBB}, in a more general case in Lemma \eqref{parmeter}. In some references, this inequality is stated in the following form
		\begin{align*}
			&	\left\| \big( \int_{s}^{.}S_{.-\tau}Z_{\tau}\circ\mathrm{d}\mathbf{X}_{\tau},Z\big) \right\| _{\mathcal{D}_{\mathbf{X},{\alpha}}^{\gamma}(I)}\\&\quad \leq \bigg( \vert Z_{s}\vert_{\alpha-\eta}^{(n)}+\vert Z^{\prime}_s\vert_{\alpha-\eta-\gamma}^{(n\times n)} + C_{\gamma,\eta}(t-s)^{\gamma-\eta}	\varrho_{\gamma}(\mathbf{X},[s,t])\Vert (Z,Z^{\prime})\Vert_{(\mathcal{D}_{\mathbf{X},{\alpha-\eta}}^{\gamma}(I))^{n}}\bigg).
		\end{align*}
		This, as we will see in Lemma \eqref{parmeter}, is not true and needs a minor correction. 
	\end{remark}
	Finally, we can define the mild solution to \eqref{SPDE_EQU}:
	\begin{definition}
		We say that $Z\in \mathcal{D}_{\mathbf{X},{\alpha}}^{\gamma}(I)$ solves equation \eqref{SPDE_EQU} if and only if $Z$ satisfies the identity 
		\begin{align}\label{MILD}
			Z_t=S_{t-s}Z_{s}+\int_{s}^{t}S_{t-\tau}F(Z_{\tau}) \, \mathrm{d}\tau+\int_{s}^{t}S_{t-\tau}G(Z_{\tau})\circ\mathrm{d}\mathbf{X}_{\tau}, \ \ Z_{0}=z_0, \ s,t\in \mathbb{R},
		\end{align}
		where the second integral is understood as \eqref{SEW}.
	\end{definition}
	
	\begin{remark}
		Existence and uniqueness of the solution for this type of equation are discussed in several articles. For example, in \cite{HN22}, the authors prove that the mild solutions of the equation under Assumption \ref{existence} exists, is unique and globally defined. 
	\end{remark}

	\section{An integrable a priori bound}\label{sec:integrable_bound}
	We aim to prove that the solution to \eqref{SPDE_EQU} has an integrable bound. As we stated earlier, the well-posedness and global existence of the solution for this family of equations is well understood. However, the a priori bounds that are provided are not optimal in the sense that they fail to be integrable for Gaussian noises\footnote{For instance, in \cite{HN22}, a careful inspection of the proofs reveals a bound of the form $\exp((\varrho_{\gamma}(\mathbf{X},[s,t]))^{\frac{1}{\gamma-\eta}})$ that is clearly not integrable.}. The main obstacle we face here is the presence of the semigroup in the rough integral: if we just apply Gr\"onwall's lemma naively, we will only get an exponential bound in terms of the noise which is not integrable. To overcome this problem, we employ a modified version of the greedy points technique introduced in \cite{CLL13} and modify the Sewing lemma. \smallskip
	
	In the following section, $\mathbf{X} = (X,\mathbb{X})$ always denotes a $\gamma$-H\"older rough path where $\gamma \in (\frac{1}{3},\frac{1}{2}]$. \smallskip

	Let us first start with the following lemma where we introduce a new control function that will play a crucial role.
	\begin{lemma}\label{cnt}
		For $0\leq \eta_1<\gamma$, set
		\begin{align}
			\begin{split}
				&	W_{\mathbf{X},\gamma,\eta_1} \colon \Delta_{T} = \{(s,t) \in [0,T]^2 \,:\, s \leq t \} \rightarrow \mathbb{R}, \\
				&	W_{\mathbf{X},\gamma,\eta_1}(s,t) \coloneqq \sup_{\substack{\pi,\\ \pi=\lbrace s = \kappa_0 < \kappa_{1} < \ldots < \kappa_{m} = t  \rbrace}} \left\{ \sum_{j}(\kappa_{j+1}-\kappa_j)^{\frac{-\eta_1}{\gamma-\eta_1}}\big{[}\Vert(\delta X)_{\kappa_{j},\kappa_{j+1}}\Vert^{\frac{1}{\gamma-\eta_1}}+\Vert\mathbb{X}_{\kappa_{j},\kappa_{j+1}}\Vert^{\frac{1}{2(\gamma-\eta_1)}} \big{]} \right\}.
			\end{split}
		\end{align} 
		where the supremum ranges over all finite partitions of the interval $[s,t]$. Then  $W_{\mathbf{X},\gamma,\eta_1}$ is a \emph{control function}, i.e. it is continuous and satisfies 
		\begin{align}\label{DFDF}
			W_{\mathbf{X},\gamma,\eta_1}(s,u)+W_{\mathbf{X},\gamma,\eta_1}(u,t)\leq W_{\mathbf{X},\gamma,\eta_1}(s,t),\ \ \  s\leq u\leq t.
		\end{align}
	\end{lemma}
	\begin{proof}
		Follows from our assumption on $\mathbf{X}$.
	\end{proof}

	The next lemma is basic.
	\begin{lemma}
		Assume $Z\in \mathcal{D}_{\mathbf{X},{\alpha}}^{\gamma}(I)$ is a mild solution to \eqref{SPDE_EQU}. Then  
		\begin{align*}
			&Z^{\prime}_{s}=G(Z_s) \quad \text{and} \quad Z^{\#}_{s,t}=(\delta Z)_{s,t}-G(Z_s)\circ (\delta X)_{s,t}.
		\end{align*}
		Moreover, $[G(Z)]^{\prime}_{s}=D_{Z_{s}}G [G(Z_s)]$ and 
		\begin{align}\label{Remaider}
			\begin{split}
				[G(Z)]^{\#}_{s,t} &= G(Z_t)-G(Z_s)-D_{Z_{s}}G [G(Z_s)\circ (\delta X)_{s,t}] \\
				&= \int_{0}^{1}\int_{0}^{1}\sigma D^{2}_{Z_s+\sigma u(\delta Z)_{s,t}}G [G(Z_s)\circ (\delta X)_{s,t},G(Z_s)\circ (\delta X)_{s,t}+Z_{s,t}^{\#}] \, \mathrm{d}u \, \mathrm{d}\sigma  \\
				&\quad +\int_{0}^{1}D_{Z_s+\sigma (\delta Z)_{s,t}}G [Z^{\#}_{s,t}] \, \mathrm{d}\sigma .
			\end{split}
		\end{align}
		Let $s\leq u\leq v\leq w\leq t$ and 
		\begin{align*}
			\tilde{\Xi}^{u,v}_{s,t} \coloneqq S_{t-u}G(Z_{u})\circ (\delta X)_{u,v}+S_{t-u}D_{Z_{u}}G [G(Z_u)]\circ\mathbb{X}_{u,v}.
		\end{align*}
		Then
		\begin{align}\label{Increment}
			\begin{split}
				&\tilde{\Xi}^{u,v}_{s,t}+\tilde{\Xi}^{v,w}_{s,t}-\tilde{\Xi}^{u,w}_{s,t} \\
				&\ =S_{t-u}([G(Z)]^{\#}_{u,v})\circ (\delta X)_{v,w}-S_{t-v}(S_{v-u}-I)G(Z_v)\circ (\delta X)_{v,w} - S_{t-v}(S_{v-u}-I)D_{Z_{v}}G [G(Z_v)]\circ\mathbb{X}_{v,w} \\
				&\quad + S_{t-u}\big(\int_{0}^{1}D^{2}_{Z_{u}+\sigma (\delta z)_{u,v}}G\big[G(Z_u)\circ (\delta X)_{u,v}+Z^{\#}_{u,v},G(Z_u)\big] \, \mathrm{d}\sigma\big)\circ \mathbb{X}_{v,w}\\
				&\quad + S_{t-u}\big(D_{Z_v}G\big[\int_{0}^{1}D_{Z_{u}+\sigma(\delta Z)_{u,v}}G[G(Z_u)\circ (\delta X)_{u,v}+Z^{\#}_{u,v}] \, \mathrm{d}\sigma\big]\big)\circ \mathbb{X}_{v,w}.
			\end{split}
		\end{align}
	\end{lemma}
	\begin{proof}
		Follows from  definition of the mild solutions and our assumptions.
	\end{proof}
	
	\begin{lemma}\label{SSA}
		Let us fix $s<t$ and set $\tau^{n}_{m} \coloneqq s + \frac{n}{2^m}(t-s)$ where $0\leq n<2^m-1$.  We also define
		\begin{align}
			\Xi^{n,m}_{s,t} \coloneqq S_{t-\tau^{n}_{m}}G(Z_{\tau^{n}_{m}})\circ (\delta X)_{\tau^{n}_{m},\tau^{n+1}_{m}}+S_{t-\tau^{n}_{m}}D_{Z_{\tau^{n}_{m}}}G[G(Z_{\tau^{n}_{m}})]\circ \mathbb{X}_{\tau^{n}_{m},\tau^{n+1}_{m}}.
		\end{align}
		Then for $A^{n}_{m}:=G(Z_{\tau^{2n}_{m+1}})\circ (\delta X)_{{\tau^{2n}_{m+1}},{\tau^{2n+1}_{m+1}}}$ we have
		\begin{align}\label{increment}
			\begin{split}
				&\Xi^{2n,m+1}_{s,t}+\Xi^{2n+1,m+1}_{s,t}-\Xi^{n,m}_{s,t} \\
				&\ = S_{t-\tau^{2n}_{m+1}} \left( \int_{0}^{1}\int_{0}^{1}\sigma D^{2}_{Z_{\tau^{2n}_{m+1}}+\sigma u (\delta Z)_{\tau^{2n}_{m+1},\tau^{2n+1}_{m+1}}}G[A^{n}_{m};A^{n}_{m}+Z^{\#}_{{\tau^{2n}_{m+1}},{\tau^{2n+1}_{m+1}}} ] \, \mathrm{d}u\, \mathrm{d}\sigma \right) \circ (\delta X)_{\tau^{2n+1}_{m+1},\tau^{2n+2}_{m+1}}\\
				&\quad + S_{t-\tau^{2n}_{m+1}} \left( \int_{0}^{1}D_{Z_{\tau^{2n}_{m+1}}+\sigma (\delta Z)_{\tau^{2n}_{m+1},\tau^{2n+1}_{m+1}}}G[Z^{\#}_{{\tau^{2n}_{m+1}},\tau^{2n+1}_{m+1}}] \, \mathrm{d}\sigma\right) \circ (\delta X)_{\tau^{2n+1}_{m+1},\tau^{2n+2}_{m+1}}\\
				&\quad + S_{t-\tau^{2n}_{m+1}} \left( \int_{0}^{1}D^{2}_{Z_{\tau^{2n}_{m+1}} + \sigma (\delta Z)_{\tau^{2n}_{m+1},\tau^{2n+1}_{m+1}}}G[A^{n}_{m}+Z^{\#}_{\tau^{2n}_{m+1},\tau^{2n+1}_{m+1}};G(Z_{\tau^{2n}_{m+1}})]\,\mathrm{d}\sigma \right) \circ\mathbb{X}_{\tau^{2n+1}_{m+1},\tau^{2n+2}_{m+1}}\\
				&\quad + S_{t-\tau^{2n}_{m+1}} \left( D_{Z_{\tau^{2n+1}_{m+1}}}G \left[ \int_{0}^{1}D_{Z_{\tau^{2n}_{m+1}}+\sigma (\delta Z)_{\tau^{2n}_{m+1},\tau^{2n+1}_{m+1}}}G [A^{n}_{m}+Z^{\#}_{\tau^{2n}_{m+1},\tau^{2n+1}_{m+1}}]\, \mathrm{d}\sigma \right] \right) \circ \mathbb{X}_{\tau^{2n+1}_{m+1},\tau^{2n+2}_{m+1}}\\
				&\quad - S_{t-\tau^{2n+1}_{m+1}}\big(S_{\tau^{2n+1}_{m+1}-\tau^{2n}_{m+1}}-I\big) \left( G(Z_{\tau^{2n+1}_{m+1}})\circ (\delta X)_{\tau^{2n+1}_{m+1},\tau^{2n+2}_{m+1}}+D_{Z_{\tau^{2n+1}_{m+1}}}G[G(Z_{\tau^{2n+1}_{m+1}})]\circ \mathbb{X}_{\tau^{2n+1}_{m+1},\tau^{2n+2}_{m+1}} \right).
			\end{split}
		\end{align}
	\end{lemma}
	\begin{proof}
		Follows from \eqref{Remaider} and \eqref{Increment}.
	\end{proof}
	In the next proposition, we obtain an upper bound over the latter formula in terms of our control function defined in Lemma \ref{cnt}. 
	\begin{proposition}\label{increment_estimate}
		For $i\in \lbrace 0,1,2\rbrace$ and $\epsilon>0$ chosen such that $\eta+\epsilon<\gamma$, there exists a constant $M_{\epsilon}$ depending on $\epsilon$ and $G$ such that
		\begin{align}\label{cn-es}
			\begin{split}
				&\sum_{m\geq 0}\sum_{0\leq n<2^{m}} \left| S_{t-\tau^{2n}_{m+1}} \left(\int_{0}^{1}\int_{0}^{1}\sigma D^{2}_{Z_{\tau^{2n}_{m+1}} + \sigma u (\delta Z)_{\tau^{2n}_{m+1},\tau^{2n+1}_{m+1}}} G[A^{n}_{m};Z^{\#}_{{\tau^{2n}_{m+1}},{\tau^{2n+1}_{m+1}}} ]\, \mathrm{d}u\, \mathrm{d}\sigma \right) \circ (\delta X)_{\tau^{2n+1}_{m+1},\tau^{2n+2}_{m+1}} \right|_{\alpha-i\gamma}\\
				&\quad + \left|  S_{t-\tau^{2n}_{m+1}} \left( \int_{0}^{1}D_{Z_{\tau^{2n}_{m+1}} + \sigma (\delta Z)_{\tau^{2n}_{m+1},\tau^{2n+1}_{m+1}}}G[Z^{\#}_{{\tau^{2n}_{m+1}},\tau^{2n+1}_{m+1}}]\, \mathrm{d}\sigma \right) \circ (\delta X)_{\tau^{2n+1}_{m+1},\tau^{2n+2}_{m+1}} \right|_{\alpha-i\gamma}\\
				&\quad + \left| S_{t-\tau^{2n}_{m+1}} \left( \int_{0}^{1}D^{2}_{Z_{\tau^{2n}_{m+1}}+\sigma (\delta Z)_{\tau^{2n}_{m+1},\tau^{2n+1}_{m+1}}}G[Z^{\#}_{\tau^{2n}_{m+1},\tau^{2n+1}_{m+1}};G(Z_{\tau^{2n}_{m+1}})] \right) \circ\mathbb{X}_{\tau^{2n+1}_{m+1},\tau^{2n+2}_{m+1}} \right|_{\alpha-i\gamma}\\
				&\quad + \left| S_{t-\tau^{2n}_{m+1}} \left( D_{Z_{\tau^{2n+1}_{m+1}}}G[\int_{0}^{1}D_{Z_{\tau^{2n}_{m+1}}+\sigma (\delta Z)_{\tau^{2n}_{m+1},\tau^{2n+1}_{m+1}}}G [Z^{\#}_{\tau^{2n}_{m+1},\tau^{2n+1}_{m+1}}]\, \mathrm{d}\sigma] \right) \circ \mathbb{X}_{\tau^{2n+1}_{m+1},\tau^{2n+2}_{m+1}}\right|_{\alpha-i\gamma} \\
				&\leq M_{\epsilon}(t-s)^{i\gamma}\max\big\lbrace(t-s)^{\epsilon}\big(W_{\mathbf{X},\gamma,\eta+\epsilon}(s,t)\big)^{\gamma-\eta-\epsilon},(t-s)^{2\epsilon}\big(W_{\mathbf{X},\gamma,\eta+\epsilon}(s,t)\big)^{2(\gamma-\eta-\epsilon)}\big\rbrace\vert Z^{\#}\vert_{\mathcal{E}^{\gamma,2\gamma}_{{\alpha};I}}.
			\end{split}
		\end{align}
	\end{proposition}
	\begin{proof}
		We will concentrate on the case when $G$ is bounded. If $G$ is bounded linear, some terms in \eqref{cn-es} are zero due to $D^{2}G=0$, therefore the computations become even more straightforward. We have to show that each term on the left hand side of \eqref{cn-es} can be bounded up to a constant depending on $\epsilon$ and $G$ by $(t-s)^{i\gamma+\epsilon}W_{\mathbf{X},\gamma,\eta+\epsilon}(s,t)^{\gamma-\eta-\epsilon}$. We will show this bound for the second and fourth term whose proofs have some distinctions. For the remaining terms, our claim can be confirmed by a similar technique. Remember that $\Vert Z^{\#}\Vert_{\mathcal{E}^{\gamma,2\gamma}_{{\alpha};I}}<\infty$ and also that \eqref{SEM_I} holds true. Then
		\begin{align*}
			&\sum_{m\geq 0}\sum_{0\leq n<2^{m}} \left| S_{t-\tau^{2n}_{m+1}} \left( \int_{0}^{1}D_{\sigma Z_{\tau^{2n}_{m+1},\tau^{2n+1}_{m+1}} + Z_{\tau^{2n}_{m+1}}}G[Z^{\#}_{{\tau^{2n}_{m+1}},Z{\tau^{2n+1}_{m+1}}}]\mathrm{d}\sigma \right) \circ (\delta X)_{\tau^{2n+1}_{m+1},\tau^{2n+2}_{m+1}} \right|_{\alpha-i\gamma}\\
			\lesssim \ &\vert Z^{\#}\vert_{\mathcal{E}^{\gamma,2\gamma}_{{\alpha};I}}\sum_{m\geq 1}\sum_{0\leq n<2^{m}} ( t-\tau^{2n}_{m+1})^{\gamma(i-2)-\eta}(t-s)^{2\gamma} \left( \frac{1}{2^{m+1}} \right)^{2\gamma}\vert {X}_{\tau^{2n+1}_{m+1},\tau^{2n+2}_{m+1}}\vert\\
			\leq\ &(t-s)^{\gamma i+\epsilon}\vert Z^{\#}\vert_{\mathcal{E}^{\gamma,2\gamma}_{{\alpha};I}}\sum_{m\geq 1}\sum_{0\leq n<2^{m}} \left(1-\frac{2n}{2^{m+1}} \right)^{\gamma(i-2)-\eta} \left( \frac{1}{2^{m+1}} \right)^{2\gamma+\eta+\epsilon} W_{\mathbf{X},\gamma,\eta+\epsilon}(\tau^{2n+1}_{m+1},\tau^{2n+2}_{m+1})^{\gamma-\eta-\epsilon}.
		\end{align*}
		From the H\"older inequality and \eqref{DFDF},
		\begin{align*}
			&\sum_{0\leq n<2^{m}} \left( 1-\frac{2n}{2^{m+1}} \right)^{\gamma(i-2)-\eta} \left( \frac{1}{2^{m+1}} \right)^{2\gamma+\eta+\epsilon} W_{\mathbf{X},\gamma,\eta+\epsilon}(\tau^{2n+1}_{m+1},\tau^{2n+2}_{m+1})^{\gamma-\eta-\epsilon} \\
			\leq\  &\left( \sum_{0\leq n<2^{m}} (1-\frac{2n}{2^{m+1}})^{\frac{\gamma(i-2)-\eta}{1-\gamma+\eta+\epsilon}}(\frac{1}{2^{m+1}})^{\frac{2\gamma+\eta+\epsilon)}{1-\gamma+\eta+\epsilon}} \right)^{1-\gamma+\eta+\epsilon} \left( \sum_{0\leq n<2^{m}}W_{\mathbf{X},\gamma,\eta+\epsilon}(\tau^{2n+1}_{m+1},\tau^{2n+2}_{m+1}) \right)^{\gamma-\eta-\epsilon} \\
			\leq\ &W_{\mathbf{X},\gamma,\eta+\epsilon}(s,t)^{\gamma-\eta-\epsilon} \left( \sum_{0\leq n<2^{m}} (1-\frac{2n}{2^{m+1}})^{\frac{\gamma(i-2)-\eta}{1-\gamma+\eta+\epsilon}}(\frac{1}{2^{m+1}})^{\frac{2\gamma+\eta+\epsilon)}{1-\gamma+\eta+\epsilon}} \right)^{1-\gamma+\eta+\epsilon}.
		\end{align*}
		Note that $\frac{1}{2^{m+1}}\leq 1-\frac{2n}{2^{m+1}}$. Therefore,  
		\begin{align*}
			\sum_{0\leq n<2^{m}} \left( 1-\frac{2n}{2^{m+1}} \right)^{\frac{\gamma(i-2)-\eta}{1-\gamma+\eta+\epsilon}} \left( \frac{1}{2^{m+1}} \right)^{\frac{2\gamma+\eta+\epsilon}{1-\gamma+\eta+\epsilon}} \leq \left( \frac{1}{2^m} \right)^{\frac{\frac{\epsilon}{2}}{1-\gamma+\eta+\epsilon}}\sum_{0\leq n<2^{m}}\left( 1-\frac{2n}{2^{m+1}} \right)^{\frac{\gamma(i+1)-1-\eta-\frac{\epsilon}{2}}{1-\gamma+\eta+\epsilon}}\frac{1}{2^{m+1}}.
		\end{align*}
		Since 
		\begin{align*}
			\lim_{m\rightarrow\infty} \sum_{0\leq n<2^{m}} \left( 1-\frac{2n}{2^{m+1}} \right)^{\frac{\gamma(i+1)-1-\eta-\frac{\epsilon}{2}}{1-\gamma+\eta+\epsilon}}\frac{1}{2^{m+1}} = \frac{1}{2}\int_{0}^{1}(1-x)^{\frac{\gamma(i+1)-\eta-1-\frac{\epsilon}{2}}{1-\gamma+\eta+\epsilon}}\, \mathrm{d} x,
		\end{align*} 
		we can conclude that for some $\tilde{M}_{\epsilon}<\infty$,
		\begin{align*}
			&\sum_{m\geq 1}\sum_{0\leq n<2^{m}} \left( 1-\frac{2n}{2^{m+1}} \right)^{\gamma(i-2)-\eta} \left( \frac{1}{2^{m+1}} \right))^{(2\gamma+\eta+\epsilon)} W_{\mathbf{X},\gamma,\eta+\epsilon}(\tau^{2n+1}_{m+1},\tau^{2n+2}_{m+1})^{\gamma-\eta-\epsilon} \\
			\leq\ &\tilde{M}_{\epsilon} W_{\mathbf{X},\gamma,\eta+\epsilon}(s,t)^{\gamma-\eta-\epsilon}.
		\end{align*}
		This proves our claim for the second term. For estimating the fourth term, we use the same idea with some modifications. First note that in \eqref{cn-es} from the interpolation property,
		$\vert Z^{\#}_{s,t}\vert_{\alpha-2\gamma+\eta}\lesssim (t-s)^{2\gamma-\eta}\vert Z^{\#}\vert_{\mathcal{E}^{\gamma,2\gamma}_{{\alpha};I}}$. For 
		\begin{align*}
			{B}^{i,\epsilon}_{\mathbf{X}}(s,t) \coloneqq (t-s)^{\gamma i+2\epsilon}W_{\mathbf{X},\gamma,\eta+\epsilon}(s,t)^{2(\gamma-\eta-\epsilon)},
		\end{align*} 
		we have
		\begin{align*}
			&\sum_{m\geq 0}\sum_{0\leq n<2^{m}} \left| S_{t-\tau^{2n}_{m+1}} \left( D_{Z_{\tau^{2n+1}_{m+1}}} G \left[ \int_{0}^{1}D_{Z_{\tau^{2n}_{m+1}}+\sigma Z_{\tau^{2n}_{m+1},\tau^{2n+1}_{m+1}}}G [Z^{\#}_{\tau^{2n}_{m+1},\tau^{2n+1}_{m+1}} \right] \, \mathrm{d}\sigma] \right) \circ \mathbb{X}_{\tau^{2n+1}_{m+1},\tau^{2n+2}_{m+1}} \right|_{\alpha-i\gamma} \\
			\lesssim\ &\vert Z^{\#}\vert_{\mathcal{E}^{\gamma,2\gamma}_{{\alpha};I}} \sum_{m\geq 0}\sum_{0\leq n<2^{m}}(t-\tau_{m+1}^{2n})^{\gamma(i-2)-\eta}(t-s)^{2\gamma-\eta} \left(\frac{1}{2^{m+1}} \right)^{2\gamma-\eta}\Vert\mathbb{X}_{\tau^{2n+1}_{m+1},\tau^{2n+2}_{m+1}}\Vert\ \\
			\leq\ &\vert Z^{\#}\vert_{\mathcal{E}^{\gamma,2\gamma}_{{\alpha};I}}\sum_{m\geq 0}\sum_{0\leq n<2^{m}}(t-\tau^{2n}_{m+1})^{\gamma(i-2)-\eta}(t-s)^{2\gamma+\eta+2\epsilon} \left(\frac{1}{2^{m+1}} \right)^{2\gamma+\eta+2\epsilon} W_{\mathbf{X},\gamma,\eta+\epsilon}(\tau^{2n+1}_{m+1},\tau^{2n+2}_{m+1})^{2(\gamma-\eta-\epsilon)} \\
			=\ &(t-s)^{\gamma i+2\epsilon}\vert Z^{\#}\vert_{\mathcal{E}^{\gamma,2\gamma}_{{\alpha};I}}\sum_{m\geq 0}\sum_{0\leq n<2^{m}} \left(1-\frac{2n}{2^{m+1}}\right)^{\gamma(i-2)-\eta} \left(\frac{1}{2^{m+1}}\right)^{2\gamma+\eta+2\epsilon} W_{\mathbf{X},\gamma,\eta+\epsilon}(\tau^{2n+1}_{m+1},\tau^{2n+2}_{m+1})^{2(\gamma-\eta-\epsilon)} \\
			\leq\ &{B}^{i,\epsilon}_{\mathbf{X}}(s,t)\vert Z^{\#}\vert_{\mathcal{E}^{\gamma,2\gamma}_{{\alpha};I}} \sum_{m\geq 0} \left( \sum_{0\leq n<2^{m}} \left(1-\frac{2n}{2^{m+1}}\right)^{\frac{\gamma(i-2)-\eta}{1-2\gamma+2\eta+2\epsilon}} \left(\frac{1}{2^{m+1}}\right)^{\frac{2\gamma+\eta+2\epsilon}{1-2\gamma+2\eta+2\epsilon}}\right)^{1-2\gamma+2\eta+2\epsilon} \\
			\leq\ &{B}^{i,\epsilon}_{\mathbf{X}}(s,t)\vert Z^{\#}\vert_{\mathcal{E}^{\gamma,2\gamma}_{{\alpha};I}}\sum_{m\geq 0} \left(\frac{1}{2^{m+1}} \right)^{\epsilon} \left( \sum_{0\leq n<2^{m}} \left( 1-\frac{2n}{2^{m+1}} \right)^{\frac{\gamma(i+2)-2\eta-\epsilon-1}{1-2\gamma+2\eta+2\epsilon}}\frac{1}{2^{m+1}}\right)^{1-2\gamma+2\eta+2\epsilon} \\
			\leq\ &\tilde{\tilde{M}}_{\epsilon}(t-s)^{\gamma i+2\epsilon}W_{\mathbf{X},\gamma,\eta+\epsilon}(s,t)^{2(\gamma-v-\epsilon)}\vert Z^{\#}\vert_{\mathcal{E}^{\gamma,2\gamma}_{{\alpha};I}},
		\end{align*}
		where $\tilde{{\tilde{M}}}_{\epsilon}<\infty$ only depends on $\epsilon$.
	\end{proof}
	
	Now we can prove the following lemma:
	\begin{lemma}\label{LKLL}
		Assume $\eta+\epsilon<\gamma$ and that $Z\in \mathcal{E}^{\gamma,2\gamma}_{{\alpha};I}$ solves equation \eqref{SPDE_EQU}. Then for $M_{\epsilon}<\infty$ and $i\in \lbrace 0,1,2\rbrace$,
		\begin{align}\label{ESTIMAT_LEVEL}
			\begin{split}
				&\left| \int_{s}^{t}S_{t-\tau}G(Z_{\tau})\circ\mathrm{d}\mathbf{X}_{\tau}-S_{t-s}G(Z_s)\circ (\delta X)_{s,t}-S_{t-s}D_{Z_s}G[G(Z_s)]\circ \mathbb{X}_{s,t} \right|_{\alpha-i\gamma}\\ 
				\leq\ & M_{\epsilon}(t-s)^{i\gamma}\max \left\{ (t-s)^{\epsilon} W_{\mathbf{X},\gamma,\eta+\epsilon}(s,t)^{\gamma-\eta-\epsilon}, (t-s)^{2\epsilon} W_{\mathbf{X},\gamma,\eta+\epsilon}(s,t) ^{2(\gamma-\eta-\epsilon)} \right\} \vert Z^{\#}\vert_{\mathcal{E}^{\gamma,2\gamma}_{{\alpha};I}} \\
				&\qquad + (t-s)^{i\gamma+3(\gamma-\eta)}\Vert X\Vert_{\gamma,[s,t]}[\Vert X\Vert_{\gamma,[s,t]}^{2}+\Vert \mathbb{X}\Vert_{2\gamma,[s,t]}]+(t-s)^{i\gamma+(\gamma-\eta)}\Vert {X}\Vert_{2\gamma,[s,t]}\\
				&\qquad +(t-s)^{i\gamma+2(\gamma-\eta)}\Vert \mathbb{X}\Vert_{2\gamma,[s,t]}.
			\end{split}
		\end{align}
	\end{lemma}
	\begin{proof}
		We use the same notation as in Lemma \ref{SSA}. Set $\Gamma_{s,t}^{m} \coloneqq \sum_{0\leq n\leq 2^m }\Xi^{n,m}_{s,t}$. By the Sewing lemma,  
		\begin{align*}
			&\left| \int_{s}^{t}S_{t-\tau}G(Z_{\tau})\circ\mathrm{d}\mathbf{X}_{\tau}-S_{t-s}G(Z_s)\circ (\delta X)_{s,t}-S_{t-s}D_{Z_s}G[G(Z_s)]\circ \mathbb{X}_{s,t} \right|_{\alpha-i\gamma} \\
			\leq\ &\sum_{m\geq 0}\vert\Gamma_{s,t}^{m+1}-\Gamma_{s,t}^{m}\vert_{\alpha-i\gamma}\leq \sum_{m\geq 0}\sum_{0\leq n<2^{m}}\vert\Xi^{2n,m+1}_{s,t}+\Xi^{2n+1,m+1}_{s,t}-\Xi^{n,m}_{s,t}\vert_{\alpha-i\gamma}.
		\end{align*}
		For the terms in the last sum, we use the identity provided in Lemma \ref{SSA}. To estimate the respective terms that involve $Z^{\#}$, we use the estimates from Proposition \ref{increment_estimate}. The terms on the right hand side of \eqref{ESTIMAT_LEVEL} that include $\Vert X\Vert_{\gamma,[s,t]}, \Vert \mathbb{X}\Vert_{2\gamma,[s,t]}$ appear when we want to find bounds for the remaining terms in \eqref{increment} where $Z^{\#}$ don't emerge. These estimates are even simpler to obtain, that is why we will only show the main ideas here. We will estimate the last term, the arguments for the rest are similar. Recall $D_{Z}G:\mathcal{B}_{\alpha-\eta}\rightarrow\mathcal{B}_{\alpha-2\eta}$ is a linear bounded operator. For $i=1,2$ we choose $\sigma_2=i\gamma,\sigma_{1}=2\eta$ and for $i=0$ we select $1-2\gamma<\sigma_2<1$ and $1-2\gamma+2\eta<\sigma_1<1$, such that $\sigma_{1}-\sigma_2=2\eta$ and apply on\eqref{SEMI_II}. Then
		\begin{align*}
			&\sum_{m\geq 0}\sum_{0\leq n<2^{m}} \left|S_{t-\tau^{2n+1}_{m+1}}\big(S_{\tau^{2n+1}_{m+1}-\tau^{2n}_{m+1}}-I\big)\big(D_{Z_{\tau^{2n+1}_{m+1}}}G[G(Z_{\tau^{2n+1}_{m+1}})]\circ \mathbb{X}_{\tau^{2n+1}_{m+1},\tau^{2n+2}_{m+1}} \big)\right|_{\alpha-i\gamma}\\&\quad\lesssim \Vert \mathbb{X}\Vert_{2\gamma,[s,t]}\sum_{m\geq 0}\sum_{0\leq n<2^{m}}(t-\tau^{2n+1}_{m+1})^{-\sigma_1}(t-s)^{i\gamma+2(\gamma-\eta)+\sigma_1}(\frac{1}{2^{m+1}})^{i\gamma+2(\gamma-\eta)+\sigma_1}\\&\qquad =(t-s)^{i\gamma+2(\gamma-\eta)}\Vert\mathbb{X}\Vert_{2\gamma,[s,t]}\sum_{m\geq 0}\sum_{0\leq n<2^{m}}(1-\frac{2n+1}{2^{m+1}})^{-\sigma_1}(\frac{1}{2^{m+1}})^{i\gamma+2(\gamma-\eta)+\sigma_1-1}\frac{1}{2^{m+1}}\\
			&\myquad[3]\lesssim (t-s)^{i\gamma+2(\gamma-\eta)}\Vert\mathbb{X}\Vert_{2\gamma,[s,t]}\sum_{m\geq 0}(\frac{1}{2^{m+1}})^{\epsilon}\sum_{0\leq n<2^{m}}(1-\frac{2n+1}{2^{m+1}})^{-\sigma_1}(\frac{1}{2^{m+1}})^{i\gamma+2(\gamma-\eta)+\sigma_1-1-\epsilon}\frac{1}{2^{m+1}}\\
			&\myquad[4]\lesssim (t-s)^{i\gamma+2(\gamma-\eta)}\Vert\mathbb{X}\Vert_{2\gamma,[s,t]}
		\end{align*}
		where in the last step, we choose $0<\epsilon<i\gamma+2(\gamma-\eta)+\sigma_1-1$ and again use the estimate
		\begin{align*}
			\sum_{0\leq n<2^{m}}(1-\frac{2n+1}{2^{m+1}})^{-\sigma_1}(\frac{1}{2^{m+1}})^{i\gamma+2(\gamma-\eta)+\sigma_1-1-\epsilon}\frac{1}{2^{m+1}}\lesssim \int_{0}^{1} (1-x)^{i\gamma+2(\gamma-\eta)-1-\epsilon}\mathrm{d}x<\infty.
		\end{align*}
	\end{proof}

	The following lemma is a straightforward application of Young's theory of integration.
	\begin{lemma}
		Assume $\mathbf{X}=(X,\mathbb{X})\in\mathscr{C}^{\gamma}([s,t],\mathbb{R}^n)$ is a $\gamma$-rough path with $\frac{1}{3} < \gamma \leq \frac{1}{2}$. Let $\gamma^{\prime} > 0$ with $\gamma+\gamma^{\prime}>1$. Assume that for a given path $h \colon [s,t]\rightarrow \mathbb{R}^{n}$,
		\begin{align}\label{variation}
			\sup_{\substack{\pi,\\ \pi=\lbrace s = \kappa_0 < \kappa_{1} < \ldots <\kappa_{m} = t \rbrace}}\big{\lbrace} \sum_{j}\big{[}| (\delta h)_{\kappa_{j},\kappa_{j+1}}|^{\frac{1}{{\gamma}^{\prime}}} \big{]}\big{\rbrace}<\infty.,
		\end{align} 
		Then this path can be enhanced to a rough path $\mathbf{h}=(h,\int h\otimes\mathrm{d}h)$ where the integrals are defined as Young integrals. In addition, $\int X\otimes\mathrm{d}h$ and $\int h\otimes\mathrm{d}X$ can also be defined as Young integrals.
	\end{lemma}
	
	The following result is an extension of \cite[Lemma 11.4]{FH20}.
	\begin{lemma}\label{traa}
		Assume that $I =[s,t]$ is a closed interval and for $\frac{1}{3}<\gamma<\frac{1}{2}$,  $\mathbf{X}=(X,\mathbb{X})$ is $\gamma$-rough path such that for $\eta_1<\gamma$, $W_{\mathbf{X},\gamma,\eta_1}(s,t)<\infty$. In addition, for $\gamma^{\prime}>0$ with $\gamma+\gamma^{\prime}>1$ and for $h: I\rightarrow \mathbb{R}^{d}$ being a continuous path satisfying \eqref{variation}, we assume that $W_{\mathbf{h},\gamma^{\prime},\eta_1}(s,t)<\infty$. If $\gamma+\gamma^{\prime}-2\eta_1>1$, then
		\begin{align}\label{Tra}
			&W_{T_{h}(\mathbf{X}),\gamma,\eta_1}(s,t)\leq C_{\eta_1}\big{[}W_{\mathbf{X},\gamma,\eta_1}(s,t)+ W_{\mathbf{h},\gamma^{\prime},\eta_1}(s,t)^{\frac{\gamma^{\prime}-\delta_1}{\gamma-\eta_1}}\big{]},
		\end{align}
		where $T_{h}(\mathbf{X})=(h+X,\int h\otimes\mathrm{d}h+\int h\otimes\mathrm{d}X+\int X\otimes\mathrm{d}h+\mathbb{X})$.
	\end{lemma}
	\begin{proof}
		Remember that by Young's integration theory,
		\begin{align}\label{SEW_VAR}
			\int_{s}^t (\delta h)_{s,\tau}\otimes\mathrm{d}X_{\tau} = \lim_{|\pi|\rightarrow 0}\int_{\pi}(\delta h)_{s,\tau}\circ\mathrm{d}X_{\tau} = \lim_{|\pi|\rightarrow 0}\sum_{0\leq i<m}(\delta h)_{s,\kappa_i}\otimes (\delta X)_{\kappa_i,\kappa_{i+1}},
		\end{align}
		where $\pi=\lbrace s =  \kappa_0 < \kappa_{1} < \ldots < \kappa_{m}=t \rbrace$ is a partition for $[s,t]$ and 
		\begin{align*}
			\int_{\pi}(\delta h)_{s,\tau}\circ\mathrm{d}X_{\tau} \coloneqq \sum_{0\leq i<m}(\delta h)_{s,\kappa_i}\otimes (\delta X)_{\kappa_i,\kappa_{i+1}}.
		\end{align*} 
		Clearly,
		\begin{align}\label{SEW_APX}
			\begin{split}
				&\left| \int_{\pi} (\delta h)_{s,\tau}\otimes\mathrm{d}X_{\tau} - \int_{\pi\setminus\lbrace \kappa_j\rbrace} (\delta h)_{s,\tau}\otimes\mathrm{d}X_{\tau} \right| \leq |(\delta h)_{\kappa_{j-1},\kappa_{j}} |  |(\delta X)_{\kappa_{j},\kappa_{j+1}}| \\
				\leq\ &\left[W_{\mathbf{h},\gamma^\prime,\eta_1}^{\frac{\gamma^\prime-\eta_1}{\gamma+{\gamma}^{\prime}-2\eta_1}}(\kappa_{j-1},\kappa_{j}) W_{\mathbf{X},\gamma,\eta_1}^{\frac{\gamma-\eta_1}{\gamma+{\gamma}^{\prime}-2\eta_1}}(\kappa_{j},\kappa_{j+1}) \right]^{\gamma+{\gamma}^{\prime}-2\eta_1}(\kappa_{j}-\kappa_{j-1})^{\eta_1}(\kappa_{j+1}-\kappa_{j})^{\eta_1}.
			\end{split}
		\end{align}
		Note that since $W_{\mathbf{X},\gamma,\eta_1}$ and $W_{\mathbf{h},\gamma^\prime,\eta_1}$ are control functions (c.f. Lemma \ref{cnt}), we can find  $1\leq j<m$ such that $$W_{\mathbf{h},\gamma^\prime,\tilde{\sigma}}^{\frac{\gamma^\prime-\eta_1}{\gamma+{\gamma}^{\prime}-2\eta_1}}(\kappa_{j-1},\kappa_{j})W_{\mathbf{X},\gamma,\eta_1}^{\frac{\gamma-\eta_1}{\gamma+{\gamma}^{\prime}-2\eta_1}}(\kappa_{j},\kappa_{j+1})<\frac{2}{m-1} W_{\mathbf{h},\gamma^\prime,\eta_1}^{\frac{\gamma^\prime-\eta_1}{\gamma+{\gamma}^{\prime}-2\eta_1}}(s,t)W_{\mathbf{X},\gamma,\eta_1}^{\frac{\gamma-\eta_1}{\gamma+{\gamma}^{\prime}-2\eta_1}}(s,t).$$ We can repeat this argument for the new partition $\pi\setminus\lbrace \kappa_j\rbrace$. From \eqref{SEW_VAR} and \eqref{SEW_APX}, we can eventually conclude that
		\begin{align*}
			&\left| \int_{s}^t (\delta h)_{s,\tau}\otimes\mathrm{d} X_{\tau} \right| \leq \sum_{k\geq 1}\frac{2^{\gamma+\gamma^{\prime}-2\eta_1}}{k^{\gamma+\gamma^{\prime}-2\eta_1}} (t-s)^{2\eta_1}W_{\mathbf{h},\gamma^\prime,\eta_1}^{\gamma^\prime-\eta_1}(s,t)W_{\mathbf{X},\gamma,\eta_1}^{\gamma-\eta_1}(s,t),
		\end{align*}
		therefore
		\begin{align*}
			\frac{\left| \int_{s}^t (\delta h)_{s,\tau}\otimes\mathrm{d}X_{\tau}\right|^{\frac{1}{2(\gamma-\eta_1)}}}{(t-s)^{\frac{\eta_1}{\gamma-\eta_1}}} \leq \left[ \sum_{k\geq 1}\frac{2^{\gamma+\gamma^{\prime}-2\eta_1}}{k^{\gamma+\gamma^{\prime}-2\eta_1}} \right]^{\frac{1}{2(\gamma-\eta_1)}}W_{\mathbf{h},\gamma^\prime,\eta_1}^{\frac{\gamma^\prime-\eta1}{2(\gamma-\eta_1)}}(s,t)W_{\mathbf{X},\gamma,\eta_1}^{\frac{1}{2}}(s,t).
		\end{align*}
		A analogous argument can be run to obtain similar bounds for $\int X\otimes\mathrm{d}h$ and $\int h\otimes\mathrm{d}h$. This finishes the proof.
	\end{proof}
	We will assume up to the end of this part:
	\begin{assumption}\label{Cameron-Martin}
		Let $(\mathcal{W},\mathcal{H},\mu)$ be an abstract Wiener space and assume that $X$ is a Gaussian process defined on it such that it can be enhanced to a geometric $\gamma$-H\"older rough path $\mathbf{X}=(X,\mathbb{X})$, $\frac{1}{3} < \gamma \leq \frac{1}{2}$. For every $h\in \mathcal{H}$, let Condition \eqref{variation} be fulfilled. In this case, by \cite[Lemma 5.4]{CLL13}, on a measurable subset $E\subset \mathcal{W}$ with full measure, 
		\begin{align}\label{ALM}
			T_{h}\mathbf{X}(\omega)\equiv \mathbf{X}(\omega +h) \quad \text{for all }\omega\in \mathcal{W} \text{ and }  h\in\mathcal{H}.
		\end{align} 
		We assume that for every $h\in \mathcal{H}$,
		\begin{align}\label{EBS}
			W_{\mathbf{h},\gamma^{\prime},\eta_1}(0,1)\lesssim\vert h\vert_{\mathcal{H}}^{\frac{1}{\gamma^{\prime}-\eta_1}}.
		\end{align}
	\end{assumption}
	In the following, we will show that the rough paths lift of a fractional Brownian in the sense of \cite{FV10} satisfies Assumption \ref{Cameron-Martin}.
	
	\begin{proposition}
		Assume $H\in (\frac{1}{4},\frac{1}{2})$ and let $B^{H}$ be a fractional Brownian motion with Hurst parameter $H$. Let $\mathcal{H}^{H}$ denote the  associated Cameron–Martin space for this process. Then Assumption \ref{Cameron-Martin} holds for every $\frac{1}{2}<\gamma^{\prime}<H+\frac{1}{2}$ and $\eta_1<\gamma^{\prime}-\frac{1}{2}$.
	\end{proposition}
	\begin{proof}
		We only have to check that \eqref{EBS} holds. Assume $n = 1$ first. For $\theta\in (0,1)$ and $q\in (1,\infty)$, for a measurable path $g \colon [0,1] \rightarrow \mathbb{R}$, define
		\begin{align*}
			\vert g\vert_{W^{\theta,q}} \coloneqq \left( \int_{[0,1]^2}\frac{|g(u)-g(v)|^q}{|u-v|^{1+\theta q}} \, \mathrm{d}u\, \mathrm{d}v \right)^{\frac{1}{q}}.
		\end{align*}$$$$
		Then $W^{\theta,q}$ is defined as the set of all measurable paths $g$ such that 
		\begin{align*}
			\vert g\vert_{W^{\theta,q}} + \left( \int_{0}^{1}|g(u)|^q \, \mathrm{d} u \right)^{\frac{1}{q}}<\infty
		\end{align*} 
		holds. We set
		\begin{align*}
			W^{\theta,q}_0 \coloneqq \lbrace g\in W^{\theta,q}:g(0)=0 \rbrace.
		\end{align*}$$$$
		By \cite[Theorem 3]{FV06b}, for $q=2$ and $\frac{1}{2}<\gamma^\prime<H+\frac{1}{2}$, $\mathcal{H}^{H}$ is compactly embedded in $W^{\gamma^\prime,2}_0$. Therefore, for a constant $C(\gamma^\prime)<\infty$,
		\begin{align}\label{CM_S}
			\left( \iint_{[0,1]^2}\frac{|h(u)-h(v)|^2}{|u-v|^{1+2\gamma^\prime}} \, \mathrm{d}u \, \mathrm{d}v \right)^{\frac{1}{2}}\leq C_1(\gamma^\prime)\vert h\vert_{\mathcal{H}^H}
		\end{align}
		for every $h \in \mathcal{H}^H$ where by $| \cdot |_{\mathcal{H}^H}$, we mean the corresponding  Hilbert norm in $\mathcal{H}^H$. Also, by the Besov–H\"older embedding theorem (cf. \cite[Corollary A.2]{FV10}), for all $0\leq s<t\leq 1$ and a constant $C_2(\gamma^\prime)<\infty$,
		\begin{align}\label{BH}
			\begin{split}
				|h(t)-h(s)|^{2} \leq C_{2}(\gamma^\prime)(t-s)^{2\gamma^\prime-1} |h|_{W^{\gamma^\prime,2}}^{2} = C_{2}(\gamma^{\prime})(t-s)^{2\gamma^\prime-1} \iint_{[s,t]^2}\frac{|h(u)-h(v)|^2}{|u-v|^{1+2\gamma^\prime}} \, \mathrm{d}u \, \mathrm{d}v.
			\end{split}
		\end{align}
		Now assume $\pi=\lbrace 0 = \kappa_0 < \kappa_{1} < \ldots < \kappa_{m}=1 \rbrace$. From \eqref{BH},
		\begin{align*}
			&\sum_{0\leq i<m}\frac{|h(\kappa_{i+1})-h(\kappa_{i})|^{\frac{1}{\gamma^{\prime}-\eta_1}}}{(\kappa_{i+1}-\kappa_{i})^{\frac{\eta_1}{\gamma^{\prime}-\eta_1}}}\\
			\leq\ &C_{2}(\gamma')^{\frac{1}{2(\gamma^\prime-\eta_1)}}\sum_{0\leq i<m}(\kappa_{i+1}-\kappa_{i})^{\frac{2\gamma^{\prime}-2\eta_1-1}{2(\gamma^{\prime}-\eta_1)}} \left( \iint_{[\kappa_{i},\kappa_{i+1}]^2}\frac{ |h(u)-h(v)|^2}{\vert u-v\vert^{1+2\gamma^{\prime}}}\, \mathrm{d}u\, \mathrm{d}v\right)^{\frac{1}{2(\gamma^{\prime}-\eta_1)}}.
		\end{align*}
		Now for $0<\eta_1<\gamma^{\prime}-\frac{1}{2}$, applying the H\"older inequality and \eqref{CM_S} yields
		\begin{align*}
			&\sum_{0\leq i<m}\frac{|h(\kappa_{i+1})-h(\kappa_{i})|^{\frac{1}{\gamma^{\prime}-\eta_1}}}{(\kappa_{i+1}-\kappa_{i})^{\frac{\eta_1}{\gamma^{\prime}-\eta_1}}} \\
			\leq\ &C_{2}(\gamma^{\prime})^{\frac{1}{2(\gamma^\prime-\eta_1)}} \left(	\sum_{0\leq i<m} \iint_{[\kappa_{i},\kappa_{i+1}]^2}\frac{|h(u)-h(v)|^2}{\vert u-v\vert^{1+2\gamma^{\prime}}}\, \mathrm{d}u\, \mathrm{d}v \right)^{\frac{1}{2(\gamma^{\prime}-\eta_1)}} \\
			\leq\ &C_{2}(\gamma^{\prime})^{\frac{1}{2(\gamma^\prime-\eta_1)}} \left( \iint_{[0,1]^2}\frac{|h(u)-h(v)|^2}{\vert u-v\vert^{1+2\gamma^{\prime}}} \, \mathrm{d}u\, \mathrm{d}v \right)^{\frac{1}{2(\gamma^{\prime}-\eta_1)}} \\
			\leq\ &C_{2}(\gamma^{\prime})^{\frac{1}{2(\gamma^\prime-\eta_1)}}C_1(\gamma^\prime)^{\frac{1}{\gamma^\prime-\eta_1}}\vert h\vert_{\mathcal{H}^H}^{\frac{1}{\gamma^\prime-\eta_1}}.
		\end{align*}
		Deriving the same bound for $h=(h_{1},...,h_{n})\in\mathbb{R}^n$ follows directly from the later inequality. For obtaining the corresponding bound for the iterated integral $\int h\otimes\mathrm{d}h$, we can proceed as before, using \eqref{SEW_APX} and the bound obtained for the increments of $h$.
	\end{proof} 
	
	\begin{remark}
		Similar to \cite[Theorem 1.1]{FGGR16}, we expect that a more general condition involving the mixed $(1,\rho)$-variation on the covariance function of a general Gaussian process can be formulated that implies Assumption \ref{Cameron-Martin}. 
	\end{remark}
	
	\begin{definition}\label{GRRED}
		For $I=[a,b]$, the \emph{sequence of greedy points} denoted by $\lbrace\tau^{I}_{n,\eta_1,\omega}(\chi)\rbrace_{n\geq 0}$ is defined by setting  $\tau^{I}_{0,\eta_{1},\omega}(\chi)=a$ and recursively
		\begin{align}\label{s}
			\begin{split}
				\tau_{n+1,\eta_1,\omega}^{I}(\chi) \coloneqq \sup\big\lbrace\tau:\  \tau^{I}_{n,\eta_1,\omega}(\chi)\leq \tau \leq b \ \ \text{and}\ \ W_{\mathbf{X}(\omega),\gamma,\eta_1}^{{\gamma-\eta_1}}(\tau_{n,\eta_1,\omega}^{I}(\chi),\tau)\leq \chi\big\rbrace.
			\end{split}
		\end{align}
		For $0<\eta_1< \gamma$ and $\chi>0$, we set 
		\begin{align*}
			N(I,\eta_1,\chi,\mathbf{X}(\omega)) \coloneqq \inf \lbrace n>0: \ \tau_{n,\eta_1,\omega}^{I}(\chi)=b\rbrace.
		\end{align*}
	\end{definition}
	
	The following proposition is analogous to \cite[Proposition 6.2]{CLL13}.
	\begin{proposition}\label{NAZ}
		Assume $\gamma+\gamma^\prime-2\eta>1$, choose $0<\epsilon<\frac{\gamma+\gamma^\prime-2\eta-1}{2}$ and set $\eta_1=\eta+\epsilon$. Then for a constant $T(\eta_1)<\infty$, we have
		\begin{align}\label{MMMN}
			\big[N\big{(}I,\eta_1,2^{\frac{1}{\gamma-\eta_1}}C_{\eta_1}W_{\mathbf{X}(\omega-h),\gamma,\eta_1}(a,b),\mathbf{X}(\omega)\big{)} -1\big] W_{\mathbf{X}(\omega-h),\gamma,\eta_1}^{{\frac{\gamma-\eta_1}{\gamma^{\prime}-\eta_1}}}(a,b)\leq T(\eta_1)\vert h\vert_{\mathcal{H}}^{\frac{1}{\gamma^{\prime}-\eta_1}},
		\end{align}
		where $C_{\eta_1}$ is the constant in \eqref{Tra}. Also for $\chi>0$, there exist $M_1(\eta_1,\chi), M_{2}(\eta_1,\chi)<\infty$ such that
		\begin{align}\label{tail}
			\mu\big{\lbrace}\omega: N(I,\eta_1,\chi,\mathbf{X}(\omega))> n\big{\rbrace}\leq M_1(\eta_1,\chi)\exp(-M_2(\eta_1,\chi)n^{2(\gamma^\prime-\eta_1)})
		\end{align}
		for every $n \geq 1$.
	\end{proposition}

	\begin{proof}
		From \eqref{ALM} and \eqref{Tra},
		\begin{align}\label{ZZSS}
			\begin{split}
				&W_{\mathbf{X}(\omega),\gamma,\eta}^{\gamma-\eta_1}(\tau_{n,\omega}^{I}(\chi),\tau_{n+1,\omega}^{I}(\chi)) = W_{T_{h}(\mathbf{X}(\omega-h)),\gamma,\eta_1}^{\gamma-\eta_1}(\tau_{n,\omega}^I(\chi),\tau_{n+1,\omega}^I(\chi))\\
				\leq\  &C_{\eta_1}^{\gamma-\eta_1}\big{[}W_{\mathbf{X}(\omega-h),\gamma,\eta_1}^{\gamma-\eta_1}(\tau_{n,\omega}^I(\chi),\tau_{n+1,\omega}^I(\chi))+ W_{\mathbf{h},\gamma^{\prime},\eta_1}^{\gamma^{\prime}-\eta_1}(\tau^{I}_{n,\eta_1,\omega}(\chi),\tau^{I}_{n+1,\eta_1,\omega}(\chi))\big{]}.
			\end{split}
		\end{align}
		Note that if $\tau_{n+1,\omega}^{I}(\chi)<b$, then the continuity of $W_{\mathbf{X}(\omega)}$ yields $$W_{\mathbf{X}(\omega),\gamma,\eta_1}(\tau_{n,\omega}^{I}(\chi),\tau_{n+1,\omega}^{I}(\chi))=\chi .$$
		Set $\chi \coloneqq 2^{\frac{1}{\gamma-\eta_1}}C_{\eta_1}W_{\mathbf{X}(\omega-h),\gamma,\eta_1}(a,b)$. From \eqref{ZZSS},
		\begin{align*}
			W_{\mathbf{X}(\omega-h),\gamma,\eta_1}^{{\frac{\gamma-\eta_1}{\gamma^{\prime}-\eta_1}}}(a,b)\leq W_{\mathbf{h},\gamma^{\prime},\eta_1}(\tau^{I}_{n,\eta_1,\omega}(\chi),\tau^{I}_{n+1,\eta_1,\omega}(\chi)) \quad \text{if}\ \tau_{n+1}^{I}(\chi)<b .
		\end{align*}
		Summing up yields
		\begin{align*}
			\big[N\big{(}I,\eta_1,2^{\frac{1}{\gamma-\eta_1}}C_{\eta_1}W_{\mathbf{X}(\omega-h),\gamma,\eta_1}(a,b),\mathbf{X}(\omega)\big{)} -1\big] W_{\mathbf{X}(\omega-h),\gamma,\eta_1}^{{\frac{\gamma-\eta_1}{\gamma^{\prime}-\eta_1}}}(a,b)\leq W_{\mathbf{h},\gamma^{\prime},\eta_1}(a,b).
		\end{align*}
		Now it is enough to use \eqref{EBS} to prove \eqref{MMMN}. For the second claim, let $\chi>0$. From \eqref{MMMN},
		\begin{align}\label{DSDD}
			\begin{split}
				&\left\{\omega \in \mathcal{W}\, :\, N(I,\eta_1,\chi,\mathbf{X}(\omega))> n \right\} \cap E \\
				\subset\ &\mathcal{W} \setminus \left\{ \omega \in \mathcal{W}\, :\, \frac{\chi}{2^{\frac{1}{\gamma-\eta_1}+1}C_{\eta_1}}\leq W_{\mathbf{X}(\omega),\gamma,\eta_1}(a,b) \leq\frac{\chi}{2^{\frac{1}{\gamma-\eta_1}}C_{\eta_1}} \right\} + r_{n}\mathcal{K},
			\end{split}
		\end{align}
		where $r_n:=\frac{(n-1)^{\gamma^{\prime}-\eta_1}\chi^{\gamma-\eta_1}}{{2^{\gamma-\eta_1+1}T(\eta_1)^{\gamma^{\prime}-\eta_1}C_{\eta_1}^{\gamma-\eta_1}}}$, $\mathcal{K}$ is the unit ball in $\mathcal{H}$ and $E$ is the set defined in Assumption \ref{Cameron-Martin}. Indeed: Obviously, if $\chi_1<\chi_2$, one has $N(I,\eta_1,\chi_2,\mathbf{X}(\omega))\leq N(I,\eta_1,\chi_1,\mathbf{X}(\omega))$. Let
		\begin{align*}
			\omega_1+h \in \underbrace{\left\{ \omega \in \mathcal{W}\, :\, \frac{\chi}{2^{\frac{1}{\gamma-\eta_1}+1}C_{\eta_1}}\leq W_{\mathbf{X}(\omega),\gamma,\eta_1}(a,b) \leq \frac{\chi}{2^{\frac{1}{\gamma-\eta_1}} C_{\eta_1}} \right\}}_{\eqqcolon A} + r_{n}\mathcal{K}.
		\end{align*} 
		Note that $A$ has positive $\mu$-measure due to the Support Theorem for Gaussian rough paths \cite{FV10}. Then from \eqref{MMMN}, setting $\omega=\omega_1+h$,
		\begin{align*}
			\big[N(I,\eta_1,\chi,\mathbf{X}(\omega_1+h))-1\big] \left( \frac{\chi}{2^{\frac{1}{\gamma-\eta_1}+1}C_{\eta_1}} \right)^{{\frac{\gamma-\eta_1}{\gamma^{\prime}-\eta_1}}}\leq  (n-1) \left(\frac{\chi}{2^{\frac{1}{\gamma-\eta_1}+1}C_{\eta_1}}\right)^{\frac{\gamma-\eta_1}{\gamma^{\prime}-\eta_1}}.
		\end{align*}
		Therefore \eqref{tail}, follows from \eqref{DSDD} and Borell's inequality (cf.\cite[Theorem D.4]{FV10})
	\end{proof}
	
	We come back now to the initial question of this section. Remember that the solution $Z$ to \eqref{SPDE_EQU} satisfies
	\begin{align}\label{composition}
		\begin{split}
			(\delta Z)_{s,t} &= (S_{t-s}-I)Z_{s}+\int_{s}^{t}S_{t-\tau}F(Z_{\tau}) \, \mathrm{d}\tau + S_{t-s}G(Z_s)\circ (\delta X(\omega))_{s,t} + S_{t-s}D_{Z_s}G[G(Z_s)]\circ \mathbb{X}_{s,t}(\omega)\\
			&\quad + \int_{s}^{t}S_{t-\tau}G(Z_{\tau})\circ\mathrm{d}\mathbf{X}_{\tau}(\omega) - S_{t-s}G(Z_s)\circ (\delta X)_{s,t}(\omega)-S_{t-s}D_{Z_s}G[G(Z_s)]\circ \mathbb{X}_{s,t}(\omega).
		\end{split}
	\end{align}
	Note that from Assumption \ref{existence} and \eqref{SEM_I}, for $i=0,1,2$ and every $0<u<v$ with $v-u<1$,
	\begin{align}\label{Drift_Estimate}
		\begin{split}
			\left|\int_{u}^{v}S_{v-\tau}F(Z_{\tau}) \, \mathrm{d}\tau \right|_{\alpha-i\gamma} &\leq \int_{u}^{v}\left| S_{v-\tau}F(Z_{\tau})\right|_{\alpha-i\gamma}\, \mathrm{d}\tau\\
			&\lesssim (1+\sup_{\tau\in [u,v]}\vert Z_\tau\vert_{\alpha})\int_{u}^{v}\max\lbrace(v-\tau)^{-\sigma+i\gamma},1\rbrace \, \mathrm{d}\tau\\
			&\lesssim (v-u)^{\min\lbrace 1,1-\sigma+i\gamma\rbrace}(1+\sup_{\tau\in [u,v]}\vert Z_\tau\vert_{\alpha}).
		\end{split}
	\end{align}
	Also, since $Z^\prime_{s}=G(Z_s)$, 
	\begin{align}\label{TTT}
		\begin{split}
			|(\delta G(Z))_{u,v} \vert_{\alpha-2\gamma} &\lesssim \vert (\delta Z)_{u,v}\vert_{\alpha-2\gamma+\eta}\lesssim   \vert (\delta Z)_{u,v}\vert_{\alpha-2\gamma}^{\frac{\gamma-\eta}{\gamma}}\vert(\delta Z)_{u,v}\vert_{\alpha-\gamma}^{\frac{\eta}{\gamma}} \quad \text{and} \\ 
			Z^{\#}_{u,v} &= (\delta Z)_{u,v}-G(Z_{u})\circ(\delta X(\omega))_{u,v}.
		\end{split}
	\end{align}
	Set $L(x) \coloneqq \max \lbrace x,x^2\rbrace$, $\bar{\sigma}:=\max\lbrace \sigma,2\gamma\rbrace$ and assume for $I=[s,t]$ that  $t-s\leq 1$. From Lemma \ref{LKLL} and Assumption \ref{existence}, also \eqref{Drift_Estimate} and \eqref{TTT}, we can conclude that for an $M_\epsilon>1$ which is independent from $\mathbf{X}$ it holds that
	\begin{align}\label{Grid}
		\begin{split}
			\Vert (Z,G(Z))\Vert_{\mathcal{D}_{\mathbf{X},{\alpha}}^{\gamma}([u,v])} \leq &M_\epsilon\Big{[}\vert Z_u\vert_{\alpha}+(v-u)^{1-\bar{\sigma}}\Vert (Z,G(Z))\Vert_{\mathcal{D}_{\mathbf{X}(\omega),{\alpha}}^{\gamma}([u,v])} + 1\\
			&\qquad + \Vert X\Vert_{\gamma,I}\big{(}\Vert X(\omega)\Vert^{2}_{\gamma,I}+\Vert\mathbb{X}(\omega)\Vert_{2\gamma,I}\big{)}+\Vert X(\omega)\Vert_{\gamma,I}+\Vert\mathbb{X}(\omega)\Vert_{2\gamma,I}\\
			&\qquad + L\big(W_{\mathbf{X}(\omega),\gamma,\eta+\epsilon}(u,v)^{\gamma-\eta-\epsilon}\big)\Vert (Z,G(Z))\Vert_{\mathcal{D}_{\mathbf{X}(\omega),{\alpha}}^{\gamma}([u,v])}\Big{]},
		\end{split}
	\end{align}
	where $[u,v]\subset [s,t]$.	Choose $0<\chi<1$ such that $M_\epsilon\chi^{\gamma-\eta-\epsilon}\leq\frac{1}{4}$. We will assume further that $M_\epsilon(t-s)^{1-\bar{\sigma}}\leq \frac{1}{4}$. Let $\lbrace\tau^{I}_{n,\eta_1,\omega}(\chi)\rbrace_{n\geq 0}$ be the greedy points that are defined in Definition \ref{GRRED} with $\eta_{1}=\eta+\epsilon$ such that $0<\epsilon<$  .
From \eqref{Grid},
\begin{align}\label{DDCCXX}
	\Vert (Z,G(Z))\Vert_{\mathcal{D}_{\mathbf{X}(\omega),{\alpha}}^{\gamma}([\tau_{n},\tau_{n+1}])}\leq 2M_\epsilon|Z_{\tau_{n}}|_{\alpha}+2M_\epsilon P(\Vert X(\omega)\Vert_{\gamma,I},\Vert\mathbb{X}(\omega)\Vert_{2\gamma,I}),
\end{align}
where
\begin{align*}
	P(\Vert X(\omega)\Vert_{\gamma,I},\Vert\mathbb{X}(\omega)\Vert_{\gamma,I}) = 1+\Vert X(\omega)\Vert_{\gamma,I}+\Vert \mathbb{X}(\omega)\Vert_{\gamma,I}+\Vert X(\omega)\Vert_{\gamma,I}(\Vert X(\omega)\Vert_{\gamma,I}^{2}+\Vert\mathbb{X}(\omega)\Vert_{2\gamma,I}).
\end{align*}
Therefore, for $\tilde{M}_\epsilon \coloneqq \log(2M_\epsilon)$,
\begin{align}\label{BBY}
	\begin{split}
		\sup_{\tau\in [s,t]}\vert Z_{\tau}\vert_{\alpha} &\leq  \exp\big{(}N(I,\eta_1,\chi,\mathbf{X}(\omega))\tilde{M}_\epsilon\big{)}\vert Z_s\vert_{\alpha}\\
		&\quad + \frac{\exp\big{(}N(I,\eta_1,\chi,\mathbf{X}(\omega))\tilde{M}_\epsilon+\tilde{M}_\epsilon\big{)}-1}{2M_\epsilon-1}P(\Vert X(\omega)\Vert_{\gamma,I},\Vert\mathbb{X}(\omega)\Vert_{\gamma,I}).
	\end{split}
\end{align}
We can now summarize our main result in the following theorem:
\begin{theorem}\label{thm:integrability_RPDE}
	Suppose that $F \colon \mathcal{B}_{\alpha}\rightarrow\mathcal{B}_{\alpha-\sigma}$ is a locally Lipschitz continuous function with linear growth and for $\theta\in \lbrace 0,\gamma,2\gamma\rbrace$, $G \colon \mathcal{B}_{\alpha-\theta}\rightarrow\mathcal{B}_{\alpha-\theta-\eta}$ is a bounded Fr\'echet differentiable function with 3 bounded derivatives or a bounded linear function. Consider the equation
	\begin{align}\label{BMMM}
		\mathrm{d}Z_t = AZ_t \, \mathrm{d}t + F(Z_t)\, \mathrm{d}t + G(Z_t)\circ\mathrm{d}\mathbf{X}_t, \ \ \ \  Z_{0}=\xi\in \mathcal{B}_{\alpha}.
	\end{align}
	Then the following holds:
	\begin{itemize}
		\item[(i)] \eqref{BMMM} admits a unique and global solution $Z$ such that for $\eta_{1}=\eta+\epsilon$ and $\bar{\sigma}:=\max\lbrace \sigma,2\gamma\rbrace$, one has
		\begin{align}\label{integrable bound}
			\begin{split}
				\sup_{\tau\in [s,t]}|Z_{\tau}|_{\alpha} &\leq \sup_{0\leq n\leq N(I,\eta_1,\chi,\mathbf{X}(\omega))-1}\Vert (Z,G(Z))\Vert_{\mathcal{D}_{\mathbf{X}(\omega),{\alpha}}^{\gamma}([\tau_{n},\tau_{n+1}])} \leq\exp\big{(}N([s,t],\eta_1,\chi,\mathbf{X})\tilde{M}_\epsilon\big{)}|Z_s|_{\alpha} \\
				&\quad + \frac{\exp\big{(}N([s,t],\eta_1,\chi,\mathbf{X})\tilde{M}_\epsilon+\tilde{M}_\epsilon\big{)}-1}{2M_\epsilon-1}P(\Vert X\Vert_{\gamma,[s,t]},\Vert\mathbb{X}\Vert_{\gamma,[s,t]}),
			\end{split}
		\end{align}
		where $M_{\epsilon}>1$, $ M_\epsilon\chi^{\gamma-\eta-\epsilon}\leq\frac{1}{4}, \ M_\epsilon(t-s)^{1-\bar{\sigma}}\leq\frac{1}{4}$ and  $\tilde{M}_\epsilon:=\log(2M_\epsilon)$. In addition, $P(x,y)=1+y+ x+x(x^2+y)$.
		
		\item [(ii)] For a constant $\tilde{M}>0$, we have the following bound
		\begin{align}\label{integrable bound_2}
			\begin{split}
				&	\Vert (Z,G(Z))\Vert_{\mathcal{D}_{\mathbf{X},{\alpha}}^{\gamma}([s,t])}\leq  \tilde{M}N([s,t],\eta_1,\chi,\mathbf{X})(1+\Vert X\Vert_{\gamma,[s,t]})\bigg[\exp\big{(}N([s,t],\eta_1,\chi,\mathbf{X})\tilde{M}_\epsilon\big{)} |Z_s|_{\alpha} \\
				&\quad + \frac{\exp\big{(}N([s,t],\eta_1,\chi,\mathbf{X})\tilde{M}_\epsilon+\tilde{M}_\epsilon\big{)}-1}{2M_\epsilon-1}P(\Vert X\Vert_{\gamma,[s,t]},\Vert\mathbb{X}\Vert_{\gamma,[s,t]})\bigg] .
			\end{split}
		\end{align}
		
		\item[(iii)] Let $(\mathcal{W},\mathcal{H},\mu)$ be an abstract Wiener space and assume that $X$ is a Gaussian process defined on it that can be enhanced to a weakly $\gamma$-geometric rough path $\mathbf{X}(\omega)=(X(\omega),\mathbb{X}(\omega))$, $\frac{1}{3} < \gamma \leq \frac{1}{2}$. In addition, let $\gamma^{\prime}>0$ satisfy $\gamma+\gamma^{\prime}-2\eta>1$, and for every $h\in \mathcal{H}$, for $\eta_1=\eta+\epsilon$ such that $0<\epsilon<\frac{\gamma+\gamma^\prime-2\eta-1}{2}$, we assume that Condition \eqref{variation} and Assumption \ref{Cameron-Martin} hold. 
		Then
		\begin{align}\label{integrability_2}
			\Vert (Z,G(Z))\Vert_{\mathcal{D}_{\mathbf{X},{\alpha}}^{\gamma}([s,t])}\in\cap_{p\geq 1}\mathcal{L}^{p}(\mathcal{W}).
		\end{align}
	\end{itemize} 
\end{theorem}

\begin{proof}
	The first item is proved in \eqref{BBY}.
	The second item follows from this fact that for every $Z\in \mathcal{D}_{\mathbf{X},{\alpha}}^{\gamma}([0,T])$, and $u,v,w\in [s,t]$ with $u<v<w$, $Z^{\#}_{s,t}=Z^{\#}_{s,u}+Z^{\#}_{u,t}+Z^{\prime}_{s,u}(\delta X)_{u,t}$, therefore for a constant $\tilde{M}>1$,
	\begin{align}\label{total_norm}
		\begin{split}
			\Vert(Z,G(Z))\Vert_{\mathcal{D}_{\mathbf{X},{\alpha}}^{\gamma}([s,t])} \leq\ \tilde{M}(1+\Vert X\Vert_{\gamma,[0,T]})\sum_{0\leq n\leq N(I,\eta_1,\chi,\mathbf{X}(\omega))-1}\Vert(Z,G(Z))\Vert_{\mathcal{D}_{\mathbf{X},{\alpha}}^{\gamma}([\tau_{n},\tau_{n+1}])}.
		\end{split}
	\end{align}
	Then, the claim follows from (i). For the integrability claim (iii), by \eqref{integrable bound_2}, it is enough to prove that
	\begin{align}\label{integrability}
		\begin{split}
			&N([s,t],\eta_1,\chi,\mathbf{X}(\omega))\exp\big{(}N([s,t],\eta_1,\chi,\mathbf{X}(\omega))\tilde{M}_\epsilon\big{)}(1+\Vert X(\omega)\Vert_{\gamma,[s,t]})\\
			&\quad \times P(\Vert X(\omega)\Vert_{\gamma,[s,t]},\Vert\mathbb{X}(\omega)\Vert_{\gamma,[s,t]})\in\mathcal{L}^p(\mathcal{W})
		\end{split}
	\end{align}
	for every $p \geq 1$. Remember that from \eqref{tail}, we know that
	\begin{align*}
		\mu\big{\lbrace}\omega \in \mathcal{W}\, :\, N(I,\eta_1,\chi,\mathbf{X}(\omega))> n\big{\rbrace}\leq M_1(\eta_1,\chi)\exp(-M_2(\eta_1,\chi)n^{2(\gamma^\prime-\eta_1)}).
	\end{align*}
	Since by assumption $\gamma+\gamma^\prime-2\eta_1>1$, we can easily conclude $2(\gamma^\prime-\eta_1)>1$. This proves integrability of $\exp\big{(}N([s,t],\eta_1,\chi,\mathbf{X})\tilde{M}_\epsilon\big{)}$. Since $P$ is a polynomial term and $\mathbf{X}$ is a Gaussian process, the integrability for every moment is clear. Finally, the integrability of the product of these terms is a straightforward consequence of H\"older's inequality.
\end{proof}

\begin{remark}
	In Assumption \ref{existence}, we assumed that $G$ is a bounded and Fr\'echet-differentiable function with bounded derivatives or a bounded linear function. In fact, we can even weaken a bit our assumptions on $G$ and replace it by the following:
	\begin{itemize}
		\item  	For $\theta\in\lbrace 0,\gamma,2\gamma\rbrace$ , $G \colon \mathcal{B}_{\alpha-\theta}\rightarrow\mathcal{B}_{\alpha-\theta-\eta}^{n}$ is  Fr\'echet differentiable function up to three times with bounded derivatives and $D_{.}G[G(.)]:\mathcal{B}_{\alpha-\theta}\rightarrow\mathcal{B}_{\alpha-\theta-\eta}$ assumed to be bounded.
	\end{itemize}
	For (finite-dimensional) rough differential equations, the corresponding condition was used to prove that global solutions exist \cite{Lej12}. Note that this assumption covers Assumption \ref{existence}. We do not want to reformulate our main results to keep the calculations as simple as possible. However, to sketch the main idea, the point is to use another representation of the remainder term $[G(Z)]^{\#}$. In fact, one has to use the identity
	\begin{align}\label{Remaider_1}
		\begin{split}
			[G(Z)]^{\#}_{s,t} &= G(Z_t)-G(Z_s)-D_{Z_{s}}G [G(Z_s)\circ (\delta X)_{s,t}] \\
			&= \int_{0}^{1}\left(D_{Z_s+\sigma(\delta Z)_{s,t}}G[G(Z_s)\circ(\delta X)_{s,t}]-D_{Z_s}G[G(Z_s)\circ (\delta X)_{s,t}]\right) \, \mathrm{d}\sigma \\
			&\quad +\int_{0}^{1}D_{Z_s+\sigma(\delta Z)_{s,t}}G [Z^{\#}_{s,t}]\, \mathrm{d}\sigma\\
			&= \int_{0}^{1}\left(D_{Z_s+\sigma(\delta Z)_{s,t}}G[G(Z_s+\sigma(\delta Z)_{s,t})\circ(\delta X)_{s,t}]-D_{Z_s}G[G(Z_s)\circ (\delta X)_{s,t}]\right)\, \mathrm{d}\sigma\\
			&\quad - \int_{0}^{1}D_{Z_s+\sigma(\delta Z)_{s,t}}G\big[\int_{0}^{1}\sigma D_{Z_s+\sigma u(Z)_{s,t}}G[G(Z_s)\circ(\delta X)_{s,t}]\mathrm{d}u\circ(\delta X_{s,t})\big]\, \mathrm{d}\sigma\\
			&\quad - \int_{0}^{1}D_{Z_s+\sigma(\delta Z)_{s,t}}G\big[\int_{0}^{1}uD_{Z_s+\sigma u(Z)_{s,t}}G[Z^{\#}_{s,t}]\, \mathrm{d}u\circ(\delta X_{s,t})\big]\, \mathrm{d}\sigma \\
			&\quad +\int_{0}^{1}D_{Z_s+\sigma(\delta Z)_{s,t}}G [Z^{\#}_{s,t}]\, \mathrm{d}\sigma
		\end{split}
	\end{align}
	in \eqref{Increment} and to reformulate the other lemmas accordingly.
\end{remark}

\section{Linearization}
The crucial step to prove the existence of invariant manifolds is to differentiate the flow and to control the growth of it. In this section, first we  address the regularity (in Fr\'echet's sense) of the solution map induced by equation \eqref{SPDE_EQU} with respect to the initial value and then derive some inequalities for our future goals.\smallskip

Remember that we are dealing with the equation
\begin{align}\label{SPDE_EQU1}
	\mathrm{d}Z_t=AZ_t \, \mathrm{d}t + F(Z_t)\, \mathrm{d}t + G(Z_t)\circ\mathrm{d}\mathbf{X}_t, \ \ \ \  Z_{0}=\xi\in \mathcal{B}_{\alpha}.
\end{align}
The proof of existence and uniqueness of the solution is based on a standard fixed-point argument for the map
\begin{align*}
	&\mathcal{F} \colon \mathcal{D}_{\mathbf{X},{\alpha},0}^{\gamma}([0,T])\cap \overline{ B_{\mathcal{D}_{\mathbf{X},{\alpha}}^{\gamma}([0,T])}(0,\epsilon)}\times \mathcal{B}_\alpha\cap \overline{B_{\mathcal{B}_\alpha}(0,M)}\longrightarrow \mathcal{D}_{\mathbf{X},{\alpha},0}^{\gamma}([0,T])\cap \overline{ B_{\mathcal{D}_{\mathbf{X},{\alpha}}^{\gamma}([0,T])}(0,\epsilon)},\\
	&\quad \mathcal{F}(W,\xi)(t) \coloneqq S_{t}\xi-\xi + \int_{0}^{t}S_{t-\tau} F(W_{\tau}+\xi)\, \mathrm{d}\tau + \int_{0}^{t}S_{t-\tau}G(W_{\tau}+\xi)\circ\mathrm{d}\mathbf{X}_\tau,
\end{align*} 
where $M>0$,
\begin{align*}
	\mathcal{D}_{\mathbf{X},{\alpha},0}^{\gamma}([0,T]) \coloneqq \lbrace W\in\mathcal{D}_{\mathbf{X},{\alpha}}^{\gamma}([0,T])\, :\,  W_0=0 \rbrace,
\end{align*} 
and $\epsilon=\epsilon(M),\ T=T(M)$ might depend on $M$. We will further assume that $F \colon \mathcal{B}_\alpha\rightarrow \mathcal{B}_{\alpha-\sigma}$ is a $C^{1}$-function in Fr\'echet's sense. In this case, since also $G\in C^3$ (cf. Assumption \ref{existence}), we conclude that $\mathcal{F}$ is a $C^1$-map. Note that for every $M>0$, the parameters $T$ and $\epsilon$ can be chosen in such a way that $\mathcal{F}$ is contraction, i.e. for $\mu<1$,
\begin{align*}
	\Vert\mathcal{F}(W,\xi)-\mathcal{F}(W^{\prime},\xi)\Vert_{\mathcal{D}_{\mathbf{X},{\alpha}}^{\gamma}([0,T])}\leq \mu\Vert W-W^\prime\Vert_{\mathcal{D}_{\mathbf{X},{\alpha}}^{\gamma}([0,T])}\quad \text{for all } (W,\xi),(W^\prime,\xi)\in \text{dom}(\mathcal{F}).
\end{align*}
This yields for $\mathcal{F}^{\prime}(W,\xi):=\mathcal{F}(W,\xi)-W$ that
\begin{align*}
	\frac{\partial\mathcal{F}^{\prime}(W,\xi)}{\partial W}
\end{align*}
is an isomorphism. By the implicit function theorem for Banach spaces, cf.\cite[Theorem 2.5.7]{AMR88}, for every $\xi\in \mathcal{B}_\alpha\cap \overline{B_{\mathcal{B}_\alpha}(0,M)}$, there exists a unique $W(\xi)\in\mathcal{D}_{\mathbf{X},{\alpha},0}$, such that $\mathcal{F}^{\prime}(W(\xi),\xi)=0$. Since $\mathcal{F}$ is $C^1$, then $\xi \mapsto W(\xi)$ is also $C^1$. Therefore, $W(\xi)+\xi$, which is the solution to \eqref{SPDE_EQU1}, is a differentiable function. Note that our solution is not exploding in a finite time due to \eqref{integrable bound}. Therefore, for every $T_0>0$, by repeating this argument and gluing together the solutions, the differentiability of the solution at every $T_0$ can be verified. \smallskip

Let us agree to use $\varphi^{t}_{\mathbf{X}}(\xi)$ to denote the solution to \eqref{SPDE_EQU1} at time $t$ with initial value $Z_0=\xi$. 
We already showed that $\varphi^{t}_{\mathbf{X}}(\xi)$ is differentiable. We will use $D_{\xi}\varphi^{t}_{\mathbf{X}}[\zeta]$ to denote the derivative of $\varphi^{t}_{\mathbf{X}}(\xi)$ at $\xi$ in direction $\zeta$. We claim that $D_{\xi}\varphi^{t}_{\mathbf{X}}[\zeta]$ satisfies the equation
\begin{align}\label{LLLL}
	\mathrm{d}D_{\xi}\varphi^{t}_{\mathbf{X}}[\zeta] = AD_{\xi}\varphi^{t}_{\mathbf{X}}[\zeta]\, \mathrm{d}t + D_{\varphi^{t}_{\mathbf{X}}(\xi)}F[D_{\xi}\varphi^{t}_{\mathbf{X}}[\zeta]]\,\mathrm{d}t + D_{\varphi^{t}_{\mathbf{X}}(\xi)}G[D_{\xi}\varphi^{t}_{\mathbf{X}}[\zeta]]\circ\mathrm{d}\mathbf{X}_t,\ \ \ \  D_{\xi}\varphi^{0}_{\mathbf{X}}[\zeta]=\zeta,
\end{align}
or equivalently
\begin{align}\label{LLL}
	\begin{split}
		D_{\xi}\varphi^{t}_{\mathbf{X}}[\zeta]
		= S_{t-s}D_{\xi}\varphi^{s}_{\mathbf{X}}[\zeta] + \int_{s}^{t}S_{t-\tau}D_{\varphi^{\tau}_{\mathbf{X}}(\xi)}F[D_{\xi}\varphi^{\tau}_{\mathbf{X}}[\zeta]]\,\mathrm{d}\tau+\int_{s}^{t}S_{t-\tau}D_{\varphi^{\tau}_{\mathbf{X}}(\xi)}G[D_{\xi}\varphi^{\tau}_{\mathbf{X}}[\zeta]]\circ\mathrm{d}\mathbf{X}_{\tau}.
	\end{split}
\end{align}
In fact, the proof of formula \eqref{LLL} is relatively straightforward. We already showed, using the implicit function theorem, that
\begin{align}\label{limitt}
	\lim_{\epsilon\rightarrow 0} \left\| \frac{\varphi^{.}_{\mathbf{X}}(\xi+\epsilon\zeta)-\varphi^{.}_{\mathbf{X}}(\xi)}{\epsilon} - D_{\xi}\varphi^{.}_{\mathbf{X}}[\zeta] \right\|_{\mathcal{D}_{\mathbf{X},{\alpha}}^{\gamma}([0,T])} = 0
\end{align}
holds true. By definition,
\begin{align*}
	&\frac{\delta(\varphi^{.}_{\mathbf{X}}(\xi+\epsilon\zeta))_{s,t}-\delta(\varphi^{.}_{\mathbf{X}}(\xi))_{s,t}}{\epsilon} - (S_{t-s}-I)D_{\xi}\varphi^{s}_{\mathbf{X}}[\zeta]-\int_{s}^{t}S_{t-\tau}D_{\varphi^{\tau}_{\mathbf{X}}(\xi)} F[D_{\xi}\varphi^{\tau}_{\mathbf{X}}[\zeta]]\, \mathrm{d}\tau \\
	&\quad - \int_{s}^{t}S_{t-\tau}D_{\varphi^{\tau}_{\mathbf{X}}(\xi)}G[D_{\xi}\varphi^{\tau}_{\mathbf{X}}[\zeta]]\circ\mathrm{d}\mathbf{X}_{\tau} \\
	=\ &(S_{t-s}-I)(\frac{\varphi^{s}_{\mathbf{X}}(\xi+\epsilon\zeta)-\varphi^{s}_{\mathbf{X}}(\xi)}{\epsilon}-D_{\xi}\varphi^{s}_{\mathbf{X}}[\zeta]) \\
	&\quad - \int_{s}^{t} S_{t-\tau}(\frac{F(\varphi^{\tau}_{\mathbf{X}}(\xi+\epsilon\zeta)-F(\varphi^{\tau}_{\mathbf{X}}(\xi)}{\epsilon}-D_{\varphi^{\tau}_{\mathbf{X}}(\xi)}F[D_{\xi}\varphi^{\tau}_{\mathbf{X}}[\zeta]])\, \mathrm{d}\tau \\
	&\quad - \int_{s}^{t} S_{t-\tau}(\frac{G(\varphi^{\tau}_{\mathbf{X}}(\xi+\epsilon\zeta)-G(\varphi^{\tau}_{\mathbf{X}}(\xi)}{\epsilon}-D_{\varphi^{\tau}_{\mathbf{X}}(\xi)}G[D_{\xi}\varphi^{\tau}_{\mathbf{X}}[\zeta]])\circ\mathrm{d}\mathbf{X}_\tau .
\end{align*}
From \eqref{limitt} and our assumptions on $F$ and $G$, it can be shown that
\begin{align*}
	\lim_{\epsilon\rightarrow 0}\sup_{\tau\in [0,T]} \left| \frac{F(\varphi^{\tau}_{\mathbf{X}}(\xi+\epsilon\zeta))-F(\varphi^{\tau}_{\mathbf{X}}(\xi))}{\epsilon}-D_{\varphi^{\tau}_{\mathbf{X}}(\xi)}F[D_{\xi}\varphi^{\tau}_{\mathbf{X}}[\zeta]] \right|_{\alpha-\eta} &= 0 \quad \text{and} \\
	\lim_{\epsilon\rightarrow 0} \left\| \frac{G(\varphi^{.}_{\mathbf{X}}(\xi+\epsilon\zeta))-G(\varphi^{.}_{\mathbf{X}}(\xi))}{\epsilon}-D_{\varphi^{.}_{\mathbf{X}}(\xi)}G[D_{\xi}\varphi^{.}_{\mathbf{X}}[\zeta]]\right\|_{\mathcal{D}_{\mathbf{X},{\alpha}}^{\gamma}([0,T])} &= 0
\end{align*}
which yields the identity \eqref{LLL}. \smallskip

In the next proposition, we obtain an a priori bound for the solution to equation \eqref{LLL}.
\begin{proposition}\label{linear boundd}
	Let $\xi,\zeta \in \mathcal{B}_\alpha$ and $I=[0,T]$. Assume that $DF:\mathcal{B}_{\alpha}\rightarrow\mathcal{L}(\mathcal{B}_\alpha,\mathcal{B}_{\alpha-\sigma}) $ is locally Lipschitz and there exists a polynomial $P_1$ such that for every $\tilde{\xi}\in\mathcal{B}_{\alpha}$,
	\begin{align}\label{poly}
		\Vert D_{\tilde{\xi}}F\Vert_{\mathcal{L}(\mathcal{B}_\alpha,\mathcal{B}_{\alpha-\sigma})}\leq P_1(\vert\tilde{\xi} \vert_{\alpha}).
	\end{align}
	Set $\bar{\sigma} \coloneqq \max\lbrace \sigma,2\gamma\rbrace$. Then there exists a constant $M > 0$ such that 
	\begin{align}\label{bounde_derivative}
		\Vert (D_{\xi}\varphi^{.}_{\mathbf{X}}[\zeta],D_{\varphi^{.}_{\mathbf{X}}}G[ D_{\xi}\varphi^{.}_{\mathbf{X}}[\zeta]])\Vert_{\mathcal{D}_{\mathbf{X},{\alpha}}^{\gamma}([0,T])}\leq M \Vert \zeta\Vert_{\alpha}\exp\big(M S(0,T,\xi,\mathbf{X})\big),
	\end{align}
	where
	\begin{align}
		\begin{split}\label{eqn:def_S}
			S(0,T,\xi,\mathbf{X}) &\coloneqq \log(\varrho_{\gamma}(\mathbf{X},[0,T])) \\
			&\quad \times \left(\left[\sup_{\tau\in [0,T]}P_1(\vert \varphi^{\tau}_{\mathbf{X}}(\xi) \vert_{\alpha})+\varrho_{\gamma}^3(\mathbf{X},[0,T])\big(1+\Vert(\varphi^{.}_{\mathbf{X}}(\xi),G(\varphi^{.}_{\mathbf{X}}(\xi)))\Vert_{\mathcal{D}_{\mathbf{X},{\alpha}}^{\gamma}([0,T])}^2\big)\right]^{\frac{1}{\min\lbrace 1-\bar{\sigma},\gamma-\eta\rbrace}}+1\right)
		\end{split}
	\end{align}

\end{proposition}
\begin{proof}
	From \eqref{LLL}, for $u<v$,
	\begin{align}\label{CCBB}
		D_{\xi}\varphi^{v}_{\mathbf{X}}[\zeta] = S_{v-u}D_{\xi}\varphi^{u}_{\mathbf{X}}[\zeta] + \int_{u}^{v}S_{v-\tau}D_{\varphi^{\tau}_{\mathbf{X}}(\xi)} F[D_{\xi}\varphi^{\tau}_{\mathbf{X}}[\zeta]] \, \mathrm{d}\tau + \int_{u}^{v}S_{v-\tau}D_{\varphi^{\tau}_{\mathbf{X}}(\xi)}G[D_{\xi}\varphi^{\tau}_{\mathbf{X}}[\zeta]]\circ\mathrm{d}\mathbf{X}_{\tau}.
	\end{align}
	Note that from \eqref{LLLL}, also $(D_{\xi}\varphi^{.}_{\mathbf{X}}[\zeta])^{\prime}(\tau)=D_{\varphi^{\tau}_{\mathbf{X}}(\xi)}G[D_{\xi}\varphi^{\tau}_{\mathbf{X}}[\zeta]]$ and
	\begin{align*}
		&(\delta D_{\varphi^{.}_{\mathbf{X}}(\xi)}G[D_{\xi}\varphi^{.}_{\mathbf{X}}[\zeta]])_{\tau,\nu} \\
		=\ &D_{\varphi^{\tau}_{\mathbf{X}}(\xi)}G[D_{\varphi^{\tau}_{\mathbf{X}}(\xi)}G[D_{\xi}\varphi^{\tau}_{\mathbf{X}}[\zeta]]\circ (\delta X)_{\tau,\nu}]+D^{2}_{\varphi^{\tau}_{\mathbf{X}}(\xi)}G[G(\varphi^{\tau}_{\mathbf{X}}(\xi))\circ (\delta X)_{\tau,\nu},D_{\xi}\varphi^{\tau}_{\mathbf{X}}[\zeta]] \\
		&\quad +  D_{\varphi^{\tau}_{\mathbf{X}}(\xi)}G\big[[D_{\xi}\varphi_{\mathbf{X}}^{.}[\zeta]]^{\#}_{\tau,\nu}\big]+D^{2}_{\varphi^{\tau}_{\mathbf{X}}(\xi)}G\big[[\varphi^{.}_{\mathbf{X}}(\xi)]^{\#}_{\tau,\nu},D_{\xi}\varphi^{\tau}_{\mathbf{X}}[\zeta]\big]\\
		&\quad +\int_{0}^{1}(1-\theta)D^{3}_{\theta \varphi^{\nu}_{\mathbf{X}}(\xi)+(1-\theta)\varphi^{\tau}_{\mathbf{X}}(\xi)}G[(\delta\varphi^{.}_{\mathbf{X}}(\xi))_{\tau,\nu},(\delta\varphi^{.}_{\mathbf{X}}(\xi))_{\tau,\nu},D_{\xi}\varphi^{\tau}_{\mathbf{X}}[\zeta]]\, \mathrm{d}\theta\\
		&\quad +\int_{0}^{1}D^{2}_{\theta \varphi^{\nu}_{\mathbf{X}}(\xi)+(1-\theta)\varphi^{\tau}_{\mathbf{X}}(\xi)} G[(\delta\varphi^{.}_{\mathbf{X}}(\xi))_{\tau,\nu},(\delta D_{\xi}\varphi^{.}_{\mathbf{X}}[\zeta])_{\tau,\nu}]\, \mathrm{d}\theta .
	\end{align*}
	Therefore,
	\begin{align}\label{DDEER}
		\begin{split}
			&(D_{\varphi^{\tau}_{\mathbf{X}}(\xi)}G[D_{\xi}\varphi^{\tau}_{\mathbf{X}}[\zeta]])^{\prime}\circ (\delta X)_{\tau,\nu} \\
			=\ &D_{\varphi^{\tau}_{\mathbf{X}}(\xi)}G[D_{\varphi^{\tau}_{\mathbf{X}}(\xi)}G[D_{\xi}\varphi^{\tau}_{\mathbf{X}}[\zeta]]\circ (\delta X)_{\tau,\nu}]+D^{2}_{\varphi^{\tau}_{\mathbf{X}}(\xi)}G[G(\varphi^{\tau}_{\mathbf{X}}(\xi))\circ (\delta X)_{\tau,\nu},D_{\xi}\varphi^{\tau}_{\mathbf{X}}[\zeta]]
		\end{split}
	\end{align}
	and
	\begin{align}\label{reminder_2222}
		\begin{split}
			&\big[D_{\varphi^{.}_{\mathbf{X}}(\xi)}G[D_{\xi}\varphi^{.}_{\mathbf{X}}[\zeta]]\big]^{\#}_{\tau,\nu} \\
			=\ &D_{\varphi^{\tau}_{\mathbf{X}}(\xi)}G\big[[D_{\xi}^{.}\varphi_{\mathbf{X}}[\zeta]]^{\#}_{\tau,\nu}\big] + D^{2}_{\varphi^{\tau}_{\mathbf{X}}(\xi)}G\big[[\varphi^{.}_{\mathbf{X}}(\xi)]^{\#}_{\tau,\nu},D_{\xi}\varphi^{\tau}_{\mathbf{X}}[\zeta]\big]\\
			&\quad +\int_{0}^{1}(1-\theta)D^{3}_{\theta \varphi^{\nu}_{\mathbf{X}}(\xi)+(1-\theta)\varphi^{\tau}_{\mathbf{X}}(\xi)}G[(\delta\varphi^{.}_{\mathbf{X}}(\xi))_{\tau,\nu},(\delta\varphi^{.}_{\mathbf{X}}(\xi))_{\tau,\nu},D_{\xi}\varphi^{\tau}_{\mathbf{X}}[\zeta]]\,\mathrm{d}\theta\\
			&\quad +\int_{0}^{1}D^{2}_{\theta \varphi^{\nu}_{\mathbf{X}}(\xi)+(1-\theta)\varphi^{\tau}_{\mathbf{X}}(\xi)} G[(\delta\varphi^{.}_{\mathbf{X}}(\xi))_{\tau,\nu},(\delta D_{\xi}\varphi^{.}_{\mathbf{X}}[\zeta])_{\tau,\nu}]\,\mathrm{d}\theta .
		\end{split}
	\end{align}
	From \eqref{DDEER},
	\begin{align*}
		&\sup_{\tau\in [s,t]}\vert (D_{\varphi^{\tau}_{\mathbf{X}}(\xi)}G[D_{\xi}\varphi^{\tau}_{\mathbf{X}}[\zeta]])^{\prime}\vert_{\alpha-\eta-\gamma}^{(n)}\lesssim \sup_{\tau\in[s,t]}\vert D_{\xi}\varphi^{\tau}_{\mathbf{X}}[\zeta]\vert_{\alpha-\gamma+\eta} \\
		&\lesssim\Vert(D_{\xi}\varphi^{.}_{\mathbf{X}}[\zeta],D_{\varphi^{\tau}_{\mathbf{X}}(\xi)}G[D_{\xi}\varphi^{.}_{\mathbf{X}}[\zeta]])\Vert_{\mathcal{D}_{\mathbf{X},{\alpha}}^{\gamma}([s,t])} .
	\end{align*}
	Also by \eqref{NMMMMMLY},
	\begin{align*}
		&\sup_{\substack{\tau,\nu\in [s,t],\\ \tau<\nu}}\frac{\vert (\delta D_{\varphi^{.}_{\mathbf{X}}(\xi)} G[D_{\xi}\varphi^{.}_{\mathbf{X}}[\zeta]])_{\tau,\nu}^{\prime}\vert_{\alpha-\eta-2\gamma}^{(n)}}{(\nu-\tau)^\gamma}\\
		\lesssim\ &(1+\Vert X\Vert_{\gamma,[s,t]}) \Vert(\varphi^{.}_{\mathbf{X}}(\xi),G(\varphi^{.}_{\mathbf{X}}(\xi)))\Vert_{\mathcal{D}_{\mathbf{X},{\alpha}}^{\gamma}([s,t])} \Vert (D_{\xi}\varphi^{.}_{\mathbf{X}}[\zeta],D_{\varphi^{\tau}_{\mathbf{X}}(\xi)}G[D_{\xi}\varphi^{.}_{\mathbf{X}}[\zeta]])\Vert_{\mathcal{D}_{\mathbf{X},{\alpha}}^{\gamma}([s,t])}.
	\end{align*}
	In addition from \eqref{NMMMMMLY} and \eqref{reminder_2222}  
	\begin{align*}
		&\Vert \big[D_{\varphi^{.}_{\mathbf{X}}(\xi)}G[D_{\xi}\varphi^{.}_{\mathbf{X}}[\zeta]]\big]^{\#}\Vert_{\mathcal{E}^{\gamma,2\gamma}_{{\alpha}-\eta;[s,t]}}\\& \quad\lesssim(1+\Vert X\Vert_{\gamma,[s,t]})^2 \big(1+\Vert(\varphi^{.}_{\mathbf{X}}(\xi),G(\varphi^{.}_{\mathbf{X}}(\xi)))\Vert_{\mathcal{D}_{\mathbf{X},{\alpha}}^{\gamma}([s,t])}^2\big) \Vert(D_{\xi}\varphi^{.}_{\mathbf{X}}[\zeta],D_{\varphi^{.}_{\mathbf{X}}(\xi)}G[D_{\xi}\varphi^{.}_{\mathbf{X}}[\zeta]])\Vert_{\mathcal{D}_{\mathbf{X},{\alpha}}^{\gamma}([s,t])}.
	\end{align*}
	Now from \eqref{BBBBBB} and \eqref{DDEER}, for $t-s<1$,
	\begin{align*}
		&\left\| \left( \int_{s}^{.}S_{.-\tau}D_{\varphi^{\tau}_{\mathbf{X}}(\xi)}G[D_{\xi}\varphi^{\tau}_{\mathbf{X}}[\zeta]]\circ\mathrm{d}\mathbf{X}_{\tau},D_{\varphi^{.}_{\mathbf{X}}(\xi)}G[D_{\xi}\varphi^{.}_{\mathbf{X}}[\zeta]] \right) \right\|_{\mathcal{D}_{\mathbf{X},{\alpha}}^{\gamma}([s,t])} \\
		\lesssim\ & \varrho_{\gamma}(\mathbf{X},[s,t])\vert D_{\xi}\varphi^{s}_{\mathbf{X}}[\zeta]\vert_{\alpha} + (t-s)^{\gamma-\eta}\varrho_{\gamma}^3(\mathbf{X},[s,t])\big(1+\Vert(\varphi^{.}_{\mathbf{X}}(\xi),G(\varphi^{.}_{\mathbf{X}}(\xi)))\Vert_{\mathcal{D}_{\mathbf{X},{\alpha}}^{\gamma}([s,t])}^2\big) \\
		&\qquad \times \Vert (D_{\xi}\varphi^{.}_{\mathbf{X}}[\zeta],D_{\varphi^{\tau}_{\mathbf{X}}(\xi)}G[D_{\xi}\varphi^{.}_{\mathbf{X}}[\zeta]])\Vert_{\mathcal{D}_{\mathbf{X},{\alpha}}^{\gamma}([s,t])}.
	\end{align*}
	Similar to \eqref{Drift_Estimate}, from \eqref{poly},
	\begin{align}\label{Drift_Estimate11}
		\begin{split}
			&\left| \int_{\tau}^{\nu}S_{\nu-x} D_{\varphi^{x}_{\mathbf{X}}(\xi)} F[D_{\xi}\varphi^{x}_{\mathbf{X}}[\zeta]] \, \mathrm{d}x \right|_{\alpha-i\gamma} \leq \int_{\tau}^{\nu} \vert S_{\nu-x}D_{\varphi^{x}_{\mathbf{X}}(\xi)} F[D_{\xi}\varphi^{x}_{\mathbf{X}}[\zeta]]\vert_{\alpha-i\gamma} \, \mathrm{d}x\\ &\quad\lesssim \sup_{x\in [\tau,\nu]}P_{1}(\vert \varphi^{x}_{\mathbf{X}}(\xi)\vert_{\alpha})\sup_{x\in [\tau,\nu]}\vert D_{\xi}\varphi^{x}_{\mathbf{X}}[\zeta]\vert_{\alpha}\int_{\tau}^{\nu}\max\lbrace(\nu-x)^{-\sigma+i\gamma},1\rbrace \, \mathrm{d}x\\ &\qquad\lesssim (\nu-\tau)^{\min\lbrace 1,1-\sigma+i\gamma\rbrace}\sup_{x\in [\tau,\nu]}P_{1}(\vert \varphi^{x}_{\mathbf{X}}(\xi)\vert_{\alpha})\sup_{x\in [\tau,\nu]}\vert D_{\xi}\varphi^{x}_{\mathbf{X}}[\zeta]\vert_{\alpha}.
		\end{split}
	\end{align}
	Consequently, from \eqref{CCBB}, \eqref{Drift_Estimate11}, for $t-s\leq1$ and a constant $M>1$,
	\begin{align}\label{LINEAR}
		\begin{split}
			&\Vert(D_{\xi}\varphi^{.}_{\mathbf{X}}[\zeta],D_{\varphi^{.}_{\mathbf{X}}(\xi)}G[D_{\xi}\varphi^{.}_{\mathbf{X}}[\zeta]])\Vert_{\mathcal{D}_{\mathbf{X},{\alpha}}^{\gamma}([s,t])}\leq M\bigg[ \varrho_{\gamma}(\mathbf{X},[s,t])\vert D_{\xi}\varphi^{s}_{\mathbf{X}}[\zeta]\vert_{\alpha}+\big[(t-s)^{1-\bar{\sigma}}\sup_{\tau\in [s,t]}P_{1}(\vert \varphi^{\tau}_{\mathbf{X}}(\xi)\vert_{\alpha}) \\&\quad +(t-s)^{\gamma-\eta}\varrho_{\gamma}^3(\mathbf{X},[s,t])\big(1+\Vert(\varphi^{.}_{\mathbf{X}}(\xi),G(\varphi^{.}_{\mathbf{X}}(\xi)))\Vert_{\mathcal{D}_{\mathbf{X},{\alpha}}^{\gamma}([s,t])}^2\big)\big]\Vert(D_{\xi}\varphi^{.}_{\mathbf{X}}[\zeta],D_{\varphi^{.}_{\mathbf{X}}(\xi)}G[D_{\xi}\varphi^{.}_{\mathbf{X}}[\zeta]])\Vert_{\mathcal{D}_{\mathbf{X},{\alpha}}^{\gamma}([s,t])}\bigg].
		\end{split}
	\end{align}
	Now we extend this estimate to larger intervals. Let us fix  $0<\epsilon<1$ and set $\tau_{0}=0$ and $\nu \coloneqq \min\lbrace 1-\bar{\sigma},\gamma-\eta\rbrace$. We also set
	\begin{align}\label{BNVRT}
		\tilde{\tau}^{\nu} \coloneqq \frac{1-\epsilon}{M\big(\sup_{\tau\in [0,T]}P_1(\vert \varphi^{\tau}_{\mathbf{X}}(\xi) \vert_{\alpha})+\varrho_{\gamma}^3(\mathbf{X},[0,T])\big(1+\Vert(\varphi^{.}_{\mathbf{X}}(\xi),G(\varphi^{.}_{\mathbf{X}}(\xi)))\Vert_{\mathcal{D}_{\mathbf{X},{\alpha}}^{\gamma}([0,T])}^2\big)\big)}
	\end{align}
	and $\tau_{n+1}:=\tau_{n}+\tilde{\tau}$. Then from \eqref{LINEAR} for $I_n=[\tau_n,\tau_{n+1}]\subset [0,T]$,
	\begin{align*}
		\Vert(D_{\xi}\varphi^{.}_{\mathbf{X}}[\zeta],D_{\varphi^{.}_{\mathbf{X}}(\xi)}G[D_{\xi}\varphi^{.}_{\mathbf{X}}[\zeta]])\Vert_{\mathcal{D}_{\mathbf{X},{\alpha}}^{\gamma}(I_{n})}\leq \frac{M\varrho_{\gamma}(\mathbf{X},[0,T])}{1-\epsilon}\vert D_{\xi}\varphi^{\tau_n}_{\mathbf{X}}[\zeta]\vert_{\alpha}.
	\end{align*}
	For $N \coloneqq \lfloor \frac{T}{\tilde{\tau}}\rfloor +1$, 
	\begin{align}\label{YHNMNMMAZ}
		\begin{split}
			\sup_{\tau\in [0,T]}\Vert D_{\xi}\varphi^{\tau}_{\mathbf{X}}\Vert_{\mathcal{L}(B_{\alpha},B_{\alpha})} &\leq \big(\frac{M\varrho_{\gamma}(\mathbf{X},[0,T])}{1-\epsilon}\big)^{N} \quad \text{and} \\
			\\ \sup_{0\leq n<N}\Vert(D_{\xi}\varphi^{.}_{\mathbf{X}}[\zeta],D_{\varphi^{.}_{\mathbf{X}}(\xi)}G[D_{\xi}\varphi^{.}_{\mathbf{X}}[\zeta]])\Vert_{\mathcal{D}_{\mathbf{X},{\alpha}}^{\gamma}(I_{n})} &\leq \big(\frac{M\varrho_{\gamma}(\mathbf{X},[0,T])}{1-\epsilon}\big)^{N}\vert\zeta\vert_{\alpha} .
		\end{split}
	\end{align}
	Note that for every $Z\in \mathcal{D}_{\mathbf{X},{\alpha}}^{\gamma}([0,T])$ and $s,u,t\in [0,T]$ with $s<u<t$, 
	\begin{align*}
		Z^{\#}_{s,t} = Z^{\#}_{s,u} + Z^{\#}_{u,t} + Z^{\prime}_{s,u}(\delta X)_{u,t}.
	\end{align*}
	Therefore for a constant $\tilde{M}>1$,
	\begin{align}\label{total_norm_1}
		\begin{split}
			&\Vert(D_{\xi}\varphi^{.}_{\mathbf{X}}[\zeta],D_{\varphi^{.}_{\mathbf{X}}(\xi)}G[D_{\xi}\varphi^{.}_{\mathbf{X}}[\zeta]])\Vert_{\mathcal{D}_{\mathbf{X},{\alpha}}^{\gamma}([0,T])}\\&\quad \leq\tilde{M}(1+\Vert X\Vert_{\gamma,[0,T]})\sum_{0\leq n\leq N}\Vert(D_{\xi}\varphi^{.}_{\mathbf{X}}[\zeta],D_{\varphi^{.}_{\mathbf{X}}(\xi)}G[D_{\xi}\varphi^{.}_{\mathbf{X}}[\zeta]])\Vert_{\mathcal{D}_{\mathbf{X},{\alpha}}^{\gamma}(I_n)}.
		\end{split}
	\end{align}
	Recall $N \coloneqq \lfloor \frac{T}{\tilde{\tau}}\rfloor +1$. Consequently from \eqref{BNVRT} and\eqref{YHNMNMMAZ}
	\begin{align*}
		\Vert( D_{\xi}\varphi^{.}_{\mathbf{X}}[\zeta],D_{\varphi^{.}_{\mathbf{X}}(\xi)}G[D_{\xi}\varphi^{.}_{\mathbf{X}}[\zeta]])\Vert_{\mathcal{D}_{\mathbf{X},{\alpha}}^{\gamma}([0,T])}\leq M_{\epsilon}\exp\big(M_{\epsilon} S(s,t,\xi,\mathbf{X})\big) \vert\zeta\vert_{\alpha}
	\end{align*}
	where $M_{\epsilon}>0$ and
	\begin{align*}
		&S(0,T,\xi,\mathbf{X}) = \log(\varrho_{\gamma}(\mathbf{X},[0,T]))\\
		&\quad \times \left(\left[\sup_{\tau\in [0,T]}P_1(\vert \varphi^{\tau}_{\mathbf{X}}(\xi) \vert_{\alpha})+\varrho_{\gamma}^3(\mathbf{X},[0,T])\big(1+\Vert(\varphi^{.}_{\mathbf{X}}(\xi),G(\varphi^{.}_{\mathbf{X}}(\xi)))\Vert_{\mathcal{D}_{\mathbf{X},{\alpha}}^{\gamma}([0,T])}^2\big)\right]^{\frac{1}{\min\lbrace 1-\bar{\sigma},\gamma-\eta\rbrace}}+1\right).
	\end{align*}
\end{proof}

\begin{corollary}\label{CC_DDE}
	Assume the same setting and notation as in Proposition \ref{linear boundd}. Let $\xi,\tilde{\xi}\in\mathcal{B}_{\alpha}$. Then there exists an $M > 0$ such that 
	\begin{align}\label{bounde_derivative_2}
		\big\Vert \big(\varphi^{.}_{\mathbf{X}}(\xi)-\varphi^{.}_{\mathbf{X}}(\tilde\xi),G(\varphi^{.}_{\mathbf{X}}(\xi))-G(\varphi^{.}_{\mathbf{X}}(\tilde\xi))\big)\big\Vert_{\mathcal{D}_{\mathbf{X},{\alpha}}^{\gamma}([0,T])}\leq M \vert \xi-\tilde{\xi}\vert_{\alpha}\exp\big(M \tilde{S}(0,T,\xi,\tilde\xi,\mathbf{X})\big)
	\end{align}
	where for $\rho(\xi,\tilde{\xi}) \coloneqq \rho\xi+(1-\rho)\tilde{\xi}$,
	\begin{align}\label{SSssSS}
		\begin{split}
			&\tilde{S}(0,T,\xi,\tilde\xi,\mathbf{X})=\log(\varrho_{\gamma}(\mathbf{X},[0,T]))\times\\& \sup_{\rho\in [0,1]}\left(1+\left[\sup_{\tau\in [0,T]}P_1(\vert \varphi^{\tau}_{\mathbf{X}}(\rho(\xi,\tilde{\xi})) \vert_{\alpha})+\varrho_{\gamma}^3(\mathbf{X},[0,T])\big(1+\Vert(\varphi^{.}_{\mathbf{X}}(\rho(\xi,\tilde{\xi})),G(\varphi^{.}_{\mathbf{X}}(\rho(\xi,\tilde{\xi}))))\Vert_{\mathcal{D}_{\mathbf{X},{\alpha}}^{\gamma}([0,T])}^2\big)\right]^{\frac{1}{\min\lbrace 1-\bar{\sigma},\gamma-\eta\rbrace}}\right)
		\end{split}
	\end{align}
\end{corollary}

\begin{proof}
	Follows from the fact that $\varphi^{.}_{\mathbf{X}}(\xi)-\varphi^{.}_{\mathbf{X}}(\tilde\xi)=\int_{0}^{1} D_{\theta\xi+(1-\theta)\tilde{\xi}}\varphi^{.}_{\mathbf{X}}[\xi-\tilde{\xi}]\, \mathrm{d}\theta$.
\end{proof}

\begin{remark}\label{COMPACTEM_12}
	Note that in equation \eqref{LLL},
	\begin{align*}
		D_{\varphi^{.}_{\mathbf{X}}(\xi)}F[D_{\xi}\varphi^{.}_{\mathbf{X}}[\zeta]]\in \mathcal{B}_{\alpha-\sigma} \quad \text{and} \quad D_{\varphi^{.}_{\mathbf{X}}(\xi)}G[D_{\xi}\varphi^{.}_{\mathbf{X}}[\zeta]]\in \mathcal{B}_{\alpha-\eta}^{n}.
	\end{align*} 
	However, from \eqref{SEM_I} and \cite[Theorem 4.5]{GH19} for $0< \epsilon<\min\lbrace 1-\sigma ,\gamma-\eta\rbrace$,
	\begin{align}
		\begin{split}
			\int_{s}^{t}S_{t-\tau}D_{\varphi^{\tau}_{\mathbf{X}}(\xi)}F[D_{\xi}\varphi^{\tau}_{\mathbf{X}}[\zeta]]\mathrm{d}\tau, \  \int_{s}^{t}S_{t-\tau}D_{\varphi^{\tau}_{\mathbf{X}}(\xi)}G[D_{\xi}\varphi^{\tau}_{\mathbf{X}}[\zeta]]\circ\mathrm{d}\mathbf{X}_{\tau}\in \mathcal{B}_{\alpha+\epsilon} .
		\end{split}
	\end{align}
	Let us fix $t>0$. Then this simple observation gives the continuity of the map
	\begin{align}\label{COMPACTEM}
		\begin{split}
			& D_{\xi}\varphi^{t}_{\mathbf{X}}\colon \mathcal{B}_{\alpha}\rightarrow \mathcal{B}_{\alpha+\epsilon}, \quad \zeta \mapsto D_{\xi}\varphi^{t}_{\mathbf{X}}[\zeta] .
		\end{split}
	\end{align}
	In particular, if the embedding $\operatorname{id} \colon \mathcal{B}_{\alpha+\epsilon}\rightarrow \mathcal{B}_{\alpha}$ is compact, the linear map 
	\begin{align*}
		D_{\xi}\varphi^{t}_{\mathbf{X}} \colon \mathcal{B}_{\alpha}\rightarrow \mathcal{B}_{\alpha}, \quad  \quad \zeta \mapsto D_{\xi}\varphi^{t}_{\mathbf{X}}[\zeta]
	\end{align*}
	is compact. This observation will be important when we apply the multiplicative ergodic theorem.
\end{remark}

We are now ready to formulate our main result about integrability of the linearized equation.

\begin{theorem}\label{thm:integrabiliy_linearization}
	Let $\mathbf{X} = (X,\mathbb{X})$ be the $\gamma$-H\"older rough path lift of a Gaussian process, $\frac{1}{3} < \gamma \leq \frac{1}{2}$, defined on an abstract Wiener space $(\mathcal{W},\mathcal{H},\mu)$ for which Condition \eqref{variation} and Assumption \ref{Cameron-Martin} hold. Assume that the conditions of Proposition \ref{linear boundd} are satisfied. Let $\xi$ be a random variable in $\mathcal{B}_{\alpha}$ with the property that
	\begin{align}\label{eqn:integrab_xi}
		|\xi(\omega)|_{\alpha} \in \bigcap_{p \geq 1} \mathcal{L}^p(\mathcal{W}).
	\end{align}
	Then it holds that
	\begin{align*}
		\sup_{t \in [0,T]} \log^{+}\big(\Vert D_{\xi} \varphi^t_{\mathbf{X}} \Vert_{\mathcal{L}(\mathcal{B}_{\alpha},\mathcal{B}_{\alpha})}\big)\in \bigcap_{p \geq 1} \mathcal{L}^{p}(\mathcal{W}).
	\end{align*}
\end{theorem}

\begin{proof}
	Let $t \in [0,T]$. From the bound \eqref{bounde_derivative} in Proposition \ref{linear boundd},
	\begin{align}\label{HHHBBB}
		\begin{split}
			&\log^{+}\big(\Vert D_{\xi} \varphi^t_{\mathbf{X}} \Vert_{\mathcal{L}(\mathcal{B}_{\alpha},\mathcal{B}_{\alpha})}\big)\leq \log(M)+ M\log(1+\Vert X\Vert_{\gamma,[0,T]})\\
			&\quad \times \big[\sup_{\tau\in [0,T]}P_1(\vert \varphi^{\tau}_{\mathbf{\omega}}(\xi) \vert_{\alpha})+\varrho_{\gamma}^3(\mathbf{X}(\omega),[0,T])\big(1+\Vert(\varphi^{.}_{\omega}(\xi),G(\varphi^{.}_{\omega}(\xi)))\Vert_{\mathcal{D}_{\mathbf{X}(\omega),{\alpha}}^{\gamma}([0,T])}^2\big)\big]^{\frac{1}{\min{1-\sigma,\gamma-\eta}}}.
		\end{split}
	\end{align}
	In addition, from the bound \eqref{integrable bound_2} in Theorem \ref{thm:integrability_RPDE},
	\begin{align*}
		&\Vert(\varphi^{.}_{\omega}(\xi),G(\varphi^{.}_{\omega}(\xi)))\Vert_{\mathcal{D}_{\mathbf{X},{\alpha}}^{\gamma}([0,T])}\\&\quad \leq \tilde{M}N([0,T],\eta_1,\chi,\mathbf{X}(\omega))(1+\Vert X(\omega)\Vert_{\gamma,[0,T]})\bigg[\exp\big{(}N([0,T],\eta_1,\chi,\mathbf{X}(\omega))\tilde{M}_\epsilon\big{)}|\xi|_{\alpha} \\
		&\qquad +\frac{\exp\big{(}N([0,T],\eta_1,\chi,\mathbf{X}(\omega))\tilde{M}_\epsilon+\tilde{M}_\epsilon\big{)}-1}{2M_\epsilon-1}P(\Vert X(\omega)\Vert_{\gamma,[0,T},\Vert\mathbb{X}(\omega)\Vert_{\gamma,[0,T]})\bigg].
	\end{align*}
	Consequently, from \eqref{integrability_2} and \eqref{eqn:integrab_xi},
	\begin{align*}
		\sup_{t \in [0,T]} \log^{+}\big(\Vert  D_{\xi} \varphi^t_{\mathbf{X}} \Vert_{\mathcal{L}(\mathcal{B}_{\alpha},\mathcal{B}_{\alpha})}\big)\in\mathcal{L}^{p}(\Omega)
	\end{align*}
	for every $p \geq 1$.
\end{proof}

We will need some further technical estimates we are going to proof now. Before stating the next proposition, we need an auxiliary lemma which is a slight generalization of the estimate \eqref{BBBBBB}.
\begin{lemma}\label{parmeter}
	Assume $\mathbf{X}=(X,\mathbb{X})\in\mathscr{C}^{\gamma}$ and  $Z:[s,t]\rightarrow \mathcal{B}_{\alpha}$ be a path  such that $(\delta Z)_{\tau,\nu}=Z^{\prime}_{\tau}\circ (\delta X)_{\tau,\nu}+Z^{\#}_{\tau,\nu}$. We say $(Z,Z^\prime)\in \mathcal{D}_{\mathbf{X},\alpha-\eta}^{\gamma_{1},\gamma_2}([s,t])$ if
	\begin{align*}
		\Vert (Z,Z^{\prime})\Vert_{\mathcal{D}_{\mathbf{X},\alpha-\eta}^{\gamma_{1},\gamma_2}([s,t])} &\coloneqq \sup_{\tau\in [s,t]}\vert Z_{\tau}\vert_{{\alpha}-\eta} + \max \left\{ \sup_{\tau\in [s,t]}\vert Z^{\prime}_{\tau}\vert_{{\alpha}-\eta-\gamma}^{(n)},\sup_{\substack{\tau,\nu\in [s,t],\\ \tau<\nu}}\frac{\vert (\delta Z^{\prime})_{\tau,\nu}\vert_{{\alpha}-\eta-2\gamma}^{(n)}}{(\nu-\tau)^{\gamma_1}}\right\}\\ &\quad + \max\left\{ \sup_{\substack{\tau,\nu\in [s,t],\\ \tau<\nu}}\frac{\vert Z^{\#}_{\tau,\nu}\vert_{{\alpha}-\eta-\gamma}}{(\nu-\tau)^{\gamma_2}},\sup_{\substack{\tau,\nu\in [s,t],\\ \tau<\nu}}\frac{\vert Z^{\#}_{\tau,\nu}\vert_{{\alpha}-\eta-2\gamma}}{(\nu-\tau)^{2\gamma_2}}\right\}<\infty.
	\end{align*}
	Assume $\gamma_1+2\gamma>1$, $2\gamma_2+\gamma>1$ and $0\leq \eta<\min\lbrace 2\gamma_2-\gamma,\gamma_1,\gamma,\frac{\gamma(2\gamma_2-\gamma)}{\gamma_2}\rbrace$. Then the linear map
	\begin{align*}
		(\mathcal{D}_{\mathbf{X},\alpha-\eta}^{\gamma_{1},\gamma_2}([s,t]))^n &\to \mathcal{D}_{\mathbf{X},{\alpha}}^{\gamma}([s,t]) \\
		(Z,Z^{\prime}) &\mapsto \big( \int_{s}^{.}S_{.-\tau}Z_{\tau}\circ\mathrm{d}\mathbf{X}_{\tau},Z\big),
	\end{align*}
	is well defined. Similar to \eqref{SEW}, the integral is defined by
	\begin{align}\label{SEW_1}
		\int_{s}^{u}S_{u-\tau}Z_{\tau}\circ\mathrm{d}\mathbf{X}_{\tau} \coloneqq \lim_{\substack{|\pi|\rightarrow 0,\\ \pi=\lbrace \tau_0=s,\tau_{1},...\tau_{m}=u \rbrace}}\sum_{0\leq j<m}\big{[}S_{u-\tau_j}Z_{\tau_j}\circ (\delta X)_{\tau_j,\tau_{j+1}}+S_{u-\tau_j}Z^{\prime}_{\tau_j}\circ\mathbb{X}_{\tau_j,\tau_{j+1}}\big{]}.
	\end{align}
	In addition, if $t-s<1$ then for $i\in\lbrace 0,1,2\rbrace$, and $\mu_1=\min\lbrace \gamma_1-\eta,\gamma-\eta,2\gamma_2-\gamma-\eta\rbrace$
	\begin{align}\label{BBBBasasBB1}
		\begin{split}
			&\left|\int_{s}^{t}S_{t-\tau}Z_{\tau}\circ\mathrm{d}\mathbf{X}_{\tau}-S_{t-s}Z_{s}\circ (\delta X)_{s,u}-S_{t-s}Z_{s}^{\prime}\circ \mathbb{X}_{s,u}\right|_{\alpha-i\gamma}\leq C_{\gamma_1,\gamma_2,\eta}\varrho_{\gamma}(\mathbf{X},[s,t]) (t-s)^{\gamma i+\mu_1} \Vert (Z,Z^\prime)\Vert_{\mathcal{D}_{\mathbf{X},\alpha-\eta}^{\gamma_{1},\gamma_2}([s,t])}
		\end{split}
	\end{align}
	and
	\begin{align}\label{BBBBasasBB}
		\begin{split}
			&\left\| \big( \int_{s}^{.}S_{.-\tau}Z_{\tau}\circ\mathrm{d}\mathbf{X}_{\tau},Z\big) \right\|_{\mathcal{D}_{\mathbf{X},{\alpha}}^{\gamma}([s,t])}\\&\quad\leq C_{\gamma_1,\gamma_2,\eta}\bigg(\vert Z_{s}\vert_{\alpha-\eta}^{(n)}\vert + \varrho_{\gamma}(\mathbf{X},[s,t])\vert Z_{s}^{\prime}\vert_{\alpha-\eta-\gamma}^{(n\times n)}+ \varrho_{\gamma}(\mathbf{X},[s,t])(t-s)^{\mu_2}\Vert (Z,Z^{\prime})\Vert_{(\mathcal{D}_{\mathbf{X},{\alpha-\eta}}^{\gamma}(I))^n}\bigg),
		\end{split}
	\end{align}
	where $\mu_2=\min\lbrace \gamma_1-\eta,\gamma-\eta,2\gamma_2-\gamma-\eta,\frac{2\gamma_2\gamma-\eta\gamma_2-\gamma^2}{\gamma },\frac{\min\lbrace \gamma,\gamma_2,\gamma_1\rbrace(\gamma-\eta)}{\gamma }\rbrace$.
\end{lemma}
\begin{proof}
	The proof is standard. For $\tau<\nu$ with $\tau,\nu\in [s,u]$, set 
	\begin{align*}
		\Xi^{\tau,\nu}_{s,u} \coloneqq S_{u-\tau}Z_{\tau}\circ (\delta X)_{\tau,\nu}+S_{u-\tau}Z_{\tau}^{\prime}\circ \mathbb{X}_{\tau,\nu}.
	\end{align*} 
	Then for $\tau<\upsilon<\nu$,
	\begin{align*}
		\Xi^{\tau,\upsilon}_{s,u}+\Xi^{\upsilon,\nu}_{s,u}-\Xi^{\tau,\nu}_{s,u} &= S_{u-\tau}Z^{\#}_{\tau,\upsilon}\circ(\delta X)_{\upsilon,\nu}+(S_{u-\upsilon}-S_{u-\tau})Z_{\upsilon}\circ (\delta X)_{\upsilon,\nu}\\
		&\quad+ (S_{u-\upsilon}-S_{u-\tau})Z_{\upsilon}^{\prime}\circ \mathbb{X}_{\upsilon,\nu}+S_{u-\tau}(\delta Y^{\prime})_{\tau,\upsilon}\circ \mathbb{X}_{\upsilon,\nu}.
	\end{align*}
	Set $\tau^{n}_{m} \coloneqq s +\frac{n}{2^m}(u-s)$, where $0\leq n<2^m-1$. Also define $\tilde\Xi^{n,m}_{s,u} \coloneqq \Xi^{\tau^{n}_{m},\tau^{n+1}_{m}}_{s,u}$. Then
	\begin{align*}
		&\tilde\Xi^{2n,m+1}_{s,u}+\tilde\Xi^{2n+1,m+1}_{s,u}-\tilde\Xi^{n,m}_{s,u} \\
		=\ &S_{u-\tau^{2n}_{m+1}}Z^{\#}_{\tau^{2n}_{m+1},\tau^{2n+1}_{m+1}}\circ(\delta X)_{\tau^{2n+1}_{m+1},\tau^{2n+2}_{m+1}} + (S_{u-\tau^{2n+1}_{m+1}}-S_{u-\tau^{2n}_{m+1}})Z_{\tau^{2n+1}_{m+1}}\circ(\delta X)_{\tau^{2n+1}_{m+1},\tau^{2n+2}_{m+1}} \\
		&\quad +(S_{u-\tau^{2n+1}_{m+1}}-S_{u-\tau^{2n}_{m+1}})Z_{\tau^{2n+1}_{m+1}}^{\prime}\circ\mathbb{X}_{\tau^{2n+1}_{m+1},\tau^{2n+2}_{m+1}} + S_{u-\tau^{2n}_{m+1}}(\delta Z^{\prime})_{\tau^{2n}_{m+1},\tau^{2n+1}_{m+1}}\circ \mathbb{X}_{\tau^{2n+1}_{m+1},\tau^{2n+2}_{m+1}}.
	\end{align*}
	Also,
	\begin{align*}
		&\left|\int_{s}^{u}S_{u-\tau}Z_{\tau}\circ\mathrm{d}\mathbf{X}_{\tau}-S_{u-s}Z_{s}\circ (\delta X)_{s,u}-S_{u-s}Z_{s}^{\prime}\circ \mathbb{X}_{s,u}\right|_{\alpha-i\gamma} \\
		\leq\ &\sum_{m\geq 0}\sum_{0\leq n<2^{m}} \left| S_{u-\tau^{2n}_{m+1}}\left(Z^{\#}_{\tau^{2n}_{m+1},\tau^{2n+1}_{m+1}}\circ(\delta X)_{\tau^{2n+1}_{m+1},\tau^{2n+2}_{m+1}}\right)\right|_{\alpha-i\gamma} \\
		&\quad + \left| (S_{u-\tau^{2n+1}_{m+1}}-S_{u-\tau^{2n}_{m+1}})\left(Z_{\tau^{2n+1}_{m+1}}\circ(\delta X)_{\tau^{2n+1}_{m+1},\tau^{2n+2}_{m+1}}\right)\right|_{\alpha-i\gamma} \\
		&\quad +\left| (S_{u-\tau^{2n+1}_{m+1}}-S_{u-\tau^{2n}_{m+1}})\left(Z_{\tau^{2n+1}_{m+1}}^{\prime}\circ\mathbb{X}_{\tau^{2n+1}_{m+1},\tau^{2n+2}_{m+1}}\right)\right|_{\alpha-i\gamma} \\
		&\quad +\left| S_{u-\tau^{2n}_{m+1}}\left((\delta Z^{\prime})_{\tau^{2n}_{m+1},\tau^{2n+1}_{m+1}}\circ \mathbb{X}_{\tau^{2n+1}_{m+1},\tau^{2n+2}_{m+1}}\right)\right|_{\alpha-i\gamma}.
	\end{align*}
	We only focus on
	\begin{align*}
		\sum_{m\geq 0}\sum_{0\leq n<2^{m}}\left\vert S_{u-\tau^{2n}_{m+1}}\left((\delta Z^{\prime})_{\tau^{2n}_{m+1},\tau^{2n+1}_{m+1}}\circ \mathbb{X}_{\tau^{2n+1}_{m+1},\tau^{2n+2}_{m+1}}\right)\right\vert_{\alpha-i\gamma},
	\end{align*} 
	the other terms can be treated similarly. Since
	\begin{align*}
		&\sum_{m\geq 0}\sum_{0\leq n<2^{m}}\left\vert S_{u-\tau^{2n}_{m+1}}\left((\delta Z^{\prime})_{\tau^{n}_{m},\tau^{2n+1}_{m+1}}\circ \mathbb{X}_{\tau^{2n+1}_{m+1},\tau^{2n+2}_{m+1}}\right)\right\vert_{\alpha-i\gamma} \\
		\leq\ &\Vert Z\Vert_{\mathcal{D}_{\mathbf{X},\alpha-\eta}^{\gamma_{1},\gamma_2}([s,t])}\sum_{m\geq 1}\sum_{0\leq n<2^{m}}(u-\tau^{n}_{m})^{\gamma(i-2)-\eta}(\frac{1}{2^m})^{\gamma_{1}+2\gamma} \\
		\lesssim\ &(u-s)^{\gamma i+\gamma_1-\eta}\Vert Z\Vert_{\mathcal{D}_{\mathbf{X},\alpha-\eta}^{\gamma_{1},\gamma_2}([s,t])},
	\end{align*}
	the claimed bound follows. To prove the \eqref{BBBBasasBB}, first note that $(Z,Z^\prime)\in \mathcal{D}_{\mathbf{X},\alpha-\eta}^{\gamma_{1},\gamma_2}([s,t])$ and $t-s\leq 1$. From the interpolation property and  decomposition $(\delta Z)_{u,v}=Z^{\prime}_{u}\circ (\delta X)_{u,v}+Z^{\#}_{u,v}$, for $[u,v]\subseteq [s,t]$
	\begin{itemize}
		\item [(i)] $\vert Z^{\#}_{u,v}\vert_{\alpha-2\gamma}\lesssim \vert Z^{\#}_{u,v}\vert_{\alpha-\eta-\gamma}^{\frac{\eta}{\gamma}}\vert Z^{\#}_{u,v}\vert_{\alpha-\eta-2\gamma}^{\frac{\gamma-\eta}{\gamma}}\lesssim \Vert Z\Vert_{\mathcal{D}_{\mathbf{X},\alpha-\eta}^{\gamma_{1},\gamma_2}([s,t])}(u-v)^{\frac{2\gamma_2\gamma-\eta\gamma_2}{\gamma }}$.
		\item [(ii)] $\vert Z^{\prime}_{u,v}\vert_{\alpha-2\gamma}^{(n)}\lesssim  \big(\vert Z^{\prime}_{u,v}\vert_{\alpha-\eta-\gamma}^{(n)}\big)^{\frac{\eta}{\gamma}}\big(\vert Z^{\prime}_{u,v}\vert_{\alpha-\eta-2\gamma}^{(n)}\big)^{\frac{\gamma-\eta}{\gamma}}\lesssim \Vert Z\Vert_{\mathcal{D}_{\mathbf{X},\alpha-\eta}^{\gamma_{1},\gamma_2}([s,t])}(u-v)^{\frac{\gamma_1(\gamma-\eta)}{\gamma }}.$
		\item [(iii)] $\vert Z_{u,v}\vert_{\alpha-\gamma}\lesssim \vert Z_{u,v}\vert_{\alpha-\eta}^{\frac{\eta}{\gamma}}\vert Z_{u,v}\vert_{\alpha-\eta-\gamma}^{\frac{\gamma-\eta}{\gamma}}\lesssim \Vert Z\Vert_{\mathcal{D}_{\mathbf{X},\alpha-\eta}^{\gamma_{1},\gamma_2}([s,t])}^{\frac{\eta}{\gamma}}\vert Z_{u,v}\vert_{\alpha-\eta-\gamma}^{\frac{\gamma-\eta}{\gamma}} $
		\item [(iv)] $\vert Z_{u,v}\vert_{\alpha-\gamma-\eta}\leq (1+\Vert X\Vert_{\gamma,[s,t]})\Vert Z\Vert_{\mathcal{D}_{\mathbf{X},\alpha-\eta}^{\gamma_{1},\gamma_2}([s,t])}(v-u)^{\min\lbrace \gamma,\gamma_2\rbrace}$
	\end{itemize}
	From item (ii), 
	\begin{align}\label{MKLOPA}
		\sup_{\tau\in [s,t]}\vert Z^{\prime}_{\tau}\vert_{{\alpha}-2\gamma}^{(n)}-\vert Z^{\prime}_{s}\vert_{{\alpha}-2\gamma}^{(n)}\lesssim  \Vert Z\Vert_{\mathcal{D}_{\mathbf{X},\alpha-\eta}^{\gamma_{1},\gamma_2}([s,t])}(t-s)^{{\frac{\gamma_1(\gamma-\eta)}{\gamma }}}.
	\end{align}	
	Remember, $(\delta Z)_{u,v}=Z^{\prime}_{u}\circ (\delta X)_{u,v}+Z^{\#}_{u,v}$, $\frac{2\gamma_2\gamma-\eta\gamma_2}{\gamma }>\gamma$ and $t-s\leq 1$. So, from item (i), item (ii), and
	\eqref{MKLOPA}, for a constant $C$
	\begin{align}\label{MKL98}
		\begin{split}
			&\sup_{\substack{\tau,\nu\in [s,t],\\ \tau<\nu}}\frac{\vert (\delta Z)_{\tau,\nu}\vert_{{\alpha}-2\gamma}}{(\nu-\tau)^{\gamma}}\\&\quad\leq \vert Z^{\prime}_{s}\vert_{{\alpha}-2\gamma}^{(n)}\Vert X\Vert_{\gamma,[s,t]}+C\varrho_{\gamma}(\mathbf{X},[s,t]) \Vert Z\Vert_{\mathcal{D}_{\mathbf{X},\alpha-\eta}^{\gamma_{1},\gamma_2}([s,t])}(t-s)^{\min\lbrace{\frac{\gamma_1(\gamma-\eta)}{\gamma },\frac{2\gamma_2\gamma-\eta\gamma_2-\gamma^2}{\gamma }\rbrace}}
		\end{split}
	\end{align}
	Also, $\vert Z_{u,v}\vert_{\alpha-\gamma}\lesssim \vert Z_{u,v}\vert_{\alpha-\eta}^{\frac{\eta}{\gamma}}\vert Z_{u,v}\vert_{\alpha-\gamma-\eta}^{\frac{\gamma-\eta}{\gamma}}$. Consequently from item (iv)
	\begin{align}\label{A4521}
		\sup_{\tau\in [s,t]}\vert Z_{\tau}\vert_{{\alpha}-\gamma}-\vert Z_{s}\vert_{{\alpha}-\gamma}\lesssim (1+\Vert X\Vert_{\gamma,[s,t]})^{\frac{\gamma-\eta}{\gamma}} \Vert Z\Vert_{\mathcal{D}_{\mathbf{X},\alpha-\eta}^{\gamma_{1},\gamma_2}([s,t])}(t-s)^{{\frac{\min\lbrace \gamma,\gamma_2\rbrace(\gamma-\eta)}{\gamma }}}.
	\end{align}
	Also $\vert Z_{s}\vert_{{\alpha}-\gamma}\lesssim \vert Z_{s}\vert_{{\alpha}-\eta}$ and $\vert Z^{\prime}_{s}\vert_{{\alpha}-2\gamma}^{(n)}\lesssim \vert Z^{\prime}_{s}\vert_{{\alpha}-\eta-\gamma}^{(n)}$. The inequality \eqref{BBBBasasBB} follows from \eqref{BBBBasasBB1} ,\eqref{MKL98} and \eqref{A4521}. Indeed, since the rest of the proof is similar to the proof of \cite [Corollary 4.6]{GHT21}, we omit the details.
\end{proof}

\begin{proposition}\label{MMMKKKKKKTTTT}
	In addition to our assumptions in Proposition \eqref{linear boundd}, assume for $0<r\leq 1$ and $\xi,\tilde{\xi}\in\mathcal{B}_{\alpha}$ that
	\begin{align}\label{key}
		\begin{split}
			\Vert D_{\xi}F-D_{\tilde{\xi}}F\Vert_{\mathcal{L}(\mathcal{B}_\alpha,\mathcal{B}_{\alpha-\sigma})} &\leq P_2(\vert\xi\vert_{\alpha},\vert\tilde{\xi}\vert_{\alpha})\vert \xi-\tilde{\xi}\vert_{\alpha}^r \quad \text{and}\\
			\Vert  D^{3}_{\xi}G-D^{3}_{\tilde{\xi}}G\Vert_{\mathcal{L}(\mathcal{B}_\alpha,\mathcal{B}_{\alpha-2\gamma-\eta})} &\leq Q_2(\vert\xi\vert_{\alpha},\vert\tilde{\xi}\vert_{\alpha})\vert \xi-\tilde{\xi}\vert_{\alpha}^r,
		\end{split}
	\end{align}
	where $P_2,Q_2$ are two polynomials. Then for every $0<\epsilon<1$, there exists a constant $E_{\epsilon}$ such that 
	\begin{align}\label{bounde_derivative_3}
		\begin{split}
			&\big\Vert \big(D_{\xi}\varphi^{.}_{\mathbf{X}}[\zeta]-D_{\tilde{\xi}}\varphi^{.}_{\mathbf{X}}[\zeta],D_{\varphi^{.}_{\mathbf{X}}(\xi)}G[D_{\xi}\varphi^{\tau}_{\mathbf{X}}[\zeta]]-D_{\varphi^{.}_{\mathbf{X}}(\tilde\xi)}G[D_{\tilde\xi}\varphi^{\tau}_{\mathbf{X}}[\zeta]]\big)\big\Vert_{\mathcal{D}_{\mathbf{X},{\alpha}}^{\gamma}([0,T])}\\
			\leq\ &E_{\epsilon}\vert\zeta\vert_{\alpha}\max\lbrace \vert\xi-\tilde{\xi}\vert_{\alpha}, \vert\xi-\tilde{\xi}\vert_{\alpha}^{1-\kappa}, \vert\xi-\tilde{\xi}\vert_{\alpha}^{\frac{r}{2}}\rbrace\\
			&\quad \times \exp\bigg(E_{\epsilon}\big[R_1\big(\sup_{0\leq \rho \leq 1}\Vert(\varphi^{.}_{\mathbf{X}}(\rho\xi+(1-\rho)\tilde{\xi}),G(\varphi^{.}_{\mathbf{X}}(\rho\xi+(1-\rho)\tilde{\xi})))\Vert_{\mathcal{D}_{\mathbf{X},{\alpha}}^{\gamma}([0,T])}\big)+R_{2}(\Vert X\Vert_{\gamma,[0,T]},\Vert \mathbb{X}\Vert_{2\gamma,[0,T]})\big]\bigg).
		\end{split}
	\end{align}
	where $R_1$ and $R_2$ are two increasing polynomials and $\eta$ satisfies the inequality 
	\begin{align*}
		\max \left\{ \frac{\eta}{\gamma},\frac{1}{\gamma}-2 \right\} <\kappa<1.
	\end{align*}
\end{proposition}

\begin{proof}
	To simplify the notation and since the Gubinelli derivatives are clear, we will write $\Vert Y\Vert_{\mathcal{D}_{\mathbf{X},{\alpha}}^{\gamma}([0,T])}$ instead of $\Vert (Y,Y^{\prime})\Vert_{\mathcal{D}_{\mathbf{X},{\alpha}}^{\gamma}([0,T])}$ during the proof. For $s<t$,
	\begin{align}\label{-2-2-2}
		\begin{split}
			&D_{\xi}\varphi^{t}_{\mathbf{X}}[\zeta]-D_{\tilde{\xi}}\varphi^{t}_{\mathbf{X}}[\zeta]\\
			=\ & S_{t-s}(D_{\xi}\varphi^{s}_{\mathbf{X}}[\zeta]-D_{\tilde\xi}\varphi^{s}_{\mathbf{X}}[\zeta]) + \int_{s}^{t}S_{t-\tau}\big(D_{\varphi^{\tau}_{\mathbf{X}}(\xi)}F[D_{\xi}\varphi^{\tau}_{\mathbf{X}}[\zeta]]-D_{\varphi^{\tau}_{\mathbf{X}}(\tilde\xi)}F[D_{\tilde\xi}\varphi^{\tau}_{\mathbf{X}}[\zeta]]\big) \, \mathrm{d}\tau\\
			&\qquad+ \int_{s}^{t}S_{t-\tau}\big(D_{\varphi^{\tau}_{\mathbf{X}}(\xi)}G[D_{\xi}\varphi^{\tau}_{\mathbf{X}}[\zeta]]-D_{\varphi^{\tau}_{\mathbf{X}}(\tilde\xi)}G[D_{\tilde\xi}\varphi^{\tau}_{\mathbf{X}}[\zeta]]\big)\circ\mathrm{d}\mathbf{X}_{\tau}.
		\end{split}
	\end{align}
	Set $L_1(\xi,\tilde{\xi},\tau) \coloneqq \varphi^{\tau}_{\mathbf{X}}(\xi)-\varphi^{\tau}_{\mathbf{X}}(\tilde\xi)$ and $L_2(\xi,\tilde{\xi},\zeta,\tau) \coloneqq D_{\xi}\varphi^{\tau}_{\mathbf{X}}[\zeta]-D_{\tilde\xi}\varphi^{\tau}_{\mathbf{X}}[\zeta]$. Then
	\begin{align}\label{-1-1-1}
		\begin{split}
			(L_2(\xi,\tilde{\xi},\zeta,\tau))^{\prime} &= D_{\varphi^{\tau}_{\mathbf{X}}(\xi)}G[D_{\xi}\varphi^{\tau}_{\mathbf{X}}[\zeta]]-D_{\varphi^{\tau}_{\mathbf{X}}(\tilde\xi)}G[D_{\tilde\xi}\varphi^{\tau}_{\mathbf{X}}[\zeta]]\\
			&= \int_{0}^{1}D^{2}_{\theta\varphi^{\tau}_{\mathbf{X}}(\xi)+(1-\theta)\varphi^{\tau}_{\mathbf{X}}(\tilde\xi)}G\big[L_{1}(\xi,\tilde{\xi},\tau),D_{\xi}\varphi^{\tau}_{\mathbf{X}}[\zeta]\big] \, \mathrm{d}\theta + D_{\varphi^{\tau}_{\mathbf{X}}(\tilde\xi)}G\big[L_2(\xi,\tilde{\xi},\zeta,\tau)\big].
		\end{split}
	\end{align}
	Also from \eqref{DDEER},
	\begin{align}\label{derivative_difference}
		\begin{split}
			&(D_{\varphi^{\tau}_{\mathbf{X}}(\xi)}G[D_{\xi}\varphi^{\tau}_{\mathbf{X}}[\zeta]]-D_{\varphi^{\tau}_{\mathbf{X}}(\tilde\xi)}G[D_{\tilde\xi}\varphi^{\tau}_{\mathbf{X}}[\zeta]])^{\prime}\circ (\delta X)_{\tau,\nu}\\
			=\ &\int_{0}^{1}D^{2}_{\theta\varphi^{\tau}_{\mathbf{X}}(\xi)+(1-\theta)\varphi^{\tau}_{\mathbf{X}}(\tilde\xi)}G\big[L_{1}(\xi,\tilde{\xi},\tau),D_{\varphi^{\tau}_{\mathbf{X}}(\xi)}G[D_{\xi}\varphi^{\tau}_{\mathbf{X}}[\zeta]]\circ (\delta X)_{\tau,\nu}\big]\mathrm{d}\theta\\
			&\qquad+D_{\varphi^{\tau}_{\mathbf{X}}(\tilde\xi)}G\left[\int_{0}^{1}D^{2}_{\theta\varphi^{\tau}_{\mathbf{X}}(\xi)+(1-\theta)\varphi^{\tau}_{\mathbf{X}}(\tilde\xi)}G[L_{1}(\xi,\tilde{\xi},\tau),D_{\xi}\varphi^{\tau}_{\mathbf{X}}[\zeta]\circ(\delta X)_{\tau,\nu}] \,\mathrm{d}\theta\right]\\
			&\myquad[3]+D_{\varphi^{\tau}_{\mathbf{X}}(\tilde\xi)}G\big[D_{\varphi^{\tau}_{\mathbf{X}}(\tilde\xi)}G[L_2(\xi,\tilde{\xi},\zeta,\tau)]\circ(\delta X)_{\tau,\nu}\big]\\
			&\myquad[4] +\int_{0}^{1}D^{3}_{\theta\varphi^{\tau}_{\mathbf{X}}(\xi)+(1-\theta)\varphi^{\tau}_{\mathbf{X}}(\tilde\xi)}G\big[L_{1}(\xi,\tilde{\xi},\tau),G(\varphi^{\tau}_{\mathbf{X}}(\xi))\circ(\delta X)_{\tau,\nu},D_\xi\varphi^{\tau}_{\mathbf{X}}[\zeta]\big]\,\mathrm{d}\theta\\
			&\myquad[5]+D^{2}_{\varphi^{\tau}_{\mathbf{X}}(\tilde\xi)} G\left[ \int_{0}^{1}D_{\theta\varphi^{\tau}_{\mathbf{X}}(\xi)+(1-\theta)\varphi^{\tau}_{\mathbf{X}}(\tilde\xi)}G[L_{1}(\xi,\tilde{\xi},\tau)]\circ (\delta X)_{\tau,\nu} \, \mathrm{d}\theta,D_{\xi}\varphi^{\tau}_{\mathbf{X}}[\zeta]\right]\\ 
			&\myquad[6]+D^{2}_{\varphi^{\tau}_{\mathbf{X}}(\tilde\xi)}G\big[G(\varphi^{\tau}_{\mathbf{X}}(\tilde\xi))\circ (\delta X)_{\tau,\nu},L_2(\xi,\tilde{\xi},\zeta,\tau)\big],
		\end{split}
	\end{align}
	and
	\begin{align}\label{remiander_difference}
		\begin{split}
			&\big[D_{\varphi^{.}_{\mathbf{X}}(\xi)}G[D_{\xi}\varphi^{.}_{\mathbf{X}}[\zeta]]-D_{\varphi^{.}_{\mathbf{X}}(\tilde\xi)}G[D_{\tilde\xi}\varphi^{.}_{\mathbf{X}}[\zeta]]\big]^{\#}_{\tau,\nu}\\
			=\ &\int_{0}^{1}D^{2}_{\theta\varphi^{\tau}_{\mathbf{X}}(\xi)+(1-\theta)\varphi^{\tau}_{\mathbf{X}}(\tilde\xi)}G\big[L_{1}(\xi,\tilde{\xi},\tau),[D_{\xi}^{.}\varphi_{\mathbf{X}}[\zeta]]^{\#}_{\tau,\nu}\big]\mathrm{d}\theta+D_{\varphi^{\tau}_{\mathbf{X}}(\tilde\xi)}G[[L_2(\xi,\tilde{\xi},\zeta,.)]^{\#}_{\tau,\nu}]\\
			&\quad + \int_{0}^{1} D^{3}_{\theta \varphi^{\nu}_{\mathbf{X}}(\tilde\xi)+(1-\theta)\varphi^{\tau}_{\mathbf{X}}(\tilde\xi)}G\big[L_{1}(\xi,\tilde{\xi},\tau),[\varphi^{.}_{\mathbf{X}}(\xi)]^{\#}_{\tau,\nu},D_{\xi}\varphi^{\tau}_{\mathbf{X}}[\zeta]\big]\mathrm{d}\theta+D^{2}_{\varphi^{\tau}_{\mathbf{X}}(\tilde\xi)}G\big[[L_{1}(\xi,\tilde{\xi},.)]^{\#}_{\tau,\nu},D_{\xi}\varphi^{\tau}_{\mathbf{X}}[\zeta]\big]\\
			&\quad + D^{2}_{\varphi^{\tau}_{\mathbf{X}}(\tilde\xi)}G\big[[\varphi^{.}_{\mathbf{X}}(\tilde\xi)]^{\#}_{\tau,\nu},L_2(\xi,\tilde{\xi},\zeta,\tau)\big]\\
			&\quad + \int_{0}^{1}(1-\theta)\big(D^{3}_{\theta \varphi^{\nu}_{\mathbf{X}}(\xi)+(1-\theta)\varphi^{\tau}_{\mathbf{X}}(\xi)}G-D^{3}_{\theta \varphi^{\nu}_{\mathbf{X}}(\tilde\xi)+(1-\theta)\varphi^{\tau}_{\mathbf{X}}(\tilde\xi)}G\big)[(\delta\varphi^{.}_{\mathbf{X}}(\xi))_{\tau,\nu},(\delta\varphi^{.}_{\mathbf{X}}(\xi))_{\tau,\nu},D_{\xi}\varphi^{\tau}_{\mathbf{X}}[\zeta]]\, \mathrm{d}\theta\\
			&\quad + \int_{0}^{1}(1-\theta)D^{3}_{\theta \varphi^{\nu}_{\mathbf{X}}(\tilde\xi)+(1-\theta)\varphi^{\tau}_{\mathbf{X}}(\tilde\xi)}G[(\delta L_{1}(\xi,\tilde{\xi},.))_{\tau,\nu},(\delta\varphi^{.}_{\mathbf{X}}(\xi))_{\tau,\nu},D_{\xi}\varphi^{\tau}_{\mathbf{X}}[\zeta]]\, \mathrm{d}\theta\\
			&\quad + \int_{0}^{1}(1-\theta)D^{3}_{\theta \varphi^{\nu}_{\mathbf{X}}(\tilde\xi)+(1-\theta)\varphi^{\tau}_{\mathbf{X}}(\tilde\xi)}G[(\delta\varphi^{.}_{\mathbf{X}}(\tilde\xi))_{\tau,\nu},(\delta L_{1}(\xi,\tilde{\xi},.))_{\tau,\nu},D_{\xi}\varphi^{\tau}_{\mathbf{X}}[\zeta]]\, \mathrm{d}\theta\\
			&\quad + \int_{0}^{1}(1-\theta)D^{3}_{\theta \varphi^{\nu}_{\mathbf{X}}(\tilde\xi)+(1-\theta)\varphi^{\tau}_{\mathbf{X}}(\tilde\xi)}G[(\delta\varphi^{.}_{\mathbf{X}}(\tilde\xi))_{\tau,\nu},(\delta\varphi^{.}_{\mathbf{X}}(\tilde\xi))_{\tau,\nu},L_{2}(\xi,\tilde{\xi},\zeta,\tau)] \, \mathrm{d}\theta\\
			&\quad + \int_{0}^{1}\big(D^{2}_{\theta \varphi^{\nu}_{\mathbf{X}}(\xi)+(1-\theta)\varphi^{\tau}_{\mathbf{X}}(\xi)} G-D^{2}_{\theta \varphi^{\nu}_{\mathbf{X}}(\tilde\xi)+(1-\theta)\varphi^{\tau}_{\mathbf{X}}(\tilde\xi)} G \big)[(\delta\varphi^{.}_{\mathbf{X}}(\xi))_{\tau,\nu},(\delta D_{\xi}\varphi^{.}_{\mathbf{X}}[\zeta])_{\tau,\nu}] \, \mathrm{d}\theta\\
			&\quad + \int_{0}^{1}D^{2}_{\theta \varphi^{\nu}_{\mathbf{X}}(\tilde\xi)+(1-\theta)\varphi^{\tau}_{\mathbf{X}}(\tilde\xi)} G[(\delta L_{1}(\xi,\tilde{\xi},.))_{\tau,\nu},(\delta D_{\xi}\varphi^{.}_{\mathbf{X}}[\zeta])_{\tau,\nu}]\, \mathrm{d}\theta\\
			&\quad + \int_{0}^{1}D^{2}_{\theta \varphi^{\nu}_{\mathbf{X}}(\tilde\xi)+(1-\theta)\varphi^{\tau}_{\mathbf{X}}(\tilde\xi)} G[(\delta\varphi^{.}_{\mathbf{X}}(\tilde\xi))_{\tau,\nu},(\delta L_{2}(\xi,\tilde{\xi},\zeta,.))_{\tau,\nu}]\,\mathrm{d}\theta.
		\end{split}
	\end{align}
	From \eqref{derivative_difference}, it is straightforward to check that
	\begin{align}\label{000}
		\begin{split}
			&\sup_{\tau\in [s,t]}\vert (D_{\varphi^{\tau}_{\mathbf{X}}(\xi)}G[D_{\xi}\varphi^{\tau}_{\mathbf{X}}[\zeta]]-D_{\varphi^{\tau}_{\mathbf{X}}(\tilde\xi)}G[D_{\tilde\xi}\varphi^{\tau}_{\mathbf{X}}[\zeta]])^{\prime}\vert_{\alpha-\eta-\gamma}^{(n)}\\
			\lesssim\ &\Vert L_{1}(\xi,\tilde{\xi},.)\Vert_{\mathcal{D}_{\mathbf{X},{\alpha}}^{\gamma}([s,t])}\Vert D_{\xi}\varphi^{.}_{\mathbf{X}}[\zeta]\Vert_{\mathcal{D}_{\mathbf{X},{\alpha}}^{\gamma}([s,t])}+\Vert L_2(\xi,\tilde{\xi},\zeta,.)\Vert_{\mathcal{D}_{\mathbf{X},{\alpha}}^{\gamma}([s,t])}.
		\end{split}
	\end{align}
	Next, we want to find a bound for the $C^{\gamma}([s,t];\mathcal{B}_{{\alpha}-\eta-2\gamma})$-norm of the term $(D_{\varphi^{.}_{\mathbf{X}}(\xi)}G[D_{\xi}\varphi^{.}_{\mathbf{X}}[\zeta]]-D_{\varphi^{.}_{\mathbf{X}}(\tilde\xi)}G[D_{\tilde\xi}\varphi^{.}_{\mathbf{X}}[\zeta]])^{\prime}$. We do this by estimating all terms on the right hand side of \eqref{derivative_difference} separately. From \eqref{NMMMMMLY}, it is straightforward to check that for all terms but
	\begin{align}\label{IV444}
		IV(\tau)\circ(\delta X)_{\tau,\nu} &\coloneqq \int_{0}^{1}D^{3}_{\theta\varphi^{\tau}_{\mathbf{X}}(\xi)+(1-\theta)\varphi^{\tau}_{\mathbf{X}}(\tilde\xi)}G\big[L_{1}(\xi,\tilde{\xi},\tau),G(\varphi^{\tau}_{\mathbf{X}}(\xi))\circ(\delta X)_{\tau,\nu},D_\xi\varphi^{\tau}_{\mathbf{X}}[\zeta]\big] \, \mathrm{d}\theta,
	\end{align}
	we can bound their $C^{\gamma}([s,t];\mathcal{B}_{{\alpha}-\eta-2\gamma})$-norm up to a constant by 
	\begin{align}
		\begin{split}\label{111}
			&(1+\Vert X\Vert_{\gamma,[s,t]})\big(1+\Vert \varphi^{.}_{\mathbf{X}}(\xi)\Vert_{\mathcal{D}_{\mathbf{X},{\alpha}}^{\gamma}([s,t])}+\Vert \varphi^{.}_{\mathbf{X}}(\tilde\xi) \Vert_{\mathcal{D}_{\mathbf{X},{\alpha}}^{\gamma}([s,t])}\big) \Vert L_{1}(\xi,\tilde{\xi},.)\Vert_{\mathcal{D}_{\mathbf{X},{\alpha}}^{\gamma}([s,t])}\Vert D_{\xi}\varphi^{.}_{\mathbf{X}}[\zeta]\Vert_{\mathcal{D}_{\mathbf{X},{\alpha}}^{\gamma}([s,t])}\\
			&\quad + (1+\Vert X\Vert_{\gamma,[s,t]})\big(1+\Vert \varphi^{.}_{\mathbf{X}}(\tilde\xi) \Vert_{\mathcal{D}_{\mathbf{X},{\alpha}}^{\gamma}([0,T])}\big)\Vert L_2(\xi,\tilde{\xi},\zeta,.)\Vert_{\mathcal{D}_{\mathbf{X},{\alpha}}^{\gamma}([s,t])}.
		\end{split}
	\end{align}
	To estimate $IV(\cdot)$ in \eqref{IV444} in the $ C^{\gamma}([s,t];\mathcal{B}_{{\alpha}-\eta-2\gamma})$-norm, we first note that 
	\begin{align*}
		IV(\tau)\circ(\delta X)_{\tau,\nu} &= \int_{0}^{1}D^{3}_{\theta\varphi^{\tau}_{\mathbf{X}}(\xi)+(1-\theta)\varphi^{\tau}_{\mathbf{X}}(\tilde\xi)}G\big[L_{1}(\xi,\tilde{\xi},\tau),G(\varphi^{\tau}_{\mathbf{X}}(\xi))\circ(\delta X)_{\tau,\nu},D_\xi\varphi^{\tau}_{\mathbf{X}}[\zeta]\big] \, \mathrm{d}\theta \\
		&= D^{2}_{\varphi^{\tau}_{\mathbf{X}}(\xi)}G\big[G(\varphi^{\tau}_{\mathbf{X}}(\xi))\circ(\delta X)_{\tau,\nu},D_\xi\varphi^{\tau}_{\mathbf{X}}[\zeta]\big]-D^{2}_{\varphi^{\tau}_{\mathbf{X}}(\tilde\xi)}G\big[G(\varphi^{\tau}_{\mathbf{X}}(\xi))\circ(\delta X)_{\tau,\nu},D_\xi\varphi^{\tau}_{\mathbf{X}}[\zeta]\big].
	\end{align*}
	Therefore,
	\begin{align*}
		\sup_{\substack{\tau,\nu\in [s,t],\\ \tau<\nu}}\frac{\vert (\delta IV)_{\tau,\nu}\vert_{\alpha-\eta-2\gamma}^{(n)}}{(\nu-\tau)^{\gamma}} &\lesssim (1+\Vert X\Vert_{\gamma,[s,t]}) (1+\Vert \varphi^{.}_{\mathbf{X}}(\xi)\Vert_{\mathcal{D}_{\mathbf{X},{\alpha}}^{\gamma}([s,t])})\Vert D_\xi\varphi^{.}_{\mathbf{X}}[\zeta]\Vert_{\mathcal{D}_{\mathbf{X},{\alpha}}^{\gamma}([0,T])}\\
		&\quad + (1+\Vert X\Vert_{\gamma,[s,t]})(1+\Vert \varphi^{.}_{\mathbf{X}}(\tilde\xi)\Vert_{\mathcal{D}_{\mathbf{X},{\alpha}}^{\gamma}([0,T])})\Vert D_{\tilde{\xi}}\varphi^{.}_{\mathbf{X}}[\zeta]\Vert_{\mathcal{D}_{\mathbf{X},{\alpha}}^{\gamma}([s,t])}.
	\end{align*}
	Furthermore,
	\begin{align}\label{222}
		&\sup_{\substack{\tau,\nu\in [s,t],\\ \tau<\nu}}\vert (\delta IV)_{\tau,\nu}\vert_{\alpha-\eta-2\gamma}^{(n)}\lesssim \Vert L_{1}(\xi,\tilde{\xi},.)\Vert_{\mathcal{D}_{\mathbf{X},{\alpha}}^{\gamma}([0,T])}\Vert D_{\xi}\varphi^{.}_{\mathbf{X}}[\zeta]\Vert_{\mathcal{D}_{\mathbf{X},{\alpha}}^{\gamma}([s,t])}.
	\end{align}
	Assume $\gamma_1=\kappa\gamma$ where we choose $0<\kappa <1$ such that $(\gamma_1,\gamma_2)=(\kappa\gamma,\gamma)$ satisfies the assumptions in Lemma \ref{parmeter}. It follows that
	\begin{align}\label{333}
		\max\lbrace \frac{\eta}{\gamma},\frac{1}{\gamma}-2\rbrace<\kappa<1.
	\end{align}
	Consequently, from a simple interpolation, for every $0< \kappa<1$,
	\begin{align}\label{444}
		\begin{split}
			&\sup_{\substack{\tau,\nu\in [s,t] \\ \tau<\nu}} \frac{\vert (\delta IV)_{\tau,\nu}\vert_{\alpha-\eta-2\gamma}^{(n)}}{(\nu-\tau)^{\kappa\gamma}} \\
			\lesssim\ &(1+\Vert X\Vert_{\gamma,[s,t]})^{\kappa} \big[(1+\Vert \varphi^{.}_{\mathbf{X}}(\xi)\Vert_{\mathcal{D}_{\mathbf{X},{\alpha}}^{\gamma}([s,t])})^{\kappa}\Vert D_\xi\varphi^{.}_{\mathbf{X}}[\zeta]\Vert_{\mathcal{D}_{\mathbf{X},{\alpha}}^{\gamma}([0,T])}\Vert L_{1}(\xi,\tilde{\xi},.)\Vert_{\mathcal{D}_{\mathbf{X},{\alpha}}^{\gamma}([s,t])}^{1-\kappa}\\
			&\qquad + (1+\Vert \varphi^{.}_{\mathbf{X}}(\tilde\xi)\Vert_{\mathcal{D}_{\mathbf{X},{\alpha}}^{\gamma}([s,t])})^{\kappa}\Vert D_{\tilde{\xi}}\varphi^{.}_{\mathbf{X}}[\zeta]\Vert_{\mathcal{D}_{\mathbf{X},{\alpha}}^{\gamma}([0,T])}^{\kappa}  \Vert D_{\xi}\varphi^{.}_{\mathbf{X}}[\zeta]\Vert_{\mathcal{D}_{\mathbf{X},{\alpha}}^{\gamma}([0,Ts,t])}^{1-\kappa} \Vert L_{1}(\xi,\tilde{\xi},.)\Vert_{\mathcal{D}_{\mathbf{X},{\alpha}}^{\gamma}([s,t])}^{1-\kappa}\big].
		\end{split}
	\end{align}
	Therefore, from \eqref{000},\eqref{111},\eqref{222},\eqref{333},\eqref{444} and also \eqref{bounde_derivative_2}, for 
	\begin{align*}
		(L_2(\xi,\tilde{\xi},\zeta,.))^{\prime\prime}=(D_{\varphi^{\tau}_{\mathbf{X}}(\xi)}G[D_{\xi}\varphi^{\tau}_{\mathbf{X}}[\zeta]]-D_{\varphi^{\tau}_{\mathbf{X}}(\tilde\xi)}G[D_{\tilde\xi}\varphi^{\tau}_{\mathbf{X}}[\zeta]])^{\prime},
	\end{align*}
	we have
	\begin{align}\label{I_I_I}	
		\begin{split}
			&\max\left\{ \sup_{\substack{\tau,\nu\in [s,t],\\ \tau<\nu}}\frac{\vert \big(\delta (L_2(\xi,\tilde{\xi},\zeta,.))^{\prime\prime}\big)_{\tau,\nu}\vert_{\alpha-\eta-2\gamma}}{(\nu-\tau)^{\kappa\gamma}},\sup_{\tau\in [s,t]}\vert  (L_2(\xi,\tilde{\xi},\zeta,\tau))^{\prime\prime}\vert_{\alpha-\eta-\gamma} \right\} \\
			\lesssim\ &(1+\Vert X\Vert_{\gamma,[s,t]})\big(1+\Vert \varphi^{.}_{\mathbf{X}}(\xi)\Vert_{\mathcal{D}_{\mathbf{X},{\alpha}}^{\gamma}([s,t])}+\Vert \varphi^{.}_{\mathbf{X}}(\tilde\xi) \Vert_{\mathcal{D}_{\mathbf{X},{\alpha}}^{\gamma}([s,t])}\big)\\
			&\quad \times \max\lbrace \Vert L_{1}(\xi,\tilde{\xi},.)\Vert_{\mathcal{D}_{\mathbf{X},{\alpha}}^{\gamma}([s,t])}, \Vert L_{1}(\xi,\tilde{\xi},.)\Vert_{\mathcal{D}_{\mathbf{X},{\alpha}}^{\gamma}([s,t])}^{1-\kappa}\rbrace \\
			&\myquad[3] \times \max\lbrace \Vert D_{\tilde{\xi}}\varphi^{.}_{\mathbf{X}}[\zeta]\Vert_{\mathcal{D}_{\mathbf{X},{\alpha}}^{\gamma}([s,t])}^{\kappa}  \Vert D_{\xi}\varphi^{.}_{\mathbf{X}}[\zeta]\Vert_{\mathcal{D}_{\mathbf{X},{\alpha}}^{\gamma}([s,t])}^{1-\kappa},\Vert D_{\xi}\varphi^{.}_{\mathbf{X}}[\zeta]\Vert_{\mathcal{D}_{\mathbf{X},{\alpha}}^{\gamma}([s,t])}\rbrace\\
			&\myquad[4] +(1+\Vert X\Vert_{\gamma,[s,t]})\big(1+\Vert \varphi^{.}_{\mathbf{X}}(\tilde\xi) \Vert_{\mathcal{D}_{\mathbf{X},{\alpha}}^{\gamma}([0,T])}\big)\Vert L_2(\xi,\tilde{\xi},\zeta,.)\Vert_{\mathcal{D}_{\mathbf{X},{\alpha}}^{\gamma}([s,t])}.
		\end{split}
	\end{align}
	The next step is to give an estimate for
	\begin{align*}
		\big\Vert\big[D_{\varphi^{.}_{\mathbf{X}}(\xi)}G[D_{\xi}\varphi^{.}_{\mathbf{X}}[\zeta]]-D_{\varphi^{.}_{\mathbf{X}}(\tilde\xi)}G[D_{\tilde\xi}\varphi^{.}_{\mathbf{X}}[\zeta]]\big]^{\#}\big\Vert_{\mathcal{E}^{\gamma,2\gamma}_{{\alpha}-\eta;[s,t]}}.
	\end{align*}
	As before, we will estimate all terms on the right hand side of \eqref{remiander_difference} separately. Using \eqref{NMMMMMLY}, all terms but
	\begin{align*}
		(\tilde{IV})_{\tau,\nu} \coloneqq  \int_{0}^{1}(1-\theta)\big(D^{3}_{\theta \varphi^{\nu}_{\mathbf{X}}(\xi)+(1-\theta)\varphi^{\tau}_{\mathbf{X}}(\xi)}G-D^{3}_{\theta \varphi^{\nu}_{\mathbf{X}}(\tilde\xi)+(1-\theta)\varphi^{\tau}_{\mathbf{X}}(\tilde\xi)}G\big)[(\delta\varphi^{.}_{\mathbf{X}}(\xi))_{\tau,\nu},(\delta\varphi^{.}_{\mathbf{X}}(\xi))_{\tau,\nu},D_{\xi}\varphi^{\tau}_{\mathbf{X}}[\zeta]]\mathrm{d}\theta
	\end{align*}
	can be bounded by a constant times
	\begin{align}\label{555}
		\begin{split}
			&(1+\Vert X\Vert_{\gamma,[s,t]})^2\Vert L_{1}(\xi,\tilde{\xi},.)\Vert_{\mathcal{D}_{\mathbf{X},{\alpha}}^{\gamma}([s,t])}(\Vert \varphi^{.}_{\mathbf{X}}(\xi)\Vert_{\mathcal{D}_{\mathbf{X},{\alpha}}^{\gamma}([s,t])}+\Vert \varphi^{.}_{\mathbf{X}}(\tilde\xi)\Vert_{\mathcal{D}_{\mathbf{X},{\alpha}}^{\gamma}([s,t])})\Vert D_{\xi}\varphi^{\tau}_{\mathbf{X}}[\zeta]\Vert_{\mathcal{D}_{\mathbf{X},{\alpha}}^{\gamma}([s,t])}\\
			&\quad+  (1+\Vert X\Vert_{\gamma,[s,t]})^2 \Vert L_2(\xi,\tilde{\xi},\zeta,.)\Vert_{\mathcal{D}_{\mathbf{X},{\alpha}}^{\gamma}([s,t])}(1+\Vert \varphi^{.}_{\mathbf{X}}(\tilde\xi)\Vert_{\mathcal{D}_{\mathbf{X},{\alpha}}^{\gamma}([s,t])})(1+\Vert \varphi^{.}_{\mathbf{X}}(
			\xi)\Vert_{\mathcal{D}_{\mathbf{X},{\alpha}}^{\gamma}([s,t])}).
		\end{split}
	\end{align}
	For the remaining term, we have the estimate
	\begin{align}\label{666}
		\begin{split}
			& \sup_{\substack{\tau,\nu\in [s,t],\\ \tau<\nu}}\Vert (\delta\tilde{IV})_{\tau,\nu}\Vert_{\alpha-\eta}\lesssim \Vert \varphi^{.}_{\mathbf{X}}(
			\xi)\Vert_{\mathcal{D}_{\mathbf{X},{\alpha}}^{\gamma}([s,t])}^2
			\Vert D_{\xi}\varphi^{.}_{\mathbf{X}}[\zeta]\Vert_{\mathcal{D}_{\mathbf{X},{\alpha}}^{\gamma}([s,t])}.
		\end{split}
	\end{align}
	From \eqref{key} and \eqref{NMMMMMLY}, for a polynomial ${Q}_3$,
	\begin{align}\label{777}
		\begin{split}
			\sup_{\substack{\tau,\nu\in [s,t],\\ \tau<\nu}}\frac{\Vert (\delta\tilde{IV})_{\tau,\nu}\Vert_{\alpha-\eta-2\gamma}}{(\nu-\tau)^{2\gamma}} &\lesssim  (1+\Vert X\Vert_{\gamma,[s,t]})^2{Q}_3(\Vert \varphi^{.}_{\mathbf{X}}(
			\xi)\Vert_{\mathcal{D}_{\mathbf{X},{\alpha}}^{\gamma}([s,t])},\Vert \varphi^{.}_{\mathbf{X}}(
			\tilde\xi)\Vert_{\mathcal{D}_{\mathbf{X},{\alpha}}^{\gamma}([s,t])})\\
			&\quad \times \Vert L_1(\xi,\tilde{\xi},.)\Vert_{\mathcal{D}_{\mathbf{X},{\alpha}}^{\gamma}([0,T])}^r  \Vert D_{\xi}\varphi^{.}_{\mathbf{X}}[\zeta]\Vert_{\mathcal{D}_{\mathbf{X},{\alpha}}^{\gamma}([s,t])}.
		\end{split}
	\end{align}
	From the interpolation property $\vert x\vert_{\alpha-\eta-\gamma}\lesssim \vert x\vert_{\alpha-\eta}^{\frac{1}{2}}\vert x\vert_{\alpha-\eta-2\gamma}^{\frac{1}{2}}$, \eqref{key}, \eqref{666} and \eqref{777}, for a polynomial $Q_4$, 
	\begin{align}\label{888}
		\begin{split}
			\Vert\tilde{IV}\Vert_{\mathcal{E}^{\gamma,2\gamma}_{{\alpha}-\eta;[s,t]}} &\lesssim (1+\Vert X\Vert_{\gamma,[s,t]})^2{Q}_4(\Vert \varphi^{.}_{\mathbf{X}}(\xi)\Vert_{\mathcal{D}_{\mathbf{X},{\alpha}}^{\gamma}([s,t])},\Vert \varphi^{.}_{\mathbf{X}}(
			\tilde\xi)\Vert_{\mathcal{D}_{\mathbf{X},{\alpha}}^{\gamma}([s,t])})\\
			&\quad \times \max\lbrace \Vert L_1(\xi,\tilde{\xi},.)\Vert_{\mathcal{D}_{\mathbf{X},{\alpha}}^{\gamma}([0,T])}^\frac{r}{2},\Vert L_1(\xi,\tilde{\xi},.)\Vert_{\mathcal{D}_{\mathbf{X},{\alpha}}^{\gamma}([0,T])}^r\rbrace  \Vert D_{\xi}\varphi^{.}_{\mathbf{X}}[\zeta]\Vert_{\mathcal{D}_{\mathbf{X},{\alpha}}^{\gamma}([s,t])}.
		\end{split}
	\end{align}
	To summarize, \eqref{555} and \eqref{888} show that there is a polynomial $Q_5$ such that
	\begin{align}\label{II_II_II}
		\begin{split}
			&\Vert(D_{\varphi^{.}_{\mathbf{X}}(\xi)}G[D_{\xi}\varphi^{.}_{\mathbf{X}}[\zeta]]-D_{\varphi^{.}_{\mathbf{X}}(\tilde\xi)}G[D_{\tilde\xi}\varphi^{.}_{\mathbf{X}}[\zeta]])^{\#}\Vert_{\mathcal{E}^{\gamma,2\gamma}_{{\alpha}-\eta;[s,t]}}\\
			\lesssim\ &(1+\Vert X\Vert_{\gamma,[s,t]})^2\max\lbrace \Vert L_1(\xi,\tilde{\xi},.)\Vert_{\mathcal{D}_{\mathbf{X},{\alpha}}^{\gamma}([0,T])}^\frac{r}{2},\Vert L_1(\xi,\tilde{\xi},.)\Vert_{\mathcal{D}_{\mathbf{X},{\alpha}}^{\gamma}([0,T])}\rbrace\\
			&\quad \times {Q}_5(\Vert \varphi^{.}_{\mathbf{X}}(\xi)\Vert_{\mathcal{D}_{\mathbf{X},{\alpha}}^{\gamma}([s,t])},\Vert \varphi^{.}_{\mathbf{X}}(\tilde\xi)\Vert_{\mathcal{D}_{\mathbf{X},{\alpha}}^{\gamma}([s,t])})\Vert D_{\xi}\varphi^{\tau}_{\mathbf{X}}[\zeta]\Vert_{\mathcal{D}_{\mathbf{X},{\alpha}}^{\gamma}([s,t])}\\
			&\myquad[3] + (1+\Vert X\Vert_{\gamma,[s,t]})^2(1+\Vert \varphi^{.}_{\mathbf{X}}(\tilde\xi)\Vert_{\mathcal{D}_{\mathbf{X},{\alpha}}^{\gamma}([s,t])})(1+\Vert \varphi^{.}_{\mathbf{X}}(
			\xi)\Vert_{\mathcal{D}_{\mathbf{X},{\alpha}}^{\gamma}([s,t])})\Vert L_2(\xi,\tilde{\xi},\zeta,.)\Vert_{\mathcal{D}_{\mathbf{X},{\alpha}}^{\gamma}([s,t])}.
		\end{split}
	\end{align}
	Choosing $\kappa$ such that \eqref{333} holds, from \eqref{-1-1-1},\eqref{I_I_I}, \eqref{II_II_II} and \eqref{BBBBasasBB}, we see that for $t-s<1$
	\begin{align}\label{ROUGH_ESTIMATE}
		\begin{split}
			&\left\| \int_{s}^{.}S_{.-\tau}\big(D_{\varphi^{\tau}_{\mathbf{X}}(\xi)}G[D_{\xi}\varphi^{\tau}_{\mathbf{X}}[\zeta]] - D_{\varphi^{\tau}_{\mathbf{X}}(\tilde\xi)} G[D_{\tilde\xi}\varphi^{\tau}_{\mathbf{X}}[\zeta]]\big)\circ\mathrm{d}\mathbf{X}_{\tau} \right\|_{\mathcal{D}_{\mathbf{X},{\alpha}}^{\gamma}([s,t])} \\
			\lesssim\  &\varrho_{\gamma}(\mathbf{X},[s,t])\vert L_{1}(\xi,\tilde{\xi},s)\vert_{\alpha}\vert D_{\xi}\varphi^{\tau}_{\mathbf{X}}[\zeta]\vert_{\alpha} + \vert L_2(\xi,\tilde{\xi},\zeta,s)\vert_{\alpha} \big)+ (t-s)^{\kappa(\gamma-\eta)}A(s,t,\xi,\tilde{\xi},\zeta,\mathbf{X}) \\
			&\quad + (t-s)^{\kappa(\gamma-\eta)}B(s,t,\xi,\tilde{\xi},\zeta,\mathbf{X})\Vert L_2(\xi,\tilde{\xi},\zeta,.)\Vert_{\mathcal{D}_{\mathbf{X},{\alpha}}^{\gamma}([s,t])},
		\end{split}
	\end{align}
	where for $L_{r,\kappa}(C) \coloneqq \max\lbrace C^{\frac{r}{2}},C,C^{1-\kappa}\rbrace$, 
	\begin{align}\label{AA11}
		\begin{split}
			A(s,t,\xi,\tilde{\xi},\zeta,\mathbf{X}) &\coloneqq \varrho_{\gamma}^3(\mathbf{X},[s,t])L_{r,\kappa}(\Vert L_1(\xi,\tilde{\xi},.)\Vert_{\mathcal{D}_{\mathbf{X},{\alpha}}^{\gamma}([s,t])})\\
			&\quad \times \max\left\{ \Vert D_{\tilde{\xi}}\varphi^{.}_{\mathbf{X}}[\zeta]\Vert_{\mathcal{D}_{\mathbf{X},{\alpha}}^{\gamma}([s,t])}^{\kappa}  \Vert D_{\xi}\varphi^{.}_{\mathbf{X}}[\zeta]\Vert_{\mathcal{D}_{\mathbf{X},{\alpha}}^{\gamma}([s,t])}^{1-\kappa},\Vert D_{\xi}\varphi^{.}_{\mathbf{X}}[\zeta]\Vert_{\mathcal{D}_{\mathbf{X},{\alpha}}^{\gamma}([s,t])} \right\} \\
			&\quad \times {Q}_5(\Vert \varphi^{.}_{\mathbf{X}}(
			\xi)\Vert_{\mathcal{D}_{\mathbf{X},{\alpha}}^{\gamma}([s,t])},\Vert \varphi^{.}_{\mathbf{X}}(
			\tilde\xi)\Vert_{\mathcal{D}_{\mathbf{X},{\alpha}}^{\gamma}([s,t])})
		\end{split}
	\end{align}
	and
	\begin{align}\label{BB22}
		B(s,t,\xi,\tilde{\xi},\zeta,\mathbf{X}) \coloneqq (1+\Vert \varphi^{.}_{\mathbf{X}}(\tilde\xi)\Vert_{\mathcal{D}_{\mathbf{X},{\alpha}}^{\gamma}([s,t])})(1+\Vert \varphi^{.}_{\mathbf{X}}(
		\xi)\Vert_{\mathcal{D}_{\mathbf{X},{\alpha}}^{\gamma}([s,t])}).
	\end{align}
	Similar to \eqref{Drift_Estimate}, from \eqref{poly} and \eqref{key}, for all $\tau,\nu\in [0,T] $ with $\nu-\tau<1$,
	\begin{align}\label{Drift_Estimate22}
		\begin{split}
			&\left| \int_{\tau}^{\nu}S_{\nu-x}(D_{\varphi^{x}_{\mathbf{X}}(\xi)} F[D_{\xi}\varphi^{x}_{\mathbf{X}}[\zeta]]-D_{\varphi^{x}_{\mathbf{X}}(\xi)} F[D_{\tilde\xi}\varphi^{x}_{\mathbf{X}}[\zeta]] )\, \mathrm{d}x \right|_{\alpha-i\gamma} \\
			\lesssim\ &\Big[\sup_{x\in[\tau,\nu]}P_{2}(\vert \varphi^{.}_{\mathbf{X}}(\xi)\vert_{\alpha},\vert \varphi^{.}_{\mathbf{X}}(\tilde\xi)\vert_{\alpha})\vert L_1(\xi,\tilde{\xi},.)\vert_{\alpha}^r\sup_{x\in[\tau,\nu]}\vert D_{\xi}\varphi^{x}_{\mathbf{X}}[\zeta]\vert_{\alpha} \\
			&\quad + \sup_{x\in [\tau,\nu]}P_{1}(\vert \varphi^{x}_{\mathbf{X}}(\tilde\xi)\vert_{\alpha})\sup_{x\in[\tau,\nu]}\vert L_2(\xi,\tilde{\xi},\zeta,.)\vert_{\alpha}\Big](\nu-\tau)^{\min\lbrace 1,1-\sigma+i\gamma\rbrace}.
		\end{split}
	\end{align}
	We are now ready to derive the claimed bounds. First, note that we can assume that all the polynomials that appear in our estimates are increasing. For $t-s<1$, from \eqref{poly}, \eqref{key}, \eqref{-2-2-2}, \eqref{ROUGH_ESTIMATE} and \eqref{Drift_Estimate22}
	\begin{align}\label{CC33}
		\begin{split}
			&\Vert L_2(\xi,\tilde{\xi},\zeta,.)\Vert_{\mathcal{D}_{\mathbf{X},{\alpha}}^{\gamma}([s,t])}\lesssim \varrho_{\gamma}(\mathbf{X},[s,t])\big(\vert L_{1}(\xi,\tilde{\xi},s)\vert_{\alpha}\vert D_{\xi}\varphi^{\tau}_{\mathbf{X}}[\zeta]\vert_{\alpha}+\vert L_2(\xi,\tilde{\xi},\zeta,s)\vert_{\alpha}\big)\\ 
			&\quad + (t-s)^{1-\bar{\sigma}}\bigg[\Vert L_{1}(\xi,\tilde{\xi},.)\Vert_{\mathcal{D}_{\mathbf{X},{\alpha}}^{\gamma}([s,t])}^r \Vert D_{{\xi}}\varphi^{.}_{\mathbf{X}}[\zeta]\Vert_{\mathcal{D}_{\mathbf{X},{\alpha}}^{\gamma}([s,t])}P_2(\Vert\varphi^{.}_{\mathbf{X}}(\xi)\Vert_{\mathcal{D}_{\mathbf{X},{\alpha}}^{\gamma}([s,t])},\Vert\varphi^{.}_{\mathbf{X}}(\tilde\xi)\Vert_{\mathcal{D}_{\mathbf{X},{\alpha}}^{\gamma}([s,t])})\\
			&\qquad +P_1(\Vert\varphi^{.}_{\mathbf{X}}(\tilde\xi)\Vert_{\mathcal{D}_{\mathbf{X},{\alpha}}^{\gamma}([s,t])})\Vert L_2(\xi,\tilde{\xi},\zeta,.)\Vert_{\mathcal{D}_{\mathbf{X},{\alpha}}^{\gamma}([s,t])}\bigg] \\
			&\quad + (t-s)^{\kappa(\gamma-\eta)}\bigg[A(s,t,\xi,\tilde{\xi},\zeta,\mathbf{X})+B(s,t,\xi,\tilde{\xi},\zeta,\mathbf{X})\Vert L_2(\xi,\tilde{\xi},\zeta,.)\Vert_{\mathcal{D}_{\mathbf{X},{\alpha}}^{\gamma}([s,t])}\bigg].
		\end{split}
	\end{align}  
	Set $\mu_1=\min\lbrace 1-\bar{\sigma},\kappa(\gamma-\eta)\rbrace$ and
	\begin{align*}
		&C(s,t,\xi,\tilde{\xi},\zeta,\mathbf{X}) \coloneqq\\&\max\big\lbrace \Vert L_{1}(\xi,\tilde{\xi},.)\Vert_{\mathcal{D}_{\mathbf{X},{\alpha}}^{\gamma}([s,t])}^r \Vert D_{{\xi}}\varphi^{.}_{\mathbf{X}}[\zeta]\Vert_{\mathcal{D}_{\mathbf{X},{\alpha}}^{\gamma}([s,t])}P_2(\Vert\varphi^{.}_{\mathbf{X}}(\xi)\Vert_{\mathcal{D}_{\mathbf{X},{\alpha}}^{\gamma}([s,t])},\Vert\varphi^{.}_{\mathbf{X}}(\tilde\xi)\Vert_{\mathcal{D}_{\mathbf{X},{\alpha}}^{\gamma}([s,t])}),\\&\qquad \varrho_{\gamma}(\mathbf{X},[0,T])\Vert L_{1}(\xi,\tilde{\xi},.)\Vert_{\mathcal{D}_{\mathbf{X},{\alpha}}^{\gamma}([s,t])} \Vert D_{{\xi}}\varphi^{.}_{\mathbf{X}}[\zeta]\Vert_{\mathcal{D}_{\mathbf{X},{\alpha}}^{\gamma}([s,t])}\big\rbrace.
	\end{align*}
	From \eqref{CC33} for a constant $D>1$ 
	\begin{align}\label{ppplll}
		\begin{split}
			&\Vert L_2(\xi,\tilde{\xi},\zeta,.)\Vert_{\mathcal{D}_{\mathbf{X},{\alpha}}^{\gamma}([s,t])}\leq  D\bigg[  \varrho_{\gamma}(\mathbf{X},[0,T])\vert L_2(\xi,\tilde{\xi},\zeta,s)\vert_{\alpha}+ A(s,t,\xi,\tilde{\xi},\zeta,\mathbf{X})+C(s,t,\xi,\tilde{\xi},\zeta,\mathbf{X})\\ &\quad +(t-s)^{\mu_1}\big[P_1(\Vert\varphi^{.}_{\mathbf{X}}(\tilde\xi)\Vert_{\mathcal{D}_{\mathbf{X},{\alpha}}^{\gamma}([s,t])})+B(s,t,\xi,\tilde{\xi},\zeta,\mathbf{X})\big] \Vert L_2(\xi,\tilde{\xi},\zeta,.)\Vert_{\mathcal{D}_{\mathbf{X},{\alpha}}^{\gamma}([s,t])}\bigg] \\
			&\qquad \leq D\bigg[\varrho_{\gamma}(\mathbf{X},[0,T])\vert L_2(\xi,\tilde{\xi},\zeta,s)\vert_{\alpha}+ A(0,T,\xi,\tilde{\xi},\zeta,\mathbf{X})+C(0,T,\xi,\tilde{\xi},\zeta,\mathbf{X})\\ &\myquad[3] +(t-s)^{\mu_1}\big[P_1(\Vert\varphi^{.}_{\mathbf{X}}(\tilde\xi)\Vert_{\mathcal{D}_{\mathbf{X},{\alpha}}^{\gamma}([0,T])})+B(0,T,\xi,\tilde{\xi},\zeta,\mathbf{X})\big] \Vert L_2(\xi,\tilde{\xi},\zeta,.)\Vert_{\mathcal{D}_{\mathbf{X},{\alpha}}^{\gamma}([s,t])}\bigg].
		\end{split}
	\end{align}
	To extend this estimate to large time intervals, we proceed as in Proposition \ref{linear boundd}. First we fix $0<\epsilon<1$ and set $\tau_0 \coloneqq 0$. We define
	\begin{align}\label{Step_2}
		{\tilde{\tau}}^{\mu_1} \coloneqq \frac{1-\epsilon}{D\big(P_1(\Vert\varphi^{.}_{\mathbf{X}}(\tilde\xi)\Vert_{\mathcal{D}_{\mathbf{X},{\alpha}}^{\gamma}([0,T])})+B(0,T,\xi,\tilde{\xi},\zeta,\mathbf{X})\big)}
	\end{align}
	and $\tau_{n+1} \coloneqq \tau_n+\tilde{\tau}$. Then from \eqref{ppplll} for every $n\geq 0$ such that $\tau_n\leq T$,
	\begin{align*}
		\Vert L_2(\xi,\tilde{\xi},\zeta,.)\Vert_{\mathcal{D}_{\mathbf{X},{\alpha}}^{\gamma}([\tau_n,\tau_{n+1}])}\leq \frac{D}{\epsilon}\bigg[ \varrho_{\gamma}(\mathbf{X},[0,T])\vert L_2(\xi,\tilde{\xi},\zeta,\tau_n)\vert_{\alpha}+ A(0,T,\xi,\tilde{\xi},\zeta,\mathbf{X})+C(0,T,\xi,\tilde{\xi},\zeta,\mathbf{X})\bigg].
	\end{align*}
	This yields, in particular, for $D_{\epsilon}=\frac{D}{\epsilon}$ 
	\begin{align*}
		\vert L_2(\xi,\tilde{\xi},\zeta,\tau_{n+1})\vert_{\alpha}\leq D_{\epsilon}\bigg[ \varrho_{\gamma}(\mathbf{X},[0,T])\vert L_2(\xi,\tilde{\xi},\zeta,\tau_n)\vert_{\alpha}+ A(0,T,\xi,\tilde{\xi},\zeta,\mathbf{X})+C(0,T,\xi,\tilde{\xi},\zeta,\mathbf{X})\bigg].
	\end{align*} 
	Furthermore, we can conclude that
	\begin{align*}
		&\sup_{\substack{n\geq 0,\\ \tau_n\leq T}}\Vert L_2(\xi,\tilde{\xi},\zeta,.)\Vert_{\mathcal{D}_{\mathbf{X},{\alpha}}^{\gamma}([\tau_n,\tau_{n+1}])}
		\leq\ \big(D_{\epsilon} \varrho_{\gamma}(\mathbf{X},[0,T])\big)^{N+1}\vert L_2(\xi,\tilde{\xi},\zeta,0)\vert_{\alpha}\\&+\frac{1}{D_{\epsilon} \varrho_{\gamma}(\mathbf{X},[0,T])-1}[\big(D_{\epsilon} \varrho_{\gamma}(\mathbf{X},[0,T])\big)^{N+1}-1][A(0,T,\xi,\tilde{\xi},\zeta,\mathbf{X})+C(0,T,\xi,\tilde{\xi},\zeta,\mathbf{X})],
	\end{align*}
	where $N=N:=\lfloor \frac{T}{\tilde{\tau}}\rfloor +1$. As in \eqref{total_norm_1}, we obtain
	\begin{align}\label{FINMNM}
		\Vert L_2(\xi,\tilde{\xi},\zeta,.)\Vert_{\mathcal{D}_{\mathbf{X},{\alpha}}^{\gamma}([0,T])}\lesssim (1+\Vert X\Vert_{\gamma,[0,T]})\sum_{1\leq i\leq N-1}\Vert L_2(\xi,\tilde{\xi},\zeta,.)\Vert_{\mathcal{D}_{\mathbf{X},{\alpha}}^{\gamma}([\tau_n,\tau_{n+1}])}.
	\end{align}
	Note that $L_2(\xi,\tilde{\xi},\zeta,0)=0$. Finally from \eqref{bounde_derivative}, \eqref{bounde_derivative_2}, \eqref{AA11}, \eqref{BB22},\eqref{Step_2} and \eqref{FINMNM} for a constant  $E_{\epsilon}>0$ and two increasing  polynomials $R_1,R_2$ 
	\begin{align*}
		&\Vert L_2(\xi,\tilde{\xi},\zeta,.)\Vert_{\mathcal{D}_{\mathbf{X},{\alpha}}^{\gamma}([0,T])} \\
		\leq\ &E_{\epsilon}\max\lbrace \vert\xi-\tilde{\xi}\vert_{\alpha},\lbrace \vert\xi-\tilde{\xi}\vert_{\alpha}^{1-\kappa},\lbrace \vert\xi-\tilde{\xi}\vert_{\alpha}^{\frac{r}{2}}\rbrace\vert\zeta\vert_{\alpha}\\
		&\quad \times \exp\big(E_{\epsilon}\big[R_1(\sup_{0\leq \rho\leq 1}\Vert\varphi^{.}_{\mathbf{X}}(\rho\xi+(1-\rho)\tilde{\xi})\Vert_{\mathcal{D}_{\mathbf{X},{\alpha}}^{\gamma}([0,T])})+R_{2}(\Vert X\Vert_{\gamma,[0,T]},\Vert \mathbb{X}\Vert_{2\gamma,[0,T]})\big]\big)
	\end{align*}

\end{proof}

\section{Lyapunovexponents, invariant manifolds, and stability}
In this part of our manuscript, we show how the obtained estimates from the former sections can be applied to deduce the existence of a Lyapunov spectrum that contains information about the long-time behaviour of the solution to a stochastic partial differential equation. Furthermore, we will prove the existence of invariant manifolds. As a corollary, we will able to prove path-wise exponential stability in a neighborhood of stationary points provided all Lyapunov exponents are negative. \smallskip

We first recall some basic definitions in this field.
\begin{definition}
	Let $(\Omega,\mathcal{F},\P)$  be a probability space and $\mathbb{T}$ be either $\mathbb{Z}$ or $\mathbb{R}$. Assume that there exists a family of measurable maps $\lbrace\theta_t\rbrace_{t\in\mathbb{T}}$ on $\Omega$ such that
	\begin{itemize}
		\item[(i)] $\theta_{0}= \operatorname{id}$,
		\item [(ii)]for every $t,s\in \mathbb{T}$: $\theta_{t+s}=\theta_{t}\circ\theta_{s}$,
		\item [(iii)] if $\mathbb{T}=\mathbb{R}$, then $(t,\omega)\rightarrow\theta_{t}\omega$ is $\mathcal{B}(\mathbb{R})\otimes\mathcal{F}/\mathcal{F}$- measurable ,
		\item [(iv)]for every $t\in\mathbb{T}$: $\P\theta_t=\P$.
	\end{itemize}
	We then call $(\Omega,\mathcal{F},\lbrace\theta_t\rbrace_{t\in\mathbb{T}},\P)$ an \emph{invertible measure-preserving dynamical system}. $(\Omega,\mathcal{F},\lbrace\theta_t\rbrace_{t\in\mathbb{T}},\P)$ in addition is called \emph{ergodic} if for every $t\in\mathbb{T}$,  $\theta_t$ is an ergodic map.
\end{definition}
Another basic definition is that of a cocycle.
\begin{definition}
	Let $\mathcal{X}$ be a separable Banach space and $(\Omega,\mathcal{F},\lbrace\theta_t\rbrace_{t\in\mathbb{T}},\P)$ an invertible measure-preserving dynamical system. Assume $\mathbb{T}^{+}$ be the non negative part of the $\mathbb{T}$. A map
	\begin{align*}
		\phi \colon \mathbb{T}^{+}\times\Omega\times\mathcal{X}\rightarrow \mathcal{X}
	\end{align*}
	that is jointly measurable and satisfies the \emph{cocycle property} 
	\begin{align*}
		\forall s,t\in \mathbb{T}^{+}, s<t: \ \phi(s+t,\omega,x)=\phi(s,\theta_{t}\omega,\phi(t,\omega,x))
	\end{align*}
	is called \emph{measurable cocycle}. This map is a \emph{$C^{k}$-cocycle} if for every fixed $(s,\omega)\in\mathbb{T}^{+}\times\Omega$, $\phi(s,\omega,.):\mathcal{X}\rightarrow\mathcal{X}$ is a $C^{k}$-map. If the same map is linear, we call it a \emph{linear cocycle}.
\end{definition}
We now recall definitions and results from \cite{BRS17} where the relation between rough path theory and random dynamical systems was studied systematically. 

\begin{definition}
	Assume that $(\Omega,\mathcal{F},\lbrace\theta_t\rbrace_{t\in\mathbb{T}},\P)$ is an invertible measure-preserving dynamical system. Let $p\geq 1$ and $N \in \N$ with $p-1 < N \leq p$. A process $\mathbf{X} \colon \mathbb{R}\times\Omega\rightarrow T^{N}(\mathbb{R}^n)$ is called a $p$-variation \emph{geometric rough path cocycle} if for all $\omega \in \Omega$ and  every $s,t \in \mathbb{T}$ with $s\leq t$,
	\begin{itemize}
		\item [(i)] $\mathbf{X}(\omega) \in \mathscr{C}_{0}^{0,p-\text{var}}(\mathbb{R},T^{N}(\mathbb{R}^n))$, i.e. $\mathbf{X}(\omega)$ is a geometric $p$-variation rough path,
		\item [(ii)] $\mathbf{X}_{s+t}(\omega)=\mathbf{X}_{s}(\omega)\otimes\mathbf{X}_{t}(\theta_s\omega)$. In other words, $\mathbf{X}_{s,s+t}(\omega)=\mathbf{X}_{t}(\theta_s\omega)$.
	\end{itemize}
	If $\mathbf{X}$ satisfies the second item, then we say that it enjoys the cocycle property.
\end{definition}
Following results from \cite{BRS17} allow us to interpret the solutions of our equation as a random dynamical systems.
\begin{theorem}
	Assume that $X \colon \mathbb{R}\rightarrow\mathbb{R}^n$ is a continuous, centered Gaussian process such that all components are independent and the distribution of the process $(X_{t+t_{0}}-X_{t_{0}})_{t\in\mathbb{R}}$ does not depend on $t_{0}$. Also, assume that its covariance function has finite 2-dimensional $\varrho$-variation (cf. \cite{FV10} for definition) on every square $[s,t]^2$ for some $\varrho\in [1,2)$. Let $\overline{\mathbf{X}}$ be the natural lift of $X$ in the sense of \cite{FV10} with sample paths in $\mathscr{C}_{0}^{0,p-\text{var}}(\mathbb{R},T^{N}(\mathbb{R}^n))$ for some $p>2\varrho$ and $p-1 < N \leq p$. Then there exists an invertible measure-preserving dynamical system $(\Omega,\mathcal{F},\lbrace\theta_t\rbrace_{t\in\mathbb{T}},\P)$ and a $\mathscr{C}_{0}^{0,p-\text{var}}(\mathbb{R},T^{N}(\mathbb{R}^n))$-valued random variable $\mathbf{X}$ with a same law as $\overline{\mathbf{X}}$ that enjoys the cocycle property.
\end{theorem}
We accept the following assumption in the rest of this section.
\begin{assumption}\label{ASSUU}
	\begin{itemize}
		\item [(i)] $(\Omega,\mathcal{F},\lbrace\theta_t\rbrace_{t\in\mathbb{T}},\P)$ is an invertible measure-preserving dynamical system.
		\item [(ii)] $(\Omega,\mathcal{F},\lbrace\theta_t\rbrace_{t\in\mathbb{T}},\P)$ is ergodic.
		\item [(iii)]  For an abstract Wiener space $(\mathcal{W},\mathcal{H},\mu)$, we assume that $X$ is a Gaussian process defined on it that can be enhanced to a weakly $\gamma$-H\"older \emph{geometric rough path} $\mathbf{X}=(X,\mathbb{X})$, $\frac{1}{3} < \gamma \leq \frac{1}{2}$ which is also a $\frac{1}{\gamma}$-variation geometric rough path cocycle.
		\item [(iv)] We accept Assumption \ref{Cameron-Martin}.
		\item [(v)] We accept the assumptions that we made in  Proposition \ref{linear boundd} and Proposition \ref{MMMKKKKKKTTTT}.
	\end{itemize}
\end{assumption}
Remember that we used $\varphi^{t}_{\mathbf{X}(\omega)}(\xi)$ to denote the solutions to our SPDE. For the sake of simplicity, we naively use $\varphi^{t}_{\omega}(\xi)$ from now on. By this notation, we can now easily prove 
\begin{proposition}
	The solution map 
	\begin{align*}
		\varphi \colon [0,\infty)\times\Omega\times \mathcal{B}_{\alpha} &\rightarrow \mathcal{B}_{\alpha} \\
		(t,\omega,\xi) &\mapsto \varphi^{t}_{\omega}(\xi)
	\end{align*}
	is a $C^{1}$-cocycle on $\mathcal{B}_{\alpha}$
\end{proposition}
\begin{proof}
	This is a direct consequence of our assumptions on $\mathbf{X}$ and the pathwise nature of the solution concept.
\end{proof}
We now define a random variable that serves as a random fixed point.
\begin{definition}\label{def:stationary_point}
	A random point $Y \colon \Omega\rightarrow \mathcal{B}_{\alpha}$ is called a \emph{stationary point} if 
	\begin{itemize}
		\item[(i)] $Y$ is a measurable map and
		\item [(ii)] for every $t>0$ and $\omega\in\Omega$, $\varphi^{t}_{\omega}(Y_{\omega}) = Y_{\theta_{t}\omega}$.
	\end{itemize}
\end{definition}
\begin{remark}
	Let $Y$ be a stationary point. Then one can easily check that the linearized map $\psi^{t}_{\omega}(\zeta) \coloneqq D_{Y_{\omega}}\varphi^{t}_{\omega}[\zeta]$ is a linear cocycle.  
\end{remark}
Before stating our first result, we need an auxiliary Lemma.
\begin{lemma}\label{NMMMAA}
	Let $(Y_{\omega})_{\omega\in\Omega}$ be a stationary point for $\varphi$ such that
	\begin{align}\label{iioo1}
		\vert Y_{\omega}\vert_{\alpha}\in\cap_{p\geq 1}\mathcal{L}^{p}(\Omega).
	\end{align} 
	Let us to fix $t_0>0$. Then we have 
	\begin{align*}
		\sup_{0\leq t_1\leq t_{0}}\log^{+}\big(\Vert \psi^{t_0-t_1}_{\theta_{t_1}\omega}\Vert_{\mathcal{L}(\mathcal{B}_{\alpha},\mathcal{B}_{\alpha})}\big)\in \mathcal{L}^{1}(\Omega).
	\end{align*}
\end{lemma}
\begin{proof}
	By definition $\psi^{t_0-t_1}_{\theta_{t_1}\omega}=D_{Y_{\theta_1\omega}}\varphi^{t_0-t_1}_{\mathbf{X}(\omega)}$.	From our the bound \eqref{bounde_derivative} in Proposition \ref{linear boundd},
	\begin{align}\label{HHHBBB1}
		\begin{split}
			&\sup_{t_1\in[0,t_0]}\log^{+}\big(\Vert D_{Y_{\theta_{t_0-t}(\omega)}} \varphi^t_{\mathbf{X(\omega)}} \Vert_{\mathcal{L}(\mathcal{B}_{\alpha},\mathcal{B}_{\alpha})}\big)\leq \log(M)+ M\log(1+\Vert X(\omega)\Vert_{\gamma,[0,t_0]})\\
			&\quad \times \sup_{t_1\in[0,t_0]}\bigg[\sup_{\tau\in [0,t_0]}P_1(\vert \varphi^{\tau}_{\mathbf{\omega}}(Y_{\theta_{t_0-t_1}\omega}) \vert_{\alpha})\\&\quad +\varrho_{\gamma}^3(\mathbf{X}(\omega),[0,t_0])\big(1+\Vert(\varphi^{.}_{\mathbf{X}(\omega)}(Y_{\theta_{t_0-t_1}\omega}),G(\varphi^{.}_{\omega}(Y_{\theta_{t_0-t_1}\omega})))\Vert_{\mathcal{D}_{\mathbf{X}(\omega),{\alpha}}^{\gamma}([0,t_0])}^2\big)\bigg]^{\frac{1}{\min{1-\bar{\sigma},\gamma-\eta}}}.
		\end{split}
	\end{align}
	So, from \ref{HHHBBB1} and Theorem \ref{thm:integrability_RPDE}, it is enough to show
	\begin{align*}
		\sup_{t_1\in[0,t_0]}\vert Y_{\theta_{t_1}\omega}\vert_{\alpha}\in\cap_{p\geq 1}\mathcal{L}^{p}(\Omega).
	\end{align*}
	From the definition of stationary point
	\begin{align}\label{NMMMM}
		\sup_{t_1\in[0,t_0]}\vert Y_{\theta_{t_1}\omega}\vert_{\alpha}\leq \Vert(\varphi^{.}_{\mathbf{X}(\omega)}(Y_{\omega}),G(\varphi^{.}_{\omega}(Y_{\omega})))\Vert_{\mathcal{D}_{\mathbf{X}(\omega),{\alpha}}^{\gamma}([0,t_0])}.
	\end{align}
	Therefore, our claim follows from \eqref{NMMMM}, our assumption and Theorem \ref{thm:integrability_RPDE}.
\end{proof}
Now, we are ready to state our first result in this section, which is a consequence of the multiplicative ergodic theorem on the Banach spaces and our careful estimations in the previous section for the Gaussian drivers.

\begin{theorem}\label{MET}
	Assume that $(Y_{\omega})_{\omega\in\Omega}$ be a stationary point for $\varphi$ such that
	\begin{align}\label{iioo}
		\vert Y_{\omega}\vert_{\alpha}\in\cap_{p\geq 1}\mathcal{L}^{p}(\Omega).
	\end{align}
	Set $\psi_{\omega}^{t}:=D_{Y_{\omega}}\varphi^{t}_{\omega}:\mathcal{B}_{\alpha}\rightarrow \mathcal{B}_{\alpha}$. Assume for some $T_0>0$ that $\psi_{\omega}^{T_0} \coloneqq D_{Y_{\omega}}\varphi^{T_0}_{\omega} \colon \mathcal{B}_{\alpha} \rightarrow \mathcal{B}_{\alpha}$ is a compact operator. For every $\lambda>0$, set
	\begin{align*}
		F_{\lambda}(\omega) \coloneqq \lbrace \xi\in\mathcal{B}_{\alpha}: \ \ \limsup_{t\rightarrow\infty}\frac{1}{t}\log\vert \psi^{t}_{\omega}(\xi)\vert_{\alpha}\leq \lambda\rbrace.
	\end{align*}
	Then on a set of full measure $\tilde{\Omega}$, invariant under $(\theta_t)_{t\in\mathbb{R}}$, there are numbers 
	\begin{align*}
		\lambda_{1}>\lambda_{2}>...\in [-\infty,\infty),
	\end{align*}
	the \emph{Lyapunov exponents}, that are either finite or satisfy $\lim_{n\rightarrow\infty}\lambda_{n}=-\infty$, and finite dimenional subspaces $ H^{i}(\omega)\subset\mathcal{B}_\alpha$, $i \in \N$, such that the following properties hold: 
	\begin{itemize}
		\item[(i)] (Invariance.)\ \ $\psi^{t}_{\omega}(H^{i}_{\omega})= H^i_{\theta_t \omega}$ for every $t \geq 0$.
		\item[(ii)] (Splitting.)\ \ $F_{\lambda_{1}}(\omega)=\mathcal{B}_{\alpha}$ and $H_{\omega}^i \oplus F_{\lambda_{i+1}}(\omega) = F_{\lambda_i}(\omega)$ for every $i$. In particular,
		\begin{align*}
			\mathcal{B}_\alpha = \bigoplus_{1\leq j\leq i}H^j_{\omega}\oplus   F_{\lambda_{i+1}}(\omega)
		\end{align*}
		for every $i$.
		\item[(iii)] ('Fast' growing subspace.)\  For each $ h\in H^{j}_{\omega} $,
		\begin{align*}
			&\lim_{t\rightarrow\infty}\frac{1}{t}\log\vert \psi^{t}_{\omega}(h)\vert_{\alpha} = \lambda_{j}\\
			& 		\lim_{t\rightarrow\infty}\frac{1}{t}\log\vert (\psi^{t}_{\theta_{-t}\omega})^{-1}(h)\vert_{\alpha} =-\lambda_{j}.
		\end{align*}
	\end{itemize}
\end{theorem}

\begin{proof}
	The goal is to apply the Multiplicative Ergodic Theorem for Banach spaces stated in \cite[Theorem 1.21]{GVR23}. We first fix a time step $t_{0}>0$. By assumption, after sufficiently many iterations of $\psi^{t_0}_{\omega}$, the operator becomes compact. 
	From Theorem \ref{thm:integrabiliy_linearization}, it follows that
	\begin{align}\label{V1}
		\sup_{0\leq t_1\leq t_{0}}	\log^{+}\big(\Vert \psi^{t_1}_{\omega}\Vert_{\mathcal{L}(\mathcal{B}_{\alpha},\mathcal{B}_{\alpha})}\big)\in\mathcal{L}^{1}(\Omega).
	\end{align}
	In addition, from Lemma \eqref{NMMMAA}
	\begin{align}\label{V2}
		\sup_{0\leq t_1\leq t_{0}}\log^{+}\big(\Vert \psi^{t_0-t_1}_{\theta_{t_1}\omega}\Vert_{\mathcal{L}(\mathcal{B}_{\alpha},\mathcal{B}_{\alpha})}\big)\in \mathcal{L}^{1}(\Omega).
	\end{align}
	Now our claim for the discrete cocycle $(\psi^{nt_0}_{\omega})_{(n,\omega)\in \mathbb{N}\times \Omega}$ follows from \cite[Theorem 1.21]{GVR23}. 
	Our claim is also valid for the continuous time version, which can be obtained from the discrete version of the multiplicative ergodic theorem in \cite[Theorem 1.21]{GVR23} with \eqref{V1} and \eqref{V2}. We omit the details, but the reader can refer to \cite[Theorem 3.3, Lemma 3.4]{LL10}, where the authors obtained the continuous version of the multiplicative ergodic theorem from the discrete case \footnote{We can not refer to \cite{LL10} directly since their assumptions for the discrete version are restrictive. However, obtaining the continuous version from the discrete one has a standard argument, which is discussed in this paper.} by imposing same assumption as \eqref{V1} and \eqref{V2}.
\end{proof}

\begin{remark}
	We already mentioned in the introduction that the only two articles that discuss invariant manifolds for RPDEs are \cite{KN23} and \cite{YLZ23}. In both papers, the authors use the Lyapunov-Perron method that does not rely on the MET. Both articles impose the assumptions that $F(0) = DF(0) = 0$ and $G(0) = DG(0) = D^2G(0) = 0$. The condition $F(0) = G(0) = 0$ assures that $0$ is a stationary point in the sense of Definition \ref{def:stationary_point}. The assumption that also the derivatives of $F$ and $G$ have $0$ as a fixed point implies that the Lyapunov spectrum that can be deduced when applying the MET to the linearized cocycle around $0$ is just the spectrum of the operator $A$. In contrast, for our method, it is not necessary that the Lyapunov spectrum is explicitly given.
\end{remark}

\subsection{Invariant manifolds}
We are now ready to derive the existence of invariant manifolds (stable, unstable, and center) around the stationary point. We first start with the result for stable manifolds. 
\begin{theorem}\label{stable_manifold}
	Assume the same setting as in Theorem \ref{MET} and set $\lambda_{j_0} \coloneqq \sup\lbrace \lambda_{j}:\lambda_j<0 \rbrace$. We fix an arbitrary time step $t_{0}>0$. For $ 0 < \upsilon < -\lambda_{j} $, we can find a family of immersed submanifolds $S^{\upsilon}_{loc}(\omega)$ of $\mathcal{B}_{\alpha}$ and a set of full measure $\tilde{\Omega}$ of $\Omega$ such that
	\begin{itemize}
		\item[(i)] There are random variables $ \rho_{1}^{\upsilon}(\omega), \rho_{2}^{\upsilon}(\omega)$, which are positive and finite on $\tilde{\Omega}$, and
		\begin{align}\label{eqn:rho_temp}
			\liminf_{p \to \infty} \frac{1}{p} \log \rho_i^{\upsilon}(\theta_{pt_0} \omega) \geq 0, \quad i = 1,2
		\end{align}
		such that
		\begin{align}\label{invari}
			\begin{split}
				\big{\lbrace} \xi \in \mathcal{B}_{\alpha}\, :\, \sup_{n\geq 0}\exp(nt_0\upsilon)\vert\varphi^{nt_0}_{\omega}(\xi)-Y_{\theta_{nt_0}\omega}\vert_{\alpha} &<\rho_{1}^{\upsilon}(\omega)\big{\rbrace}\subseteq S^{\upsilon}_{loc}(\omega)\\&\subseteq	\big{\lbrace}  \xi \in \mathcal{B}_{\alpha}\, :\, \sup_{n\geq 0}\exp(nt_0\upsilon)\vert\varphi^{nt_0}_{\omega}\xi)-Y_{\theta_{nt_0}\omega}\vert_{\alpha}<\rho_{2}^{\upsilon}({\omega})\big{\rbrace}.
			\end{split}
		\end{align}
		\item[(ii)] For $ S^{\upsilon}_{loc}(\omega)$, 
		\begin{align*}
			T_{Y_{\omega}}S^{\upsilon}_{loc}(\omega) = F_{\lambda_{j_0}}(\omega).
		\end{align*}
		\item[(iii)] For $ n\geq N(\omega) $,
		\begin{align*}
			\varphi^{nt_0}_{\omega}(S^{\upsilon}_{loc}(\omega))\subseteq S^{\upsilon}_{loc}(\theta_{nt_0}\omega).
		\end{align*}
		\item[(iv)]For $ 0<\upsilon_{1}\leq\upsilon_{2}< - \lambda_{j_{0}} $,
		\begin{align*}
			S^{\upsilon_{2}}_{loc}(\omega)\subseteq S^{\upsilon_{1}}_{loc}(\omega).
		\end{align*}
		Also for $n\geq N(\omega) $,
		\begin{align*}
			\varphi^{nt_0}_{\omega}(S^{\upsilon_{1}}_{loc}(\omega))\subseteq S^{\upsilon_{2}}_{loc}(\theta_{nt_0}(\omega))
		\end{align*}
		and consequently for $ \xi\in S^{\upsilon}_{loc}(\omega) $,
		\begin{align}\label{eqn:contr_char}
			\limsup_{n\rightarrow\infty}\frac{1}{n}\log\vert\varphi^{nt_0}_{\omega}(\xi)-Y_{\theta_{nt_0}\omega}\vert_{\alpha}\leq  t_0\lambda_{j_{0}}.
		\end{align}
		\item[(v)] 
		\begin{align*}
			\limsup_{n\rightarrow\infty}\frac{1}{n}\log\bigg{[}\sup\bigg{\lbrace}\frac{\vert\varphi^{nt_0}_{\omega}(\xi)-\varphi^{nt_0}_{\omega}(\tilde{\xi}_{\omega})\vert_{\alpha} }{\vert \xi-\tilde{\xi}\vert_{\alpha}},\ \ \xi \neq\tilde{\xi},\  \xi,\tilde{\xi}\in S^{\upsilon}_{loc}(\omega) \bigg{\rbrace}\bigg{]}\leq t_0\lambda_{j_{0}}.
		\end{align*}
	\end{itemize}
\end{theorem}

\begin{proof}
	We aim to apply \cite{GVR23}[Theorem 2.10]. To do this, we need to check that \cite{GVR23}[Equation (2.5)] holds. Set $H_{\omega}(\xi) \coloneqq \varphi^{t_0}_{\omega}(Y_{\omega}+\xi)-\varphi^{t_0}_{\omega}(Y_{\omega})-\psi^{t_0}_{\omega}(\xi)$. From \eqref{bounde_derivative_3},
	\begin{align}\label{IOPL}
		\begin{split}
			\vert&H_{\omega}(\xi)-H_{\omega}(\tilde\xi)\vert_{\alpha}\leq\int_{0}^{1}\vert\big(D_{Y_{\omega}+\theta\xi+(1-\theta)\tilde{\xi}}\varphi^{t_0}_{\omega}-D_{Y_{\omega}}\varphi^{t_0}_{\omega}\big)[\xi-\tilde{\xi}]\vert_{\alpha} \, \mathrm{d}\theta\\
			&\quad\leq 2 E_{\epsilon}|\xi-\tilde{\xi}|_{\alpha}\max\lbrace \vert\xi\vert_{\alpha}+\vert\tilde{\xi}\vert_{\alpha},\vert\xi\vert_{\alpha}^{1-\kappa}+\vert\tilde{\xi}\vert_{\alpha}^{1-\kappa},\vert\xi\vert_{\alpha}^{\frac{r}{2}}+\vert\tilde{\xi}\vert_{\alpha}^{\frac{r}{2}}\rbrace\\
			&\qquad \times \exp\bigg(E_{\epsilon}R_1\big(\sup_{0\leq \rho,\theta \leq 1}\big\Vert\big(\varphi^{.}_{\omega}(Y_{\omega}+\rho(\theta\xi+(1-\theta)\tilde{\xi})),G\big(\varphi^{.}_{\omega}(Y_{\omega}+\rho(\theta\xi+(1-\theta)\tilde{\xi}))\big)\big)\big\Vert_{{\mathcal{D}_{\mathbf{X}(\omega),{\alpha}}^{\gamma}([0,t_0])}}\big)\bigg)\\
			&\myquad[3] \times \exp\big(E_{\epsilon}R_{2}(\Vert X(\omega)\Vert_{\gamma,[0,t_0]},\Vert \mathbb{X}(\omega)\Vert_{2\gamma,[0,t_0]})\big).
		\end{split}
	\end{align}
	
	
	
	Also from \eqref{integrable bound_2}
	\begin{align}\label{UOPL}
		\begin{split}
			&\sup_{0\leq \rho,\theta \leq 1}\big\Vert\big(\varphi^{.}_{\omega}(Y_{\omega}+\rho(\theta\xi+(1-\theta)\tilde{\xi})),G\big(\varphi^{.}_{\omega}(Y_{\omega}+\rho(\theta\xi+(1-\theta)\tilde{\xi}))\big)\big)\big\Vert_{\mathcal{D}_{\mathbf{X}(\omega),{\alpha}}^{\gamma}([0,t_0])}  \\
			&\quad\leq\tilde{M}N([0,t_0],\eta_1,\chi,\mathbf{X}(\omega))(1+\Vert X(\omega)\Vert_{\gamma,[0,t_0]})\bigg[\exp\big{(}N([0,t_0],\eta_1,\chi,\mathbf{X}(\omega))\tilde{M}_\epsilon\big{)} \big(|Y_{\omega}|_{\alpha}+|\xi|_{\alpha}+|\tilde{\xi}|_{\alpha}\big) \\
			&\qquad + \frac{\exp\big{(}N([0,t_0],\eta_1,\chi,\mathbf{X}(\omega))\tilde{M}_\epsilon+\tilde{M}_\epsilon\big{)}-1}{2M_\epsilon-1}P(\Vert X(\omega)\Vert_{\gamma,[0,t_0]},\Vert\mathbb{X}(\omega)\Vert_{\gamma,[0,t_0]})\bigg]\\&\myquad[3]= T_1(\omega,t_0)(|\xi|_{\alpha}+|\tilde{\xi}|_{\alpha})+T_{2}(\omega,Y_{\omega},t_0).
		\end{split}
	\end{align}
	with obvious definition of $T_1$ and $T_2$. Since $R_1$ is a polynomial, we can find increasing polynomials $R_{1}^{(1)},R_{1}^{(2)}$ and $R_{1}^{(3)}$ such that for every $A,B,C>0$: $R_1(AB+C)\leq R_{1}^{(1)}(A)+R_{1}^{(2)}(B)+R_{1}^{(3)}(C)$. Therefore, applying $R_1$ to \eqref{UOPL} leads to
	\begin{align*}
		&R_1\big(\sup_{0\leq \rho,\theta \leq 1}\big\Vert\big(\varphi^{.}_{\omega}(Y_{\omega}+\rho(\theta\xi+(1-\theta)\tilde{\xi})),G\big(\varphi^{.}_{\omega}(Y_{\omega}+\rho(\theta\xi+(1-\theta)\tilde{\xi}))\big)\big)\big\Vert_{{\mathcal{D}_{\mathbf{X}(\omega),{\alpha}}^{\gamma}([0,t_0])}}\big)\\&\quad\leq  R_{1}^{(1)}(T_1(\omega,t_0))+R_{1}^{(2)}(|\xi|_{\alpha}+|\tilde{\xi}|_{\alpha})+R_{1}^{(3)}(T_{2}(\omega,Y_{\omega},t_0)).
	\end{align*}
	If we plug in the later inequality in \eqref{IOPL}, we obtain
	\begin{align*}
		&\vert H_{\omega}(\xi)-H_{\omega}(\tilde\xi)\vert_{\alpha}\leq 2 E_{\epsilon}|\xi-\tilde{\xi}|_{\alpha}\max\lbrace \vert\xi\vert_{\alpha}+\vert\tilde{\xi}\vert_{\alpha},\vert\xi\vert_{\alpha}^{1-\kappa}+\vert\tilde{\xi}\vert_{\alpha}^{1-\kappa},\vert\xi\vert_{\alpha}^{\frac{r}{2}}+\vert\tilde{\xi}\vert_{\alpha}^{\frac{r}{2}}\rbrace\times\\&\quad \exp\big(E_{\epsilon}[R_{1}^{(1)}(T_1(\omega,t_0))+R_{1}^{(3)}(T_{2}(\omega,Y_{\omega},t_0))+R_{2}(\Vert X(\omega)\Vert_{\gamma,[0,t_0]},\Vert \mathbb{X}(\omega)\Vert_{2\gamma,[0,t_0]})]\big)\exp\big(R_{1}^{(2)}(\vert\xi\vert_{\alpha}+\vert\tilde{\xi}\vert_{\alpha})\big).
	\end{align*}
	From Theorem \ref{thm:integrability_RPDE}, for 
	\begin{align*}
		f(\omega) \coloneqq \exp\big(E_{\epsilon}[R_{1}^{(1)}(T_1(\omega,t_0))+R_{1}^{(3)}(T_{2}(\omega,Y_{\omega},t_0))+R_{2}(\Vert X(\omega)\Vert_{\gamma,[0,t_0]},\Vert \mathbb{X}(\omega)\Vert_{2\gamma,[0,t_0]})]\big),
	\end{align*}
	we have
	\begin{align*}
		\log^{+}f(\omega) = E_{\epsilon}[R_{1}^{(1)}(T_1(\omega,t_0))+R_{1}^{(3)}(T_{2}(\omega,Y_{\omega},t_0))+R_{2}(\Vert X(\omega)\Vert_{\gamma,[0,t_0]},\Vert \mathbb{X}(\omega)\Vert_{2\gamma,[0,t_0]})]\in\mathcal{L}^1(\Omega).
	\end{align*}
	Consequently, from Birkhoff's ergodic theorem on a set of full measure,
	\begin{align*}
		\lim_{n\rightarrow\infty} \frac{1}{n} \log^{+}f(\theta_{nt_0}\omega)=0.
	\end{align*}
	Thus we can apply \cite{GVR23}[Theorem 2.10] which yields the claim.
\end{proof}

\begin{remark}\label{RRFFCC}
	One may wonder whether a continuous time version of Theorem \ref{stable_manifold} may also be deduced. In fact, we can derive a slightly weaker result for continuous time by arguing as follows: Assume $0\leq\tilde{t}\leq t_{0}$ and $\xi\in S^{\upsilon}_{loc}(\omega)$. Then
	\begin{align*}
		&\varphi^{nt_0+\tilde{t}}_{\omega}(\xi)-\varphi^{nt_0+\tilde{t}}_{\omega}(Y_{\omega})=\varphi^{\tilde{t}}_{\theta_{nt_0}\omega}(\varphi^{nt_0}_{\omega}(\xi))-\varphi^{\tilde{t}}_{\theta_{nt_0}\omega}(\varphi^{nt_0}_{\omega}(Y_{\omega})).
	\end{align*} 
	From Corollary \ref{CC_DDE},
	\begin{align}\label{continous_part}
		\begin{split}
			&\vert \varphi^{nt_0+\tilde{t}}_{\omega}(\xi)-\varphi^{nt_0+\tilde{t}}_{\omega}(Y_{\omega})\vert_{\alpha}\leq  M\vert \varphi^{nt_0}_{\omega}(\xi)-Y_{\theta_{nt_0}\omega}\vert_{\alpha}\exp\big(M \tilde{S}(0,\tilde{t},\varphi^{nt_0}_{\omega}(\xi),Y_{\theta_{nt_0}\omega},\mathbf{X}(\theta_{nt_0}\omega))\big)
		\end{split}
	\end{align}
	where $\tilde{S}$ is defined in \eqref{SSssSS}. Note that for every $0\leq \rho\leq 1$  
	\begin{align*}
		&\vert\rho(\varphi^{nt_0}_{\omega}(\xi),Y_{\theta_{nt_0}\omega})\vert_{\alpha}=\vert\rho\varphi^{nt_0}_{\omega}(\xi)+(1-\rho)Y_{\theta_{nt_0}\omega}\vert_{\alpha}\leq \vert Y_{\theta_{nt_0\omega}}\vert_{\alpha}+\vert \varphi^{nt_0}_{\omega}(\xi)-Y_{\theta_{nt_0}\omega}\vert_{\alpha}.
	\end{align*}
	Therefore, for every $p \geq 1$,
	\begin{align*}
		\limsup_{n\rightarrow\infty}\frac{1}{n}\sup_{0\leq\rho\leq 1}\vert\rho(\varphi^{nt_0}_{\omega}(\xi),Y_{\theta_{nt_0}\omega})\vert_{\alpha}^{p}\leq 0.
	\end{align*}
	Consequently, from \eqref{integrable bound} and Birkhoff's ergodic theorem on a set of full measure
	\begin{align}\label{BRIKH}
		\limsup_{n\rightarrow\infty}\sup_{0\leq \tilde{t}\leq t_0}\frac{1}{n}\tilde{S}(0,\tilde{t},\varphi^{nt_0}_{\omega}(\xi),Y_{\theta_{nt_0}\omega},\mathbf{X}(\theta_{nt_0}\omega))\leq 0.
	\end{align}
	From this observation, we conclude if $t>0$, with $t=mt_{0}+\tilde{t}$ where $0\leq \tilde{t}<t_0$ and $\upsilon_1<\upsilon$,
	\begin{align}\label{CXZAQWE}
		\begin{split}
			&\sup_{n\geq 0}\exp(nt_0\upsilon_1)\vert\varphi^{nt_0}_{\omega}(\varphi^{t}_{\omega}(\xi))-Y_{\theta_{nt_0+t}\omega}\vert_{\alpha}\leq M[\sup_{k\geq 0}\exp(kt_0\upsilon)\vert\varphi^{kt_0}_{\omega}(\xi)-Y_{\theta_{kt_0}\omega}\vert_{\alpha}]\\
			&\qquad \times \sup_{n\geq 0}\big[\exp(-(m+n)t_{0}\upsilon+nt_{0}\upsilon_{1})\exp\big(M \tilde{S}(0,\tilde{t},\varphi^{(m+n)t_0}_{\omega}(\xi),Y_{\theta_{(m+n)t_0}\omega},\mathbf{X}(\theta_{(m+n)t_0}\omega))\big)\big].
		\end{split}
	\end{align}
	Since $\upsilon_{1}<\upsilon$, from \eqref{BRIKH} and \eqref{invari}, we conclude that if $t\geq t(\omega)$, then
	\begin{align*}
		\varphi^{t}_{\omega}(S^{\upsilon_{1}}_{loc}(\omega))\subseteq S^{\upsilon}_{loc}(\theta_{t}\omega).
	\end{align*}
\end{remark}

We are now ready to formulate the unstable manifold theorem.
\begin{theorem}\label{unstable}
	Assume the same setting as in Theorem \ref{MET} and that $\lambda_{1}>0$. Set $\lambda_{i_0} \coloneqq \inf\lbrace \lambda_{i}\, :\, \lambda_{i}>0\rbrace$. We fix an arbitrary time step $t_{0}>0$. For $ 0 < \upsilon < \lambda_{i_0} $, we can find a family of immersed submanifolds $U^{\upsilon}_{loc}(\omega) $ of $\mathcal{B}_{\alpha}$ and a set of full measure $\tilde{\Omega}$ of $\Omega$ with the following properties:
	\begin{itemize}
		\item[(i)] There are random variables $ \tilde{\rho}_{1}^{\upsilon}(\omega), \tilde{\rho}_{2}^{\upsilon}(\omega)$, which are positive and finite on $\tilde{\Omega}$, and
		\begin{align*}
			\liminf_{p \to \infty} \frac{1}{p} \log \tilde{\rho}_i^{\upsilon}(\theta_{-t_0 p} \omega) \geq 0, \quad i = 1,2
		\end{align*}
		such that
		\begin{align*}
			&\bigg{\lbrace} \xi_{\omega} \in \mathcal{B}_{\alpha}\, :\, \exists \lbrace \xi_{\theta_{-nt_{0}}\omega}\rbrace_{n\geq 1} \text{ s.t. } \varphi^{mt_{0}}_{\theta_{-nt_{0}}\omega}(\xi_{\theta_{-nt_{0}}\omega}) = \xi_{\theta_{(m-n)t_{0}}\omega} \text{ for all } 0 \leq m \leq n \text{ and }\\ 
			&\quad \sup_{n\geq 0}\exp(nt_0\upsilon)\vert \xi_{\theta_{-nt_{0}}\omega} - Y_{\theta_{-nt_{0}}\omega} \vert_{\alpha} <\tilde{\rho}_{1}^{\upsilon}(\omega)\bigg{\rbrace} \subseteq U^{\upsilon}_{loc}(\omega) \subseteq \bigg{\lbrace}  \xi_{\omega} \in \mathcal{B}_{\alpha}\, :\, \exists \lbrace \xi_{\theta_{-nt_{0}}\omega}\rbrace_{n\geq 1} \text{ s.t. } \\
			&\qquad \varphi^{mt_{0}}_{\theta_{-nt_{0}}\omega}(\xi_{\theta_{-nt_{0}}\omega}) = \xi_{\theta_{(m-n)t_{0}}\omega} \text{ for all } 0 \leq m \leq n \text{ and } \sup_{n\geq 0} \exp(nt_0\upsilon)\vert \xi_{\theta_{-nt_{0}}\omega} - Y_{\theta_{-nt_{0}}\omega} \vert_{\alpha} <\tilde{\rho}_{2}^{\upsilon}(\omega)\bigg{\rbrace}.
		\end{align*}
		\item[(ii)] For $ U^{\upsilon}_{loc}(\omega) $, 
		\begin{align*}
			T_{Y_{\omega}}U^{\upsilon}_{loc}(\omega) = {{\oplus}}_{i:\lambda_{i}>0}H^{i}_{\omega},
		\end{align*}
		\item[(iii)] For $ n\geq N(\omega) $,
		\begin{align*}
			U^{\upsilon}_{loc}(\omega)\subseteq \varphi^{nt_0}_{\theta_{-nt_0}\omega}(U^{\upsilon}_{loc}(\theta_{-nt_0}\omega).
		\end{align*}
		\item[(iv)]For $ 0<\upsilon_{1}\leq\upsilon_{2}<  \lambda_{i_{0}} $,
		\begin{align*}
			U^{\upsilon_{2}}_{loc}(\omega)\subseteq U^{\upsilon_{1}}_{loc}(\omega).
		\end{align*}
		Also for $n\geq N(\omega) $,
		\begin{align*}
			U^{\upsilon_{1}}_{loc}(\omega)\subseteq \varphi^{nt_0}_{\theta_{-nt_0}\omega}(U^{\upsilon_{2}}_{loc}(\theta_{-nt_0}\omega)
		\end{align*}
		and, consequently, for $ \xi\in U^{\upsilon}_{loc}(\omega) $,
		\begin{align}\label{eqn:contr_char_1}
			\limsup_{n\rightarrow\infty}\frac{1}{n}\log\vert\xi_{\theta_{-nt_0}\omega}-Y_{\theta_{-nt_0}\omega}\vert_{\alpha}\leq  -t_0\lambda_{i_{0}}.
		\end{align}
		\item[(v)] 
		\begin{align*}
			\limsup_{n\rightarrow\infty}\frac{1}{n}\log\bigg{[}\sup\bigg{\lbrace}\frac{\vert \xi_{\theta_{-nt_0}\omega\vert_{\alpha}}-\tilde\xi_{\theta_{-nt_0}\omega\vert_{\alpha}} }{\vert \xi-\tilde{\xi}\vert_{\alpha}},\ \ \xi_\omega \neq\tilde{\xi}_{\omega},\  \xi_{\omega},\tilde{\xi}_{\omega}\in U^{\upsilon}_{loc}(\omega) \bigg{\rbrace}\bigg{]}\leq -t_0\lambda_{i_{0}}.
		\end{align*}
	\end{itemize}
\end{theorem}
\begin{proof}
	Follows by applying \cite[Theorem 2.17]{GVR23}.
\end{proof}
\begin{remark}
	As above, we aim to deduce a continuous time result. Assume that $\xi_\omega\in U^{\upsilon}_{loc}(\omega)$ and let $\lbrace \xi_{\theta_{-nt_{0}}\omega}\rbrace_{n\geq 1}$ be the corresponding sequence in item (i) of Theorem \eqref{unstable}. We can extend this sequence to all negatives times be defining
	\begin{align*}
		\text{for} \ \  nt_0\leq t<(n+1)t_0: \ \ \ \xi_{\theta_{-t}\omega}:=\varphi^{(n+1)t_0-t}_{\theta_{-(n+1)t_0}\omega}(\xi_{\theta_{-(n+1)t_0}}). 
	\end{align*}
	Note that, from \eqref{bounde_derivative_2}
	\begin{align*}
		\vert \xi_{\theta_{-t}\omega} - Y_{\theta_{-t}\omega} \vert_{\alpha} &= \vert\varphi^{(n+1)t_0-t}_{\theta_{-(n+1)t_0}\omega}(\xi_{\theta_{-(n+1)t_0}\omega})-\varphi^{(n+1)t_0-t}_{\theta_{-(n+1)t_0}\omega}(Y_{\theta_{-(n+1)t_0}\omega})\vert_{\alpha}\\
		&\leq M_{\epsilon} \vert \xi_{\theta_{-(n+1)t_0}\omega}-Y_{\theta_{-(n+1)t_0}\omega}\vert_{\alpha} \\
		&\quad \times \exp\big(M_\epsilon \sup_{0\leq\tilde{t}\leq t_0}\tilde{S}(0,\tilde{t},\xi_{\theta_{-(n+1)t_0}\omega},Y_{\theta_{-(n+1)t_0}\omega},\mathbf{X}(\theta_{\theta_{-(n+1)t_0}}\omega))\big).
	\end{align*}
	Similar to \eqref{BRIKH},
	\begin{align*}
		\limsup_{n\rightarrow\infty}\sup_{0\leq \tilde{t}\leq t_0}\frac{1}{n}\tilde{S}(0,\tilde{t},\xi_{\theta_{-(n+1)t_0}\omega},Y_{\theta_{-(n+1)t_0}\omega},\mathbf{X}(\theta_{\theta_{-(n+1)t_0}}\omega))\leq 0.
	\end{align*}
	Therefore, by a similar calculation as in \eqref{CXZAQWE}, if $\upsilon_{1}<\upsilon$ and $t\geq t(\omega)$,
	\begin{align*}
		U^{\upsilon_1}_{loc}(\omega)\subseteq \varphi^{t}_{\theta_{-t}\omega}(U^{\upsilon}_{loc}(\theta_{-t}\omega).
	\end{align*}
\end{remark}

Finally, we state our result about the existence of center manifolds. 
\begin{theorem}\label{thm:center_manifolds}
	Assume the same setting as in Theorem \ref{MET} and suppose, for some $j_0\geq1$, $\lambda_{j_0}=0$. If $j_0=1$, then we set $\lambda_{-1}=\infty$. We fix an arbitrary time step $t_{0}>0$ and $0<\nu<\min\lbrace \lambda_{j_0-1},-\lambda_{j_0+1}\rbrace$. Then, there exists a continuous cocycle 
	\begin{align*}
		\tilde{\varphi} \colon \mathbb{Z}t_{0} \times \Omega\times \mathcal{B}_{\alpha}\rightarrow\mathcal{B}_\alpha,
	\end{align*}
	$\Z t_0 = \{ zt_0\, :\, z \in \Z\}$, i.e. $\tilde\varphi^{(m+n)t_0}_{\omega}(\xi)=\tilde{\varphi}^{mt_0}_{\theta_{nt_0}\omega}\circ\tilde{\varphi}^{nt_0}_{\omega}(\xi)$, and a positive random variable $\rho^c:\Omega\rightarrow(0,\infty)$ such that
	\begin{align*}
		\liminf_{n\rightarrow\pm\infty}\frac{1}{n}\log\rho^{c}(\theta_{nt_0}\omega)\geq 0,
	\end{align*}
	such that if $\vert \xi-Y_\omega\vert_\alpha\leq \rho^c(\omega)$, then $\tilde\varphi^{t_0}_{\omega}(Y_\omega+\xi)=\varphi^{t_0}_{\omega}(Y_\omega+\xi)$. Also there exists a map 
	\begin{align*}
		h^{c}_{\omega} \colon H^{j_0}_{\omega}\rightarrow \mathcal{M}^{c,\nu}_{\omega}\subset \mathcal{B}_{\alpha},
	\end{align*}
	such that 
	\begin{itemize}
		\item[(i)] $h^{c}_{\omega}$ is a homeomorphism, Lipschitz continuous and differentiable at zero.
		\item[(ii)] $\mathcal{M}^{c,\nu}_{\omega}$ is a topological Banach manifold modeled on $H^{j_0}$.
		\item[(iii)] $\mathcal{M}^{c,\nu}_{\omega}$ is $\tilde\varphi$-invariant, i.e. for every $n \in \mathbb{N}_0$, $\tilde\varphi^{nt_0}_{\omega}(\mathcal{M}^{c,\nu}_{\omega})\subset\mathcal{M}^{c,\nu}_{\theta_{nt_0}\omega}$.
	\end{itemize}
	Moreover, for every $\xi\in \mathcal{M}^{c,\nu}_{\theta_{nt_0}\omega}$,
	\begin{align*}
		\sup_{j\in\mathbb{Z}}\exp(-\nu\vert j\vert)\vert\tilde{\varphi}^{nt_0}_{\omega}(\xi)-Y_{\theta_{nt_0}\omega}\vert_\alpha<\infty .
	\end{align*}
\end{theorem}
\begin{proof}
	Cf. \cite[Theorem 2.14]{GVR23a}.
\end{proof}

\subsection{Stability}\label{sec:stability}
It is natural to expect an exponential decay of the solution in a neighborhood of stationary points when the first Lyapunov exponent is negative. In this part, we give an affirmative answer to this question. First, we prove a result about the exponential decay around the stationary point when the first Lyapunov exponent is negative.

\begin{lemma}\label{stabilty_n}
	Assume $\lambda_1<0$ in Theorem \ref{stable_manifold}. Let $t_0>$ be the arbitrary time step which we fixed in Theorem \ref{stable_manifold}. Then for every $0<\upsilon<-\lambda_1$, then there exists a random variable $R^{\upsilon}:\Omega\rightarrow (0,\infty)$ such that $\liminf_{t\rightarrow\infty}\frac{1}{p}\log R^{\upsilon}(\theta_{t}\omega)\geq 0$ and
	\begin{align}\label{neighborhood_1}
		\lbrace\xi\in\mathcal{B}_\alpha:\ \vert\xi-Y_{\omega}\vert\leq R^{\upsilon}(\omega) \rbrace\subset S^{\upsilon}_{loc}(\omega).
	\end{align}
\end{lemma}
\begin{proof}
	The claim follows from a slight modification of \cite[Theorem 2.10]{GVR23}. We simplify our notations during the proof and adapt them to the current paper. Recall that $\lambda_{j_0} \coloneqq \sup\lbrace \lambda_{j} \, :\, \lambda_j<0\}$. From \cite[Equation 2.14]{GVR23}, 
	\begin{align*}
		S^{\upsilon}_{loc}(\omega) = \left\lbrace Y_\omega+\Pi^{0}\big(\Gamma(v)\big)\, :\,  \vert v\vert\leq R^{\upsilon}(\omega)\right\rbrace.
	\end{align*}
	Here $\Pi^{0}:\prod_{j\geqslant 0}\mathcal{B}_{\alpha}\rightarrow \mathcal{B}_\alpha$ is the projection in the first component and
	\begin{align*}
		\Gamma \colon F_{\lambda_{j_0}}(\omega)\cap \lbrace v\in F_{\lambda_{j_0}}(\omega): \vert v\vert_{\alpha}\leq R^{\upsilon}(\omega)\rbrace\rightarrow \prod_{j\geqslant 0}\mathcal{B}_{\alpha}
	\end{align*}
	is defined in \cite[Lemma 2.7]{GVR23}. In fact, $\Gamma$ is a fixed point of the map $I$, cf. \cite[Lemma 2.6]{GVR23}, i.e. $I\big(v,\Gamma(v)\big)=\Gamma(v)$. Since $\lambda_{j_0}=\lambda_1$, we conclude $F_{\lambda_{j_0}}(\omega)=\mathcal{B}_\alpha$ and in the last formula in \cite[page 122]{GVR23}, we have $\Pi^{0}\big(\Gamma(v)\big)=v$. Therefore,
	\begin{align*}
		S^{\upsilon}_{loc}(\omega) = \left\lbrace Y_\omega+v \, :\,  \vert v\vert\leq R^{\upsilon}(\omega)\right\rbrace.
	\end{align*}
	This yields the proof.
\end{proof}
\begin{corollary}\label{stability++}
	For every $0<\upsilon_1<\upsilon<-\lambda_1$, if $|\xi - Y_{\omega}| \leq R^{\upsilon}(\omega)$,
	\begin{align*}
		\sup_{t\geq 0}\exp(\upsilon_1t)\vert\varphi^{t}_{\omega}(\xi)-Y_{\theta_{t}\omega}\vert_{\alpha}<0.
	\end{align*}
	In particular, the solution $\varphi$ is exponentially stable around the stationary point $Y$.
\end{corollary}
\begin{proof}
	Follows from Remark \ref{RRFFCC} and Lemma \ref{stabilty_n}. 
\end{proof}

We are now ready to formulate our main result of this section.
\begin{theorem}\label{thm:exp_stability}
	Assume in Equation \eqref{SPDE_EQU} that $F(0)=G(0)=0$ and that for some $\lambda<0$,
	\begin{align}\label{deterministic_semigroup}
		\limsup_{t\rightarrow\infty}\frac{1}{t}\log\Vert S_{t}\Vert_{\mathcal{L}(\mathcal{B}_{\alpha},\mathcal{B}_{\alpha})}\leq\lambda.
	\end{align}
	In addition, accept the Assumption \ref{ASSUU}. Then zero is a stationary point for our equation, and there exists an $\epsilon_0>0$  such that for $\theta\in \lbrace 0,\gamma,2\gamma\rbrace$, if 
	\begin{align}\label{epsilon_0}
		\max\lbrace	\Vert D_{0}F\Vert_{\mathcal{L}(\mathcal{B}_{\alpha-\sigma},\mathcal{B}_\alpha)}, \Vert D_{0}G\Vert_{\mathcal{L}(\mathcal{B}_{\alpha-\eta},\mathcal{B}_{\alpha-\theta-\eta}^{n})}\rbrace\leq \epsilon_0,
	\end{align}
	the first Lyapounov exponent is negative. In particular, the system is exponentially is stable around a random neighborhood of zero.
\end{theorem}

\begin{proof}
	Recall that for $\xi=0$,
	\begin{align}\label{BVXZZ}
		\varphi^{t}_{\omega}(0)=\int_{0}^{t}S_{t-\tau}F(\varphi^{\tau}_{\omega}(0)) \, \mathrm{d}\tau+\int_{0}^{t}S_{t-\tau}G(\varphi^{\tau}_{\omega}(0))\circ\mathrm{d}\mathbf{X}_\tau(\omega)
	\end{align}
	where
	\begin{align*}
		&\int_{0}^{t}S_{t-\tau}G(\varphi^{\tau}_{\omega}(0))\circ\mathrm{d}\mathbf{X}_\tau(\omega)\\
		=\ &\lim_{\substack{|\pi|\rightarrow 0,\\ \pi=\lbrace 0 = \tau_0 < \tau_{1} < \ldots < \tau_{m}=t \rbrace}}\sum_{0\leq j<m}\big{[}S_{t-\tau_j}\varphi^{\tau_j}_{\omega}(0)\circ (\delta X)_{\tau_j,\tau_{j+1}}(\omega)+S_{t-\tau_j}D_{\varphi^{\tau_j}_{\omega}(0)}G[G(\varphi^{\tau_j}_{\omega}(0))]\circ\mathbb{X}_{\tau_j,\tau_{j+1}}\big{]}.
	\end{align*}
	Therefore $\varphi^{t}_{\omega}(0)=0$ solves \eqref{BVXZZ}, thus $0$ is indeed a fixed point by uniqueness of the equation. From \cite[3.3.2 Theorem]{Arn98},
	\begin{align}\label{FR_1}
		\lambda_1 = \inf_{t \geq 0} \frac{1}{t} \int_{\Omega}\log \Vert \psi^{t}_{\omega}\Vert_{\mathcal{L}(\mathcal{B}_{\alpha},\mathcal{B}_{\alpha})}\ \mathbb{P}(\mathrm{d}\omega) \leq  \frac{1}{t_0}\int_{\Omega}\log \Vert \psi^{t_0}_{\omega}\Vert_{\mathcal{L}(\mathcal{B}_{\alpha},\mathcal{B}_{\alpha})}\ \mathbb{P}(\mathrm{d}\omega)
	\end{align}
	for every $t_0 > 0$ where
	\begin{align*}
		\psi^{t}_{\omega}(\zeta) = S_{t}\zeta+\int_{0}^{t}S_{t-\tau}D_{0}F[\psi^{\tau}_{\omega}(\zeta)] \, \mathrm{d}\tau+\int_{0}^{t}S_{t-\tau}D_{0}G[\psi^{\tau}_{\omega}(\zeta)]\circ\mathrm{d}\mathbf{X}_\tau(\omega).
	\end{align*}
	From \eqref{deterministic_semigroup}, we can choose $t_{0}>0$ large enough such that $\frac{1}{t_0}\log\Vert S_{t_0}\Vert_{\mathcal{L}(\mathcal{B}_{\alpha},\mathcal{B}_{\alpha})}<0$. From 
	the fact that $\Vert \psi^{t_0}_{\omega}\Vert_{\mathcal{L}(\mathcal{B}_{\alpha},\mathcal{B}_{\alpha})}$ depends continuously on $\Vert D_{0}F\Vert_{\mathcal{L}(\mathcal{B}_{\alpha-\sigma},\mathcal{B}_\alpha)}$ and $\Vert D_{0}G\Vert_{\mathcal{L}(\mathcal{B}_{\alpha-\eta},\mathcal{B}_{\alpha-\theta-\eta}^{n})}$, we can conclude that 
	\begin{align}
		\frac{1}{t} \log\Vert \psi^{t_0}_{\omega}\Vert_{\mathcal{L}(\mathcal{B}_\alpha),\mathcal{B}_\alpha}\rightarrow\frac{1}{t_0}\log\Vert S_{t_0}\Vert_{\mathcal{L}(\mathcal{B}_\alpha,\mathcal{B}_\alpha))}\ \ a.s 
	\end{align}
	as $\max\lbrace\Vert D_{0}F\Vert_{\mathcal{L}(\mathcal{B}_{\alpha-\sigma},\mathcal{B}_\alpha)},\Vert D_{0}G\Vert_{\mathcal{L}(\mathcal{B}_{\alpha-\eta},\mathcal{B}_{\alpha-\theta-\eta}^{n})}\rbrace\rightarrow 0$. Let us to fix $N\geq 1$. Since $D_{0}F$ and $D_{0}G$ are linear, from the uniform bound that is provided in \eqref{integrable bound_2} and the dominated convergence theorem,
	\begin{align}\label{first_Lyapounov_exponent}
		\lambda_1\leq\int_{\Omega}\frac{1}{t_0}\log \Vert \psi^{t_0}_{\omega}\Vert_{\mathcal{L}(\mathcal{B}_{\alpha},\mathcal{B}_{\alpha})}\vee (-N)\ \mathbb{P}(\mathrm{d}\omega)\longrightarrow \frac{1}{t_0}\log\Vert S_{t_0}\Vert_{\mathcal{L}(\mathcal{B}_\alpha,\mathcal{B}_\alpha))}\vee (-N),
	\end{align}
	as $\max\lbrace\Vert D_{0}F\Vert_{\mathcal{L}(\mathcal{B}_{\alpha-\sigma},\mathcal{B}_\alpha)},\Vert D_{0}G\Vert_{\mathcal{L}(\mathcal{B}_{\alpha-\eta},\mathcal{B}_{\alpha-\theta-\eta}^{n})}\rbrace\rightarrow 0$. Therefore, for small  $\epsilon_0>0$, when \eqref{epsilon_0} holds and $N,t_0$ are selected to be large, from \eqref{deterministic_semigroup} and \eqref{first_Lyapounov_exponent} 
	\begin{align*}
		\lambda_1\leq\int_{\Omega}\frac{1}{t_0}\log \Vert \psi^{t_0}_{\omega}\Vert_{\mathcal{L}(\mathcal{B}_{\alpha},\mathcal{B}_{\alpha})} \mathbb{P}(\mathrm{d}\omega)<0 .   
	\end{align*}
	Our claim now follows from Corollary \ref{stability++}.
\end{proof}

\section{Applications}
In this part, we will illustrate our results by giving several examples where they can be applied. Before doing this, we will recall some background about Sobolev spaces that we will need in the sequel. Remember that one important condition in the Multiplicative Ergodic Theorem was the compactness of the operator. In Remark \ref{COMPACTEM}, we mentioned that if the embedding 
\begin{align*}
	\operatorname{id} \colon \mathcal{B}_{\alpha+\epsilon}\rightarrow \mathcal{B}_{\alpha}
\end{align*}
is compact for some $0< \epsilon<\min\lbrace 1-\sigma ,\gamma-\eta\rbrace$, the compactness of $\psi^{T}_{\omega}$ follows. In the following subsection, we will therefore recall some compact embedding theorems as well. 

\subsection{Sobolev spaces and compact embeddings}
Note that all results stated in this section can be found in \cite{HT07}. \smallskip

For $s\in\mathbb{R}$, set $w_{s}(\xi) \coloneqq (1+\Vert \xi\Vert^2)^{\frac{s}{2}}$ and
\begin{align*}
	H^{s}(\mathbb{R}^d) \coloneqq \lbrace f\in\mathscr{S}^{\prime}(\mathbb{R}^d):w_s\mathcal{F}f\in \mathcal{L}^{2}(\mathbb{R}^d)\rbrace,
\end{align*}
where $\mathscr{S}^{\prime}(\mathbb{R}^d)$ denotes the space of tempered distributions and 
\begin{align*}
	(\mathcal{F}f)(\xi) \coloneqq (2\pi)^{\frac{-d}{2}}\int_{\mathbb{R}^d}\exp(ix \circ \xi)f(x) \, \mathrm{d}x
\end{align*} 
is the Fourier transform of $f$. Note that $H^{s}(\mathbb{R}^d)$ is a Hilbert space with inner product 
\begin{align*}
	\big{\langle}f,g\big{\rangle}_{H^{s}(\mathbb{R}^d)} \coloneqq \int_{\mathbb{R}^d}w_{s}(\xi)\mathcal{F}f(\xi)\circ\overline{w_{s}(\xi)\mathcal{F}g(\xi)} \, \mathrm{d}\xi.
\end{align*}
By $W^{k}_{2}$, we denote the classical Sobolev spaces. For $s=k+\sigma$ with $k\in\mathbb{N}\cup\lbrace 0\rbrace$ and $0<\sigma<1$, we say that $f\in W^{s}_{2}(\mathbb{R}^d)$ if $f\in W^{k}_{2}(\mathbb{R}^d)$ and
\begin{align}\label{nono}
	\Vert f\Vert_{W^{s}_{2}(\mathbb{R}^d)} \coloneqq \Vert f\Vert_{W^{k}_{2}(\mathbb{R}^d)} + \left( \sum_{\vert\alpha\vert=k} \iint_{\mathbb{R}^{2d}}\frac{\Vert D^{\alpha}f(x)-D^{\alpha}f(y)\Vert^2}{\Vert x-y\Vert^{d+2\sigma}} \, \mathrm{d}x \, \mathrm{d}y\right)^{\frac{1}{2}}.
\end{align}
It can be proven that $W^{s}_{2}(\mathbb{R}^d)=H^{s}(\mathbb{R}^d)$ and that the two norms $\Vert \cdot \Vert_{W^{s}_{2}(\mathbb{R}^d)}$ and $\Vert \cdot\Vert_{H^{s}(\mathbb{R}^d)}$ are equivalent. For an arbitrary bounded $C^{\infty}$-domain $\mathcal{D}\subset\mathbb{R}^d$, we define 
\begin{align}\label{A4}
	W^{s}_{2}(\mathcal{D}) \coloneqq \lbrace f\in\mathcal{L}^{2}(\mathcal{D}):\exists g\in W^{s}_{2}(\mathbb{R}^d)\ \  \text{such that}\ \  g|_{\mathcal{D}}=f\rbrace
\end{align}
and equip that space with the norm
\begin{align}\label{A5}
	\Vert f\Vert_{W^{s}_{2}(\mathcal{D})} \coloneqq \inf\lbrace\Vert g\Vert_{W^{s}_{2}(\mathbb{R}^d)}\, :\, g\in W^{s}_{2}(\mathbb{R}^d)\  \text{and}\  g|_{\mathcal{D}}=f\rbrace.
\end{align}
Similar to \eqref{nono}, for $s=k+\sigma$ with $k\in\mathbb{N}\cup\lbrace 0\rbrace$ and $0<\sigma<1$, $\Vert \cdot \Vert_{W^{s}_{2}(\mathcal{D})}$ is equivalent to the norm 
\begin{align*}
	\Vert f\Vert_{*,W^{s}_{2}(\mathcal{D})} \coloneqq \Vert f\Vert_{W^{k}_{2}(\mathcal{D})} + \left( \sum_{\vert\alpha\vert=k}\iint_{\mathcal{D}\times \mathcal{D}}\frac{\Vert D^{\alpha}f(x)-D^{\alpha}f(y)\Vert^2}{\Vert x-y\Vert^{d+2\sigma}} \, \mathrm{d} x \, \mathrm{d}y \right)^{\frac{1}{2}}.
\end{align*}

Following Lemma is a classical result.
\begin{lemma}
	Let $\mathcal{D}$ be a bounded $C^{\infty}$-domain in $\mathbb{R}^d$. For $0<\theta<1$, let $X_{\theta} = [X_0, X_1]_\theta$ be the complex interpolation space for given Banach spaces $X_{0}$ and $X_1$ with $X_{0}\subset X_1$. Let $0<s<t$. Then we have the following:
	\begin{itemize}
		\item [(i)] $[W^{t}_{2}(\mathcal{D}),W^{s}_{2}(\mathcal{D})]_{\theta}=W^{(1-\theta)s+\theta t}_{2}(\mathcal{D})$.
		\item [(ii)] The embedding 
		\begin{align*}
			\operatorname{id} \colon W^{s}_{2}(\mathcal{D})\rightarrow W^{t}_{2}(\mathcal{D})
		\end{align*}
		is compact.
	\end{itemize}
\end{lemma}
A natural extension of $H^{s}(\mathbb{R}^d)$ are the Bessel potential spaces. For $1<p<\infty$, set
\begin{align*}
	H^{s}_p(\mathbb{R}^d):=\lbrace f\in\mathscr{S}^{\prime}(\mathbb{R}^d) \, :\, \mathcal{F}^{-1}(w_s\mathcal{F}f)\in \mathcal{L}^{p}(\mathbb{R}^d)\rbrace.
\end{align*}
For a domain $\mathcal{D}$ of $\mathbb{R}^d$, the space $H^{s}_p(\mathcal{D})$ is defined similar to $W^{s}_{2}(\mathcal{D})$ as in \eqref{A4} and \eqref{A5}.

\begin{lemma}\label{CCPP}
	For $s\in\mathbb{N}_0$, $H^{s}_p(\mathcal{D})$ is equal to the classical Sobolev space $W^{s}_{p}(\mathcal{D})$ with an equivalent norm. Moreover, for $0<s<t$,
	\begin{align*}
		[H^{s}_p(\mathcal{D}),H^{t}_p(\mathcal{D})]_{\theta}=H^{(1-\theta)s+\theta t}_p(\mathcal{D}).
	\end{align*}
	If $s>0$, the embedding
	\begin{align*}
		\operatorname{id} \colon H^{s}_p(\mathcal{D})\rightarrow \mathcal{L}^{p}(\mathcal{D})
	\end{align*}
	is compact.
\end{lemma}
For $1<p<\infty$ and $s=k+\sigma$ with $k\in\mathbb{N}\cup\lbrace 0\rbrace$ and $0<\sigma<1$, we say $f\in B^{s}_{p,p}(\mathbb{R}^d)$ if $f\in W^{k}_{p}(\mathbb{R}^d)$ and
\begin{align}\label{nono_p}
	\Vert f\Vert_{B^{s}_{p,p}(\mathbb{R}^d)} \coloneqq  \Vert f\Vert_{W^{k}_{p}(\mathbb{R}^d)} + \left( \sum_{\vert\alpha\vert=k}\iint_{\mathbb{R}^{2d}}\frac{\Vert D^{\alpha}f(x)-D^{\alpha}f(y)\Vert^p}{\Vert x-y\Vert^{d+p\sigma}} \, \mathrm{d}x \, \mathrm{d}y\right)^{\frac{1}{p}}
\end{align}
is finite. For a smooth domain $\mathcal{D}$ of $\mathbb{R}^d$, the spaces $B^{s}_{p,p}(\mathcal{D})$ are defined similar to $W^{s}_{2}(\mathcal{D})$ as in in \eqref{A4} and \eqref{A5}. The respective norms are equivalent to
\begin{align*}
	\Vert f\Vert_{*,B^{s}_{p,p}(\mathcal{D})} \coloneqq \Vert f\Vert_{W^{k}_{p}(\mathcal{D})} + \left(\sum_{\vert\alpha\vert=k}\iint_{\mathcal{D}\times \mathcal{D}}\frac{\Vert D^{\alpha}f(x)-D^{\alpha}f(y)\Vert^p}{\Vert x-y\Vert^{d+p\sigma}} \, \mathrm{d}x \, \mathrm{d}y\right)^{\frac{1}{p}}.
\end{align*}
In fact, $B^{s}_{p,p}(\mathbb{R}^d)$ is a generalization of $W^{s}_{2}$ for an arbitrary $p$. It is known that $H^{s}_{2}(\mathbb{R}^d) = W^{s}_{2}(\mathbb{R}^d)$ but if $p\ne 2$, $B^{s}_{p,p}(\mathbb{R}^d)$ and $H^{s}_{p}(\mathbb{R}^d)$ are different spaces. Similar to \ref{CCPP}, for bounded $C^\infty$-domains $\mathcal{D}$,
\begin{align*}
	[B^{t}_{p,p}(\mathcal{D}),B^{s}_{p,p}(\mathcal{D})]_{\theta}=B^{(1-\theta)s+\theta t}_{p,p}(\mathcal{D})
\end{align*}
for $s<t$. Moreover, the embedding
\begin{align*}
	\operatorname{id} \colon B^{s}_{p,p}(\mathcal{D})\rightarrow \mathcal{L}^{p}(\mathcal{D})
\end{align*}
is compact.

\subsection{Examples of rough stochastic partial differential equations}

We are now ready to present several examples. They are mostly taken from \cite{GH19},\cite{GHT21},\cite{YLZ23} and \cite{KN23}, therefore we do not go too much into the detail and refer the reader to these articles for more details.
\begin{example}\label{EX_1}
	The first example is a reaction diffusion equation on the one-dimensional torus. For $l\in\mathbb{R}\setminus\lbrace 0\rbrace$, let $\mathbb{T}=\mathbb{R}/l\mathbb{Z}$. The space $H_{0}^{2\alpha}(\mathbb{T})$ denotes the closure of $C_{c}^{\infty}(\mathbb{T})$ in $H^{2\alpha}_2(\mathbb{T})$. We consider a rough reaction-diffusion equation with periodic boundary conditions, i.e.
	\begin{align*}
		\begin{cases}
			&\mathrm{d}u_t=(\Delta u_t+F(u_t)) \, \mathrm{d}t+g(x)(-\Delta)^{\eta}u_t \, \mathrm{d}\mathbf{B}^{H}_{t},\\
			& u_0\in\mathcal{L}^{2}(\mathbb{T})\ \ \text{and}\ \ \ \int_{\mathbb{T}}u_{0}(x)\mathrm{d}x=0.
		\end{cases}
	\end{align*}
\end{example}
The equation is driven by the rough path lift of a fractional Brownian motion with Hurst parameter $\frac{1}{3}<H\leq \frac{1}{2}$ and $\eta$ is chosen such that $0\leq\eta<2H-\frac{1}{2}$. For $0\leq\sigma<1$, we assume that $F \colon \mathcal{B}_{\alpha}\rightarrow \mathcal{B}_{\alpha-\delta}$ is a locally Lipschitz-continuous map with linear growth. In addition, $g$ is a smooth and bounded function on $\mathbb{T}$. It is known that the operator $\Delta u$ generates an analytic $C_0$-semigroup on $\mathcal{B}=\mathcal{B}_0$ such that the spectrum is given by 
\begin{align*}
	\left\lbrace -(\frac{2\pi k}{l})^2:\  k\in\mathbb{Z}\right\rbrace .
\end{align*}
Assuming $F(u)=0$ implies that $u_0=0$ is a stationary point. From Lemma \ref{CCPP} and Remark \ref{COMPACTEM_12}, the linearized equation is compact, therefore we deduce the existence of invariant manifolds around $0$. 

\begin{example}\label{EX_2}
	The second example is again a reaction diffusion equation one the one-dimensional torus, but with Dirichlet boundary conditions.
	Assume 
	\begin{align*}
		\begin{cases}
			&\mathrm{d}u_t=(\Delta u_t+F(u_t))\mathrm{d}t+g(x)(-\Delta)^{\eta}u_t \, \mathrm{d}\mathbf{B}^{H}_{t},\\
			& u_{t}(0)=u_{t}(l)=0,
			\\& u_{0}(x)\in \mathcal{L}^{p}(\mathbb{T}), \quad 1<p<\infty.
		\end{cases}
	\end{align*}
	The assumptions on the parameters and on $F$ and $g$ are the same as in the previous example. In this case, we set $\mathcal{B} \coloneqq \mathcal{L}^{p}(\mathbb{T})$ and for $0\leq\alpha\leq1$, $\mathcal{B}_{\alpha} \coloneqq B^{2\alpha}_{p,p,0}(\mathbb{T})$ where by $B^{2\alpha}_{p,p,0}(\mathbb{T})$ we mean the closure of $C_{c}^{\infty}(\mathbb{T})$ in $B^{2\alpha}_{p,p}(\mathbb{T})$. Similar results as in  Example \ref{EX_1} can be stated here.
\end{example}

Before presenting the third example, we note that for $\mathbb{T}^d=\mathbb{R}^d/l\mathbb{Z}^d$, the space $H^{s}_{p}(\mathbb{T}^d)$ is an algebra if $ps>d$, i.e.
\begin{align*}
	f,g\in H^{s}_{p}(\mathbb{T}^d): \ \ \ \Vert fg\Vert_{H^{s}_{p}(\mathbb{T}^d)}\leq C\Vert f\Vert_{H^{s}_{p}(\mathbb{T}^d)} \Vert g\Vert_{H^{s}_{p}(\mathbb{T}^d)} .
\end{align*}

\begin{example}\label{EX_3} 
	Here, we consider a generic equation of the form
	\begin{align*}
		\begin{cases}
			&\mathrm{d}u_t=(\Delta u_t+F(u_t)) \, \mathrm{d}t+ G(u_t)\mathrm{d}\mathbf{B}^{H}_{t},\\
			& u_{0}\in H^{k}_{p,0}(\mathbb{T}^d).
		\end{cases}
	\end{align*}
	For $F$, we can choose $F(u_t):=u_{t}P(u_t)$ where $P$ is an analytic bounded function. $G$ can either be a bounded linear function like in the previous examples or $G(u_t)=g(u_t)$ for $g$ being an analytic bounded function. In these cases, if we select $\frac{1}{3}<\gamma\leq H$ such that $p(k-2\gamma)>d$, since  $H^{k-2\gamma}_{p}(\mathbb{T}^d)$ is an algebra, our results can be applied.
	
\end{example}
\begin{remark}\label{remark:random_stationary_point}
	For the case that $\mathbf{B}$ is a Brownian motion and $F$ is bounded, we can expect that a non-trivial stationary point exists. This stationary point is the solution to
	\begin{align*}
		u_{t,\omega}=\int_{-\infty}^{t}S_{t-\tau}F(u_{\tau,\omega}) \, \mathrm{d}\tau + \int_{-\infty}^{t}S_{t-\tau}G(u_{\tau,\omega})\, \d\mathbf{B}_{\tau}.
	\end{align*}
\end{remark}



\bibliographystyle{alpha}
\bibliography{refs}

\end{document}